\documentclass[12pt,a4paper,oneside,leqno]{amsart}

\usepackage[all]{xy}
\usepackage{hyperref,amssymb,amsmath,bbm,url,geometry,cancel,cleveref,enumitem}
\usepackage[mathscr]{euscript}
\usepackage{mathrsfs}
\usepackage[utf8]{inputenc}

\newtheorem{thmi}{Theorem}
\newtheorem{propi}[thmi]{Proposition}
\newtheorem{cjti}[thmi]{Conjecture}

\newtheorem{cori}[thmi]{Corollary}

\newtheorem{thm}{Theorem}[subsection]
\newtheorem{prop}[thm]{Proposition}
\newtheorem{cor}[thm]{Corollary}
\newtheorem{lm}[thm]{Lemma}

\theoremstyle{definition}
\newtheorem{df}[thm]{Definition}
\newtheorem{ex}[thm]{Example}
\newtheorem{num}[thm]{}

\newtheorem{cjt}[thm]{Conjecture}

\theoremstyle{remark}
\newtheorem{rem}[thm]{Remark}

\numberwithin{equation}{thm}

\chardef\tempcat=\the\catcode`\@
\catcode`\@=11

\def\cydot{{\mathsurround=0pt$\cdot$}}

\def\ubar#1{\oalign{#1\crcr\hidewidth
    \vbox to.2ex{\hbox{\char22}\vss}\hidewidth}}

\def\cprime{\/{\mathsurround=0pt$'$}}
\def\Cprime{{\mathsurround=0pt$'$}}
\def\cdprime{\/{\mathsurround=0pt$''$}}
\def\Cdprime{{\mathsurround=0pt$\ubar{\hbox{$''$}}$}}

\def\dbar{dj}           \def\Dbar{Dj}           

\def\dz{dz}
\def\Dz{Dz}
\def\dzh{dzh\cydot }
\def\Dzh{Dzh\cydot }

\def\@gobble#1{}
\def\@testgrave{\`}
\def\@stressit{\futurelet\chartest\@stresschar }

\def\@stresschar#1{\ifx #1y\def\result{\futurelet\chartest\@yligature}\else \ifx #1Y\def\result{\futurelet\chartest\@Yligature}\else \ifx\chartest\@testgrave \def\result{\accent"26 }\else \def\result{\accent"26 #1}\fi \fi \fi
  \result }

\def\@yligature{\ifx a\chartest \def\result{\accent"26 \char"1F \@gobble}\else \ifx u\chartest \def\result{\accent"26 \char"18 \@gobble}\else \def\result{\accent"26 y}\fi \fi
  \result }

\def\@Yligature{\ifx a\chartest \def\result{\accent"26 \char"17 \@gobble}\else \ifx A\chartest \def\result{\accent"26 \char"17 \@gobble}\else \ifx u\chartest \def\result{\accent"26 \char"10 \@gobble}\else \ifx U\chartest \def\result{\accent"26 \char"10 \@gobble}\else \def\result{\accent"26 Y}\fi \fi \fi \fi
  \result }

\def\!{\ifmmode \mskip-\thinmuskip \fi}

\def\cyracc{\def\cydot{{\kern0pt}}\def\cprime{\char"7E }\def\Cprime{\char"5E }\def\cdprime{\char"7F }\def\Cdprime{\char"5F }\def\dbar{dj}\def\Dbar{Dj}\def\dz{\char"1E }\def\Dz{\char"16 }\def\dzh{\char"0A }\def\Dzh{\char"02 }\def\'##1{\if c##1\char"0F \else \if C##1\char"07 \else \accent"26 ##1\fi \fi }\def\`##1{\if e##1\char"0B \else \if E##1\char"03 \else \errmessage{accent \string\` not defined in cyrillic}##1\fi \fi }\def\=##1{\if e##1\char"0D \else \if E##1\char"05 \else \if \i##1\char"0C \else \if I##1\char"04 \else \errmessage{accent \string\= not defined in cyrillic}##1\fi \fi \fi \fi }\def\u##1{\if \i##1\accent"24 i\else \accent"24 ##1\fi }\def\"##1{\if \i##1\accent"20 \char"3D \else \if I##1\accent"20 \char"04 \else \accent"20 ##1\fi \fi }\def\!{\ifmmode \def\result{\mskip-\thinmuskip}\else \def\result{\@stressit}\fi \result}}

\def\keep@cyracc{\let\cyr=\relax \let\i=\relax
        \let\ubar=\relax \let\cydot=\relax
        \let\cprime=\relax \let\Cprime=\relax
        \let\cdprime=\relax \let\Cdprime=\relax
        \let\dbar=\relax \let\Dbar=\relax
        \let\dz=\relax \let\Dz=\relax
        \let\dzh=\relax \let\Dzh=\relax
        \let\'=\relax \let\`=\relax \let\==\relax
        \let\u=\relax \let\"=\relax \let\!=\relax }

\catcode`\@=\tempcat
 \DeclareFontFamily{U}{russian}{}
\DeclareFontShape{U}{russian}{m}{n}
        { <5><6> wncyr5
        <7><8><9> wncyr7
        <10><10.95><12><14.4><17.28><20.74><24.88> wncyr10 }{}
\DeclareSymbolFont{Russian}{U}{russian}{m}{n}
\DeclareSymbolFontAlphabet{\mathcyr}{Russian}
\makeatletter
\let\@math@cyr\mathcyr
\renewcommand{\mathcyr}[1]{\@math@cyr{\cyracc #1}}
\makeatother

\newcommand{\ab}{\mathrm Ab}
\newcommand{\Sm}{\mathrm{Sm}}
\newcommand{\Sme}{\overline{\mathrm{Sm}}}
\newcommand{\Smc}{\mathrm{Sm}^{\mathrm{cor}}}
\newcommand{\Smec}{\overline{\mathrm{Sm}}^{\mathrm{cor}}}
\newcommand{\Smgen}{\mathrm{Sm}^{\mathrm{cor},(0)}}

\newcommand{\efld}{\tilde{\mathcal E}}

\DeclareMathOperator{\PSh}{PSh}

\DeclareMathOperator{\Shtr}{Sh^{tr}}

\DeclareMathOperator{\MCycl}{\mathscr MCycl}

\newcommand{\iC}{\mathscr C}
\newcommand{\iD}{\mathscr D}
\newcommand{\iI}{\mathscr I}
\newcommand{\iS}{\mathscr S} \newcommand{\iSp}{\mathscr Sp} \newcommand{\iM}{\mathscr M}
\newcommand{\catC}{\mathrm C}
\newcommand{\catD}{\mathrm D}
\newcommand{\catA}{\mathrm A}
\newcommand{\Fin}{\mathrm{Fin}_*}

\DeclareMathOperator{\Ho}{Ho}

\DeclareMathOperator{\iDer}{\mathscr D}
\DeclareMathOperator{\iCat}{\mathscr Cat_\infty}
\DeclareMathOperator{\iCatP}{\mathscr Cat^{\mathrm{pres}}_\infty}
\DeclareMathOperator{\iCatS}{\mathscr Cat^{\mathrm{st}}_\infty}
\DeclareMathOperator{\iCatM}{\mathscr Cat^{\otimes}_\infty}
\DeclareMathOperator{\iCatMS}{\mathscr Cat^{\mathrm{st\otimes}}_\infty}

\DeclareMathOperator{\iDM}{\mathscr DM}
\DeclareMathOperator{\iDMe}{\mathscr DM^{eff}}
\DeclareMathOperator{\iDMgm}{\mathscr DM_{gm}}
\DeclareMathOperator{\iDMgme}{\mathscr DM^{eff}_{gm}}
\DeclareMathOperator{\iDMgen}{\mathscr DM^{(0)}}

\DeclareMathOperator{\iSH}{\mathscr{SH}}

\DeclareMathOperator{\iPSh}{\mathscr PSh}
\DeclareMathOperator{\iSh}{\mathscr Sh}

\newcommand{\sS}{\Delta^\op\mathrm{Set}}
\DeclareMathOperator{\DM}{DM}
\DeclareMathOperator{\DMe}{DM^{eff}}
\DeclareMathOperator{\DMgm}{DM_{gm}}
\DeclareMathOperator{\DMgme}{DM^{eff}_{gm}}

\DeclareMathOperator{\Chow}{CHM}

\newcommand{\KM}[1]{\operatorname K^M_{#1}}
\newcommand{\KSM}[1]{\operatorname K^{SM}_{#1}}

\DeclareMathOperator{\uH}{\underline H}
\newcommand{\hrt}{\heartsuit}
\DeclareMathOperator{\sus}{C_*^{sus}}

\DeclareMathOperator{\Ztr}{\ZZ^{tr}}
\DeclareMathOperator{\HI}{HI}
\DeclareMathOperator{\HM}{HM}

\newcommand{\uA}{\underline{A}} 

\newcommand{\E} {\mathbf E}
\newcommand{\F} {\mathbf F}

\DeclareMathOperator{\B} {B}
\DeclareMathOperator{\BGL} {BGL}
\DeclareMathOperator{\BSL} {BSL}
\DeclareMathOperator{\GL} {GL}
\DeclareMathOperator{\SL} {SL}
\DeclareMathOperator{\SO} {SO}
\DeclareMathOperator{\SU} {SU}
\DeclareMathOperator{\Gr} {Gr}

\newcommand{\HB}{\operatorname H_{\mathcyr B}}

\newcommand{\Hmot}{\operatorname H_{\mathrm M}}

\newcommand{\NN} {\mathbb N}
\newcommand{\ZZ} {\mathbb Z}
\newcommand{\QQ} {\mathbb Q}
\newcommand{\FF} {\mathbb F}
\newcommand{\RR} {\mathbb R}
\newcommand{\CC} {\mathbb C}
\renewcommand{\AA} {\mathbb A}
\newcommand{\GG} {\mathbb G_m }
\newcommand{\GGx}[1] {\mathbb G_{m,#1} }
\newcommand{\PP} {\mathbb P}

 \newcommand{\cM}{\mathcal M}

\newcommand{\cO}{\mathcal O}

\newcommand{\cX}{\mathcal X} \newcommand{\cY}{\mathcal Y}
\newcommand{\cZ}{\mathcal Z} \newcommand{\cU}{\mathcal U}

\newcommand{\cI}{\mathcal I} 

\newcommand{\sI}{\mathscr I} \newcommand{\sJ}{\mathscr J}

\newcommand{\T} {\mathscr T}

\newcommand{\un}{\mathbbm 1}

\renewcommand{\lim}{\operatorname{lim}}
\DeclareMathOperator{\colim}{colim}
\newcommand{\plim}[1]{\underset{#1}{\text{\rm ``lim''}}}

\newcommand{\pprod}[1]{\underset{#1}{\text{\rm ``$\prod$''}}}
\DeclareMathOperator{\hocolim} {hocolim}

\DeclareMathOperator{\Ker}{Ker}
\DeclareMathOperator{\coKer}{coKer}
\DeclareMathOperator{\Img}{Im}

\DeclareMathOperator{\codim}{codim}

\DeclareMathOperator{\dtr}{trdeg} 

\newcommand{\tw}[1]{\{#1\}}  

\DeclareMathOperator{\pro}{pro-\!}
\DeclareMathOperator{\ind}{ind-\!}

\DeclareMathOperator{\M}{M} \DeclareMathOperator{\bM}{\hat M} \DeclareMathOperator{\h}{h} \DeclareMathOperator{\bh}{\hat h} 

 \newcommand{\tra}[1]{{}^{\mathrm t}#1} 

\DeclareMathOperator{\Comp}{C}

\DeclareMathOperator{\Der}{\mathrm D}

\newcommand{\nis}{\mathrm{Nis}}

\newcommand{\cdh}{\mathrm{cdh}}

\newcommand{\op}{\mathrm{op}}
\newcommand{\Id}{\mathrm{Id}}
\newcommand{\spec}[1] {\operatorname{\mathrm{Spec}}(#1)}

\DeclareMathOperator{\Spec}{Spec}

\DeclareMathOperator{\Nrv}{\mathscr N}

\DeclareMathOperator{\Hom}{Hom}
\DeclareMathOperator{\Map}{Map}
\DeclareMathOperator{\Fun}{\mathscr Fun}
\DeclareMathOperator{\uHom}{\underline{Hom}}

\DeclareMathOperator{\Ext}{Ext}

\newcommand{\Flag}{\mathscr Flag}

\DeclareMathOperator{\Pic}{Pic}
\DeclareMathOperator{\CH}{CH}

\DeclareMathOperator{\car}{char}
\DeclareMathOperator{\card}{card}

\DeclareMathOperator{\Prim}{P} \newcommand{\sslash}{\mathbin{/\mkern-6mu/}}

\begin{document}

\title{Generic motives and motivic cohomology of fields}

\author{F.~D\'eglise}
\date{May 2025}

\begin{abstract}
This paper investigates the structure of generic motives
 and their implications for the motivic cohomology of fields.
 Originating in Voevodsky's theory of motives
 and related to Beilinson's vision of a motivic $t$-structure,
 generic motives serve as pro-objects encoding essential information about cycles and cohomology.
 We present new computations of generic motives, focusing on curves and surfaces.
 These computations suggest a conjectural framework for morphisms of generic motives
 and highlight the central role of transcendental motives.

We then focus on the motivic cohomology of fields,
 building on Borel's rank computation of K-theory and its relation to higher regulators.
 We provide a direct argument for determining the weights in the $\lambda$-structure
 of the K-theory of number fields, bypassing the need for regulator maps.
 We show that motivic cohomology groups are often of infinite rank,
 typically matching the cardinality of the base field.
 For instance, we prove that motivic cohomology groups
 of $\RR$ and $\CC$ are uncountable in many bi-degrees. Despite this,
 we propose a conjecture that complements the Beilinson-Soulé vanishing conjecture, 
 suggesting that the growth of motivic cohomology
 is more controlled than these results may initially indicate.
\end{abstract}

\maketitle

{\flushright This paper is dedicated to the memory of Jacob Murre.}

\setcounter{tocdepth}{2}
\tableofcontents

\section{Introduction}

\subsection{Historical perspective}

\subsubsection*{Beilinson's conjectures}
The concept of motivic cohomology was first formulated, as it seems, by Beilinson in a 1982 letter to Soulé\footnote{This letter was once available on the K-theory preprint server,
 which unfortunately went offline. The first published reference by Beilinson mentioning motivic cohomology (under the name ``absolute'' cohomology) is \cite{BeiHR1}.}
 as a universal cohomology theory for algebraic varieties,
 analogous to singular cohomology in topology.
 Inspired initially by Bloch’s construction of higher regulators for $K_2$,
 Beilinson envisioned motivic cohomology as a universal receptacle for characteristic classes,
 serving as a refinement of Quillen’s higher algebraic K-theory. 

This vision was preceded by a series of foundational computations linking $K$-theory and special values of zeta functions.
 Early results by Moore, Tate, and Garland on the group $K_2$ of global fields (see Bass's Bourbaki talk \cite{BassB})
 led to the formulation of the famous Birch-Tate conjecture --- much progress has been made on this conjecture,
 but the general case remains open.
 This motivated Lichtenbaum to propose a series of synthetic conjectures, first in the context of the zeta function of a totally real number field,
 linking special values, $K$-theory, étale cohomology, and higher regulators (see \cite{LichtVal1}). A decisive breakthrough came with Borel’s results in 1972,
 which extended Garland’s earlier work and determined the ranks of higher algebraic $K$-groups of rings of integers in number fields; see \cite{BorelCRAS, Borel}.
 Borel also established the existence of higher regulators in this setting and verified Lichtenbaum's conjecture in the case of (arbitrary) number fields; 
 see \cite{Borel2, Borel2bis}.

These converging developments culminated in a profound conceptual synthesis through the work of Beilinson,
 who brought together several deep lines of thought into a unified vision.
 Bloch’s 1978 Irvine lectures (published later in \cite{BlochHR}),
 in which he revisited Borel's work on regulators of number fields through the lens of dilogarithms
 and proposed a visionary extension to the case of elliptic curves,  had a decisive influence.
 Around the same time, Deligne's study of special values of $L$-functions \cite{DelVal} brought Grothendieck's theory of motives into the picture.
 Building on these inputs, Beilinson formulated a striking and coherent picture:  combining the conjectural theory of mixed motives with Deligne's theory of mixed Hodge structures,
 he expressed the higher regulator as a byproduct of the realization functor from mixed motives to mixed Hodge structures.

Returning to the universal property of motivic cohomology,  Beilinson’s regulator also admits a concrete, unconditional formulation:
 rational motivic cohomology can be defined using the $\lambda$-structure on algebraic $K$-theory as developed by Soulé in \cite{Soule}.
 The Beilinson regulator then appears as the universal map to Deligne cohomology, induced by Chern classes.
 This allowed Beilinson to state conjectures predicting the special values of $L$-functions of motives in terms of this regulator,
 encompassing the earlier predictions of Lichtenbaum, and including, up to rational factors,
 both the Birch–Tate and Birch-Swinnerton-Dyer conjectures; see \cite{BeiHR1, BeiPair}.

Beilinson’s work proved extremely fruitful and inspired many mathematicians to approach his conjectures from a variety of angles.
 Bloch was the first to propose in \cite{BlochHR} an integral version of motivic cohomology through the theory of higher Chow groups,
 providing a concrete model compatible with Beilinson’s conjectures. Suslin slightly later introduced in \cite{SVsing}
 an alternative approach via Suslin homology, inspired by topological methods.

\subsubsection*{Voevodsky's motivic theory}
It was Voevodsky who brought a decisive and revolutionary shift.
 In a groundbreaking PhD thesis, he introduced a new homotopical viewpoint based on the contractibility of the affine line and the use of the h-topology. 
 Building on this approach, he developed a new framework for defining Beilinson's conjectural motivic complexes that not only met the formal expectations envisioned by Beilinson,
 but also integrated naturally into a broader theory of motivic homotopy, later developed in collaboration with Morel.
 This framework led to several landmark results, including the proofs of the Milnor and Bloch–Kato conjectures \cite{VoeBK}, as well as the Quillen–Lichtenbaum conjecture.
 Voevodsky’s approach laid the foundations for a complete theory of motivic complexes, modeled on Grothendieck's $\ell$-adic formalism, as anticipated by Beilinson.
 It was brought to maturity through the development of a full Grothendieck six-functor formalism, both in its Nisnevich and étale variants
 (see \cite{VoePHD, VSF, Ayoub, VoeSimpl, ROmod, Ivorra, CD3, AyoubEt, CD4, CD5, ParkMot, Spit}).\footnote{We review the developments of the Nisnevich variant in \Cref{rem:DM-bases}.}

A concise account of the construction of the category of motivic complexes $\DM(k)$ over a (perfect) field
 will be provided in \Cref{sec:mot-cpx}. For the purposes of this work, we have chosen to  present this construction
 within the framework of $\infty$-categories, which, for many applications in  homological algebra,
  are better suited than triangulated categories. To our knowledge, this is the first presentation of motivic complexes
 directly formulated in terms of $\infty$-categories. To assist readers unfamiliar with these concepts,
 we have included a detailed overview on $\infty$-categories in the appendix \ref{sec:infty-cat}.\footnote{The $\infty$-category
 of motivic complexes is denoted by $\iDM(k)$ while its homotopy category,
 equivalent to the non-effective version of Voevodsky's construction, is denoted by $\DM(k)$.}

\subsection{Content of the paper}

\subsubsection*{Motivic cohomology of fields}
Despite these advances, one major problem remains open: the motivic $t$-structure whose heart
 would provide the abelian category of mixed motives. With rational coefficients and other a base field,
 Beilinson's conjectural description leaves no choice, as it demands that the $\ell$-realization
 is conservative and $t$-exact, with respect to the canonical $t$-structure on the derived
 category of $\ell$-adic \'etale sheaves (Galois modules in that case).
 Therefore a rational geometric motivic complex
 should be non-negative (resp. positive) for the motivic $t$-structure
 if and only if its $\ell$-adic realization is.
 However, the fact that this actually  defines a $t$-structure\footnote{Orthogonality and existence of truncations are problematic} depends on at least two deep conjectures: the conservativity conjecture (see \cite{AyoubCons})
 and the Beilinson-Soulé vanishing conjecture. In fact, Levine showed that the latter is equivalent to the fact that the above definition induces a $t$-structure on the sub-category of rational geometric Artin-Tate
 motives; see \cite{LevineAT}.

The Beilinson-Soulé conjecture asserts that for any smooth scheme $X$ over a field $k$,
 rational motivic cohomology 
$$
\HB^{n,i}(X):=\Gr_\gamma^i K_{2i-n}(X)_\QQ
$$
 vanishes if $n \leq 0$ and $(n,i) \neq (0,0)$ (see \cite[Introduction]{BMS}).
 Using the coniveau spectral sequence, this conjecture can be reduced to the case of arbitrary
 function fields.\footnote{More precisely, the residue fields of all (schematic) points of $X$.}
 In particular,
 the key issue becomes our understanding of the motivic cohomology of fields.

For now, rational motivic cohomology of fields is known for finite fields
 and global fields, in other words, fields $K$ of Kronecker dimension less or equal
to one, where one  defines this degree as:
$$
\delta(K)=
\begin{cases}
1+\dtr(K/\QQ) & \text{if } \car(F)=0, \\
\dtr(K/\FF_p) & \text{if } \car(F)=p.
\end{cases}
$$
 Indeed, in positive characteristic $p$, if $\delta(K) \leq 1$,
 results of Quillen (finite fields) and Harder  imply that:
$$
\HB^{n,i}(K)=\begin{cases}
\QQ & n=i=0, \\
K^\times \otimes \QQ & n=i=1, \\
0 & \text{otherwise.}
\end{cases}
$$

\subsubsection*{The case of number fields.}
This is more subtle than the positive characteristic case.
 As mentioned in the preceding historical review,
 Borel computed the ranks of
 the rational K-groups of $K$, which, in degrees $>1$,
 coincide with those of its ring of integers $\cO_K$.
 These groups are concentrated in odd degrees $2n-1$.
 Borel also defined a regulator map:\footnote{See \eqref{eq:Borel-regulator} below for an explicit definition.}
$$
\rho_{Bo}:K_{2n-1}(\cO_K)_\QQ \rightarrow V_n
$$
with values in a  finite-dimensional real vector space $V_n$,
 and showed it is injective.
 On the other hand, Beilinson defined a regulator map,
 as mentioned above,
 with values in the Deligne-Beilinson cohomology of $K$
$$
\rho_{Be}:\HB^{n,i}(K) \rightarrow H_{\mathcal D}^n(\spec K,\RR(i))
$$
and claimed that this map agrees with Borel's regulator,
 once one takes into account the Chern character isomorphism:
$$
K_i(X)_\QQ \simeq \oplus_{n\geq 0} \HB^{2n-i,n}(X).
$$
As a consequence of the injectivity of Borel's regulator,
 one deduces that:
\begin{equation}\tag{\thesubsection.vanish}\label{eq:intro-vanish}
\forall (n,i) \in \ZZ^2, \HB^{n,i}(K)=0 \text{ unless } n=i=0 \text{ or } n=1.
\end{equation}
 This, in particular, implies the Beilinson-Soulé conjecture
 for number fields.

The first written comparison between Borel's and Beilinson's
 regulators was given by Rapoport in \cite{Rapop}.
Burgos later provided a more complete and accurate treatment in \cite{Burgos},
  and corrected a discrepancy
 by identifying a missing factor of two in the original comparison.

\bigskip

In particular, the whole procedure to deduce the computation
 of motivic cohomology of number fields is intricate and technically involved.
 The first task we undertake in this work is to give
 a direct proof of \eqref{eq:intro-vanish},
 without appealing to any regulator map.
 We approach this result through a direct analysis of Borel's isomorphism,
 which uses the relation of the K-theory of $\cO_K$
 with the stable real cohomology of the arithmetic group $\SL_r(\cO_K)$.
 Our strategy is to examine the compatibility of each step in Borel’s construction
 with natural $\lambda$-ring structures,
 starting from the canonical one on $K_n(\cO_K)\otimes_\QQ\RR$.
 This ultimately reduces the proof of \eqref{eq:intro-vanish}
 to determining the $\lambda$-weights of the
 indecomposable part of the real homology of the ``compact twin''
 of the real Lie group $\SL_r(K \otimes_\QQ \RR)$.

 We refer the reader to \Cref{cor:weights-cpt-dual} for the final
 determination of the $\lambda$-weights, and to \Cref{sec:Borel} for
 a detailed explanation of the Borel isomorphism and its compatibility
 with $\lambda$-structures. In this context, the referee pointed out the
 approach of Borel and Yang on the rank conjecture (see \cite{BY}),
 which makes use of partly related techniques.
 See also \cite{dJ} for a related result concerning weights in the homology of $\mathrm{GL}_n$.

\subsubsection*{Generic motives} Generic motives were introduced
 in \cite{BeiGen} and in \cite{Deg1} with distinct but complementary objectives.
 They are defined as pro-objects in the category of motivic complexes $\DM(k)$ over a perfect field $k$,
 associated with a function field $E$ over $k$ (a finitely generated extension field) and
 constructed from all possible smooth schematic models $X$ over $k$:
$$
\bM(E):=\plim{X} M(X),
$$
where $M(X)$ denotes the homological motive associated with $X/k$ --- see also \Cref{df:genmot}.
 Generic motives should be considered with twists $\bM(E)\tw n$ for arbitrary integers $n \in \ZZ$,
 defined via $\GG$-twists:  $\un\tw 1:=\un(1)[1] \subset M(\GG)$.
 This definition underscores the need for the $\infty$-categorical framework: in general, the pro-objects
 of a triangulated category do not themselves form a triangulated category, while pro-objects of an $\infty$-category
 naturally inherit the structure of an $\infty$-category (see \Cref{sec:pro-infty}).

In Beilinson's work, generic motives were employed alongside
 the filtration by dimension on the category of motivic complexes $\DM(k)$,
 to characterize uniquely the motivic t-structure, independently of 
 any realization functor.
 In contrast, in the author's PhD work, generic motives were used
 to encode the structure of Rost cycle modules.
 Specifically, the morphisms of generic motives completely capture
 the structural functoriality of cycle modules, allowing
 any motivic complex to give rise to such objects.
 Applying this result to motivic cohomology, one deduces that
 motivic cohomology of fields can be organized in a family of cycle
 modules: fixing an integer $n \in \ZZ$,  one gets
$\ZZ$-graded abelian groups associated with finitely generated extensions $E/k$:
$$
\Hmot^{n+r,r}(E), r \in \ZZ
$$
equipped with 4 basic operations: (D1) corestriction maps,
 (D2) restriction of norm maps, (D3) action of units of $E$ and
 more generally of Milnor K-theory of $E$, (D4)
 residues.\footnote{The reader is referred to \Cref{sec:funct-gen-mot}
 for more details.}

Furthermore, generic motives play a non-conjectural role in relation with the homotopy $t$-structure,
 analogous to the role envisioned by Beilinson
 for the conjectural motivic $t$-structure.
 The existence of the homotopy $t$-structure on the effective motivic complexes $\DMe(k)$
 is a corollary of Voevodsky's
 fundamental theorems on homotopy invariant sheaves with transfers,
 which constitute its heart. This $t$-structure can be extended
 to the (non-effective or stable) category $\DM(k)$  (We recall this aspect  in \ref{num:recall_DM}).
 In this framework, generic motives function like projective objects with respect
 to the homotopy $t$-structure on $\DM(k)$, and  evaluation at generic motives
 gives a conservative family. This aspect will be explained in a section
 recalling  the theory of generic motives; see in particular \Cref{prop:gen-mot_projective}
 for the last assertion.

Therefore, the category of generic motives plays a central role in the study of motivic complexes.
 In particular, the morphisms between generic motives can be interpreted as specialization morphisms,
 admitting a (partial) description in terms of the functoriality of cycle modules.
 In this work, we refine the computation of these morphisms by explicitly determining the generic motives associated with low-dimensional varieties,
 specifically curves and surfaces.
 This illustrates Beilinson's considerations on generic motives,
 as it appears that the fundamental building blocks of motives naturally emerge as constituents of the corresponding generic motives
 of their function fields.
 In the case of curves, Jacobians provide the basic bricks, while for surfaces,
 the so-called transcendental motives of Kahn, Murre, and Pedrini play a similar foundational role.
\begin{propi}\label{prop:gen_mot}
Let $X$ be a smooth proper scheme over $k$.
\begin{enumerate}[leftmargin=*]
\item If $X=C$ is a curve with Albanese scheme $A$ (dual of the Jacobian), there exists a homotopy exact sequence\footnote{here, we use the
 language of stable $\infty$-categories: the term \emph{homotopy exact} means both
 that the first term is the homotopy fiber of the right map
 and the third term is the homotopy cofiber of the left map.
 Consequently, the sequence viewed in the associated homotopy category
 becomes a distinguished triangle. See further \Cref{num:stable}.}
 of pro-motives over $k$:
$$
\pprod{x \in C^{(1)}} \bM(\kappa_x)\tw 1 \xrightarrow{\prod_x \partial_x} \bM(K)
 \xrightarrow{(\pi_0^K,\pi_1^K)} \un \oplus \uA
$$
which is split if and only if $C$ is rational. We refer to \Cref{prop:gmot_tr1}
 for more details and notation.
\item If $X=S$ is a surface over a separably closed field,
 with Albanese scheme $A$, Picard number $\rho$,
 and homological transcendental motive $\M_2^{tr}(X)$,
 one gets two homotopy exact sequences of pro-motives:
$$
\hspace{-4em}\xymatrix@C=20pt@R=20pt{
&&& \pprod{s \in S^{(2)}} \bM(\kappa_s)\tw 2\ar[d] \\
\pprod{x \in S^{(1)}} \bM(\kappa_x)\tw 1\ar^-{\prod_x \partial_x}[rr] && \bM(K)\ar[r] & \bM(S^{\leq 1})\ar[d] \\
&&& {}\makebox[4em][c]{$\un \oplus \uA \oplus \rho.\un(1)[2] \oplus \M_2^{tr}(S) \oplus \uA(1)[2] \oplus \un(2)[4]$}
}
$$

The pro-motive $\M(S^{\leq 1})$ has a precise interpretation in terms of the coniveau filtration: see \Cref{prop:gen-mot-surf} for details.
\end{enumerate}
\end{propi}
\noindent Generic motives are therefore extensions (non-split in general) of the fundamental bricks of (mixed) motives.
 Both results shed new light on morphisms of generic motives. They confirm the fact
 that morphisms of generic motives should be topologically generated by Rost's four basic maps
 (see \Cref{cjt:gen-gmot}).

\bigskip

As a consequence of the preceding result --- specifically the first point in the
 rational case --- and of the determination of motivic cohomology of number fields based on
 Borel's computation, as previously discussed, we deduce that motivic cohomology of fields
 is infinite in a wide range of degrees, depending on their Kronecker dimension.
\begin{thmi}\label{thm:main}
Let $K$ be a field of characteristic $0$ of Kronecker dimension $\delta(K)>1$.
Then for any $n \in [2,\delta(K)]$ and any $i \geq n$, the abelian group $\Hmot^{n,i}(K)$ has infinite rank
 equal to  $\card(K)$.
\end{thmi}
The proof of this theorem is straightforward, once the generic motive of a pure transcendental extension is understood
 (\Cref{prop:gmot-rational1})
 and \Cref{prop:gen_mot}(1) is established,
 along with the fact that motivic cohomology groups $\HB^{1,n}(E)$ have positive rank whenever $E$ is totally imaginary and $n>1$.
 The former, obtained as early as 2002 in the author's PhD thesis \cite{Deg1}, generalizes a classical result of Bass and Tate \cite{BT},
 while the latter has already been discussed in the historical introduction.
 However, the preceding consequence for the motivic cohomology of fields, which we believe is significant,
 seems to have gone unnoticed in the literature so far.
 We refer the reader to \Cref{sec:cohm-fields} for further discussion.

Strikingly, motivic cohomology of the real or complex numbers is uncountable in 
 bi-degree $(n,i)$ whenever $i \geq n \geq 2$.
 In terms of Bloch's higher Chow groups, this leads to the following result, which,
 though not strictly an extension of Mumford’s classical geometric theorem on Chow groups \cite{Mum},
 shares a similar flavor from an arithmetic perspective, being a consequence of Borel’s profound computation of higher K-groups:
\begin{cori}
Let $X$ be a complex smooth algebraic variety, or a real one
 with a rational point.
 Then for all integers $(p,q) \in \NN$ such that $p \leq q \leq 2p-2$,
 the higher Chow group:
$$
\CH^p(X,q)
$$
has infinite, uncountable rank.
\end{cori}

It is worth noting that, just as in Mumford's result for surfaces of general type,
 the occurrence of infinite rank in our setting is restricted to cycles in codimension
 $p>1$.

Though the preceding theorem shows that motivic cohomology of fields is usually enormous,
 the behavior with respect to purely transcendental extensions suggests the following plausible
 conjecture:
\begin{cjti}\label{cjt:main}
Let $K$ be a field finitely generated over its prime field $F$.
 Then, one has the following vanishing of the rational motivic cohomology of $K$ in
 terms of its Kronecker dimension $\delta(K)$
$$
\HB^{n,i}(K)=0
$$
unless $(n,i)=(0,0)$ or $n \in [1,\delta(K)]$.
\end{cjti}
The conjecture is known in any case for $\delta(K) \leq 1$.
 We refer the reader to the end of \Cref{sec:cohm-fields} for further discussion.

\subsection{Guide for reading}

The conventions and main notation of this paper are collected in \Cref{sec:Conv}.
 We use the language of $\infty$-categories,
 whose interest for generic motives has already been highlighted in the introduction.
 A comprehensive review of the foundational concepts required for the paper is provided
 in the Appendix, \Cref{sec:infty-cat}.
 In \Cref{sec:mot-cpx}, we present Voevodsky's theory of motives over a perfect field
 using the $\infty$-categorical framework.
 This section serves both as a reference for the remainder of the paper
 and as a concrete introduction to the $\infty$-categorical language
 through a significant case study.
 For readers already familiar with motivic complexes
 and willing to accept their $\infty$-categorical enhancement without further exposition,
 this section can be safely skipped on a first read.
 
\Cref{sec:generic-mot} develops the theory of generic motives,
 extending the definitions given in \cite{Deg1, Deg5} using
 the enhanced framework of $\infty$-categories.
 We prove various properties of these objects,
 in close relation with Voevodsky's homotopy $t$-structure.
 From a topos theoretic side, they can be considered as points according
 to \Cref{cor:gmot-pts}. From the homotopical perspective, they behave
 like projective objects as explained in \Cref{prop:gen-mot_projective}.
 We next give new presentations of the corresponding category (\Cref{sec:present-DMgen})
 and recall the link with Rost cycle modules (\Cref{sec:funct-gen-mot}).

In \Cref{sec:comput-gmot},
 we set the stage with \Cref{cjt:gen-gmot}, addressing the structure of morphisms of generic motives.
 We then proceed to new computations of generic motives,
 extending the known case of rational curves (\Cref{sec:rat-curve})
 to include arbitrary curves (\Cref{sec:positive-genus})
 and surfaces (\Cref{sec:surfaces}).

The final part of these notes, \Cref{sec:cohm-fields-main},
 focuses on motivic cohomology of fields.
 We begin with the case of number fields, showing in \Cref{sec:Borel}
 how to derive the ranks of rational motivic cohomology groups from Borel’s computation of K-groups,
 with particular attention to the interplay with appropriately defined $\lambda$-structures.
 We then then apply the results on generic motives of purely transcendental extensions of rank one
 to prove \Cref{thm:main}: see \Cref{prop:cohm-inf-rk} and \Cref{prop:cohm-uncount}.
 The section concludes with a discussion of \Cref{cjt:main} (in the text: \Cref{conj:vanishing})
 and a brief exploration of a potentially fruitful cohomology theory
 that emerges naturally from these considerations.

\subsection{Conventions}\label{sec:Conv}

\emph{Category theory}.
We will work with $\infty$-categories,
 as defined in \cite{LurieHTT, LurieHA, CisHC}.
 As previously mentioned, a comprehensive account of the theory is provided in \Cref{sec:infty-cat}.
 As a rule, we use script letters for $\infty$-categories and roman letters for (ordinary) categories.

Throughout, all our $\infty$-categories will be stable and presentable monoidal $\infty$-categories
 (see \Cref{df:presentable}, \Cref{df:icat-stable}).
 We will often refer to $\infty$-functors simply as functors when no confusion can arise.
 For a given $\infty$-category $\iC$, we denote by $\Map_\iC$ the mapping space functor
 (see \Cref{num:mapping-icat} and \Cref{rem:exactness-map}).
 We let $\Ho\iC$ be the associated homotopy category (see \Cref{num:associated-htp}),
 equipped with its canonical triangulated structure (see \Cref{thm:stable->triangulated}).
 We write $\Hom_\iC$ for the abelian groups of morphisms
 in $\Ho\iC$, so that 
$$\Hom_\iC(X,Y)=\pi_0 \Map_\iC(X,Y)$$
 where $\pi_0$ is the $0$-th stable homotopy group of the corresponding spectrum.
 The use of $\infty$-categories is particularly justified by the consideration of pro-objects,
 for which we give a quick account in \Cref{sec:pro-infty}.

With the exception of the appendix,
 we say monoidal for symmetric monoidal (applied to functors, categories, $\infty$-categories).
 We say \emph{rigid} or \emph{strongly dualizable} for an object in such a monoidal category that
 admits a strong dual.

We use homological conventions for $t$-structures (see also \Cref{df:infty-t-struct}).
 Concretely, this means:
\begin{itemize}
\item positive objects (denoted $M>0$) are stable under positive suspension,
\item $\Hom(M,N)=0$ for $M \geq 0$ and $N<0$,
\item one has a homotopy exact sequence:\footnote{Recall the $\infty$-categorical terminology \Cref{num:stable}}
 $\tau_{\geq 0} M \rightarrow M \rightarrow \tau_{<0} M$ for all $M$.
\end{itemize}

Given an abelian group $A$, we put $A_\QQ=A \otimes_\ZZ \QQ$ and $A[1/e]=A \otimes_\ZZ\ZZ[1/e]$
 for an integer $e>0$.
 If $\iC$ is a stable $\infty$-category, we let $\iC_\QQ$ (resp. $\iC[1/e]$)
 be its $\QQ$-localization (resp. $e$-localization) --- obtain by localizing
 with respect to morphisms $p.\Id$ for every prime $p$ (resp. morphisms $e.\Id$).

\bigskip

\noindent \emph{Algebraic geometry}.
We will fix a perfect base field $k$.\footnote{This assumption can be dropped
 altogether if one works with $\ZZ[1/e]$-coefficients where $e$ is the characteristic
 exponent of $k$. This comes from the fact that the motivic categories then become
 invariant under inseparable extensions allowing a reduction to the perfect closure.
 See \cite{SusDM}.}

We say smooth for smooth of finite type. We let $\Sm_k$ be the category of smooth $k$-schemes.
 We denote by $\Smc_k$ the additive category of smooth $k$-schemes with finite correspondences
 as morphisms (see below).

Given a scheme $X$, we denote by $X^{(n)}$ (resp. $X_{(n)}$) the set of points
 of $X$ of codimension $n$
 (resp. whose closure has Krull dimension $n$).

A function field over $k$ will be a separable finitely generated extension $E$ of $k$.
 Unless otherwise specified, a \emph{valuation} $v$ on $E/k$ is assumed to be discrete of rank $1$,
 with ring of integers $\cO_v$ that is a $k$-algebra essentially of finite type.
 When one wants to be precise, one says that $v$ is geometric.
 We also simply say that $(E,v)$ is a valued function field.

\bigskip

\noindent \emph{Motives}.
Our notation for motivic complexes is detailed in Section \Cref{sec:mot-cpx},
 together with the natural $\infty$-categorical enhancements.
 We denote by $\un$ the unit for the $\otimes$-structure on motivic complexes,
 effective and stable, and use the special notation for twists:
$$
\un\tw 1=\un(1)[1].
$$ 
\subsection{Acknowledgement}
Among those who greatly inspired me in the theory of motives, Jacob Murre shares a place with Sasha Beilinson
 and Vladimir Voevodsky.

I would like to sincerely thank Simon Pepin-Lehalleur, Jan Nagel, Marc Levine, Raphaël Ruimy, Robin Carlier,
 Swann Tubach, Tess Bouis, Jörg Wildeshaus and Rob de Jeu for various discussions and remarks
 which greatly contributed to the writing of this paper.
 Special thanks go to Olivier Taïbi without whom Borel's regulator would still be inaccessible to me.
 He helped me figure out its crucial compatibility with $\lambda$-operations.

Finally, I am deeply grateful to the editors of this volume, Jan Nagel and Chris Peters, for their trust and unwavering patience, the opportunity to contribute,
 and their thorough help in refining the paper --- from linguistic accuracy to the clarity of its mathematical formulations.
 I thank the anonymous referee for his vigilance.

The author is supported by the ANR project ``HQDIAG'', ANR-21-CE40-0015.

\section{An $\infty$-categorical presentation of Voevodsky's motivic theory}

\subsection{The $\infty$-category of motivic complexes and its stabilization}\label{sec:mot-cpx}

The purpose of this section is twofold.

 First we give  the notation and terminology from Voevodsky's theory of motivic complexes
 over a perfect field $k$, following the assumptions laid out in \Cref{sec:Conv}.
 While the central definition can be given over an arbitrary base,
 as shown in \cite[Part III]{CD3}, several key theorems --- outlined below ---
 require the base field $k$ to be perfect.\footnote{As mentioned, the interested reader
 can find the necessary arguments to work over a non-perfect base field $k$ after inverting
 its characteristic exponent in \cite{SusDM}.}
 This has been presented in several sources: \cite[Chap. 5]{VSF} (the original reference), \cite{MVW},
 \cite{Deg7} and \cite{BV}.

Second, we present the canonical $\infty$-categorical model of all the triangulated monoidal categories
 arising in Voevodsky's theory, so as to take full advantage of the $\infty$-categorical framework.
 This is easy using the model categorical enhancements of the theory given in \cite{Deg9, CD3},
 so that we use these works as references. We will also outline 
 a way to get direct constructions using the theory of $\infty$-categories as recalled in \Cref{sec:infty-cat}.

In this presentation, we work with $\ZZ$-coefficients, but the framework can naturally be extended to
 $R$-linear coefficients for an arbitrary ring $R$.
 To indicate the coefficient rings, we use the notation $\iDMe(k,R)$ and $\iDM(k,R)$,
 with similar conventions for other $\infty$-categories.

\begin{num}\textit{The effective category}. (\cite{Deg7}, \cite[\textsection 4.1, 5.1]{Deg9}, \cite[\textsection 11.1]{CD3})
Recall that our base field $k$ is perfect.
Voevodsky's construction starts with the introduction of the additive
 category $\Smc_k$ of smooth $k$-schemes equipped with finite correspondences.
 This is a monoidal category whose tensor product is defined by the cartesian product on $k$-schemes
 and by the exterior product of finite correspondences.
 Using the graph of morphisms of smooth $k$-schemes, one obtains the so-called \emph{graph functor}
 $\gamma:\Sm_k \rightarrow \Smc_k$.

The category of sheaves with transfers $\Shtr(k)$ over $k$
 is defined as the category of abelian presheaves $F$ on $\Smc_k$ such that $F \circ \gamma$
 is a Nisnevich sheaf.
 An important example is given by the sheaf with transfers $\Ztr(X)$, representable
 by a smooth $k$-scheme $X$. Its value over a smooth $k$-scheme $Y$ is the abelian
 group of finite correspondences $c(X,Y)$.
 For an arbitrary field, one shows that $\Shtr(k)$ is a Grothendieck abelian category.
 In particular, we can consider its derived $\infty$-category
 $\iDer(\Shtr(k))$ (see \Cref{ex:presentable-derived}). Note that,
 following our convention, this is indeed a presentable and stable $\infty$-category
 (see \Cref{ex:GAb-stable-infty}).
 Moreover, according to \Cref{ex:t-structures-icat}, it admits a canonical
 $t$-structure (see \Cref{df:infty-t-struct}). We simply denotes by $\uH_i$ its $i$-th homology object.

One obtains the $\infty$-category of \emph{motivic complexes} $\iDMe(k)$ by
 considering the $\AA^1$-localization of $\iDer(\Shtr(k))$: one formally
 inverts morphisms $\Ztr(\AA^1_X) \xrightarrow{p_*} \Ztr(X)$ induced by the canonical projection
 for an arbitrary smooth $k$-scheme $X$. It is obvious that its homotopy category
 is Voevodsky's triangulated category $\DMe(k)$ of (unbounded) motivic complexes.

By the mechanism of localization in presentable $\infty$-categories
 (see \Cref{ex:left-loc}), $\iDMe(k)$ can be identified with
 the full-$\infty$-subcategory of $\iDer(\Shtr(k))$ made of motivic complexes $K$
 that are $\AA^1$-local: \emph{i.e.} for any smooth $k$-scheme $X$, the induced morphism
$$
p^*:H^*_\nis(X,K) \rightarrow H^*_\nis(\AA^1_X,K)
$$
is an isomorphism. Let us recall the first key theorem of Voevodsky in this theory.
 This is where the assumption that $k$ is perfect occurs.
\begin{thm}\label{thm:Voevodsky-A1}
Consider the above notation and an arbitrary motivic complex $K$.
 Then $K$ is $\AA^1$-local
 if and only if its homology sheaves $\uH_i(K)$ are $\AA^1$-invariant.
\end{thm}
Using the Nisnevich hypercohomology spectral sequence, one reduces to the case
 where $K$ is a sheaf with transfers in degree $0$, in which case the proof can be found
 in \cite[Th. 4.5.1]{Deg7} or \cite{SusDM} in the non perfect case.
 One deduces that the $\infty$-category $\iDMe(k)$ admits a $t$-structure, called the \emph{homotopy $t$-structure},
 whose heart $\HI(k)$ is made of $\AA^1$-invariant sheaves with transfers,
 also called \emph{homotopy sheaves} 
 (see \cite[\textsection 5.1]{Deg9}).

Again by applying the general mechanism of localization in presentable $\infty$-categories,
 one gets the $\AA^1$-localization functor denoted by $L_{\AA^1}$, an exact endofunctor of $\iDer(\Shtr(k))$.
 According to the previous theorem,
 $L_{\AA^1}(K)$ can be identified with Suslin's singular complex $\sus(K)$.
 The \emph{motive} (\emph{aka} motivic complex) of a smooth $k$-scheme $X$ of finite type is the complex
 $M(X)=\sus(\Ztr(X))$.

One gets a monoidal $\infty$-categorical structure on $\iDMe(k)$ by using the monoidal model structure
 of \cite[Ex. 4.12]{CD1} or \cite{ROmod}, and applying the construction of \Cref{num:model-tensor}.\footnote{We
 sketch below a procedure to avoid the recourse to one of these monoidal model categories.}
 The tensor product on $\iDMe(k)$ is characterized by the fact that the $\infty$-functor 
 $M:\Nrv \Smc_k \rightarrow \iDMe(k)$ is monoidal: $M(X) \otimes M(Y)=M(X \times Y)$.

The $\infty$-category of \emph{effective geometric motives} $\DMgme(k)$ is the sub-$\infty$-category of $\DMe(k)$
 spanned by finite sums, extensions, shifts and direct summands of motivic complexes of the form $M(X)$ for $X$ smooth.
 Being geometric in this context is equivalent to being compact in the homotopy category $\DMe(k)$.
\end{num}

\begin{rem}
If one seeks a direct $\infty$-categorical construction of the monoidal structure on $\iDMe(k)$,
 the following approach can be used:
\begin{itemize}
\item Start with the presentable $\infty$-category $\PSh(\Smc)$
 of presheaves on the nerve of $\Smc$ as defined before \Cref{df:presentable}.
\item Equip this $\infty$-category with a monoidal structure induced by the Day convolution product (\Cref{ex:monoidal-icat}).
\item Localize with respect to the Nisnevich topology.\footnote{Using Nisnevich distinguished squares is enough.}
 One will still need to use \cite[Prop. 10.3.9]{CD3} to ensure that this localization is well-behaved.
\item Stabilize this monoidal $\infty$-category (see \Cref{ex:stabilization}).
\item Perform a $\ZZ$-linearization procedure  to get a presentable monoidal $\infty$-category
 whose underlying $\infty$-category is equivalent to $\iDer(\Shtr(k))$.
\end{itemize}
After these steps, one can proceed with the $\AA^1$-localization as described above,
 as it will be compatible with the monoidal structure (in the sense of \ref{df:ifty-monoidal-loc}).

Alternatively, one can directly use the derived $\infty$-category $\iDer(\mathrm{Ab})$,
 equipped with its canonical monoidal structure. Indeed, one can start with the $\infty$-category of presheaves
 $\Fun\big((\Smc)^{\op},\iDer(\mathrm{Ab})\big)$. As in the previous construction, it naturally inherits a monoidal structure
 via the Day convolution product (using \cite{GlasDay} in this context). 
 One then proceeds as before, by applying Nisnevich localization and $\AA^1$-localization.
 This approach bypasses the need for stabilization and $\ZZ$-linearization, and remains closer in spirit
 to the original construction.
\end{rem}

\begin{num}\textit{The stable category}.\label{num:recall_DM} (\cite[\textsection 4.2, 5.2]{Deg9}, \cite[\textsection 11.1]{CD3})
One defines the \emph{Tate twist} by the following formula:
$$
\un(1)=M(\PP^1_k)/M(\{\infty\})[-2]=M(\GG)/M(\{\infty\})[-1].
$$
 Recall that the motivic complex $\un(1)$ is concentrated in degree $1$ equal to the sheaf with transfers
 $\GG$. It is convenient, and justified from the point of view of the homotopy $t$-structure,
 to introduce the following redundant notation for twists:
$$
\un\tw 1=\un(1)[1]=\GG.
$$
Using the construction detailed in \Cref{num:otimes-inversion},
 there exists a universal (monoidal stable presentable) $\infty$-category $\iDM(k)$ along with
 a functor $$\Sigma^\infty:\iDMe(k) \rightarrow \iDM(k)$$ called the infinite (Tate) suspension functor,
 such that $\Sigma^\infty \un(1)$ is $\otimes$-invertible.
 The functor $\Sigma^\infty$ admits a right adjoint $\Omega^\infty$,
 the infinite (Tate) loop space functor.
 We will still simply denote by $\un_k=\Sigma^\infty M(\Spec(k))$ the unit for
 the tensor structure on $\iDM(k)$. Objects of $\iDM(k)$ are called \emph{ordinary motivic spectra}.
 Here, we add the adjective ordinary to make a difference with the objects of the stable homotopy $\infty$-category
 $\iSH(k)$ which are now frequently called motivic spectra. In this paper, there will be no risk of confusion
 so we just use the terminology \emph{motivic spectra} for objects of $\iDM(k)$.
 We denote its homotopy category by  $\DM(k)$, aligning with the notation used in the references indicated above.

There exists a unique $t$-structure on the $\infty$-category $\iDM(k)$ 
 (see \Cref{df:infty-t-struct})
 such that the functors $\Omega^\infty$
 and tensoring with $\un\{1\}$ are $t$-exact (see \cite[Prop. 4.7]{Deg11}).
 Its heart is equivalent to the abelian category $\HM(k)$
 of \emph{homotopy modules with transfers}
 (we will simply call them homotopy modules in this work, as there is no risk of confusion
 with the objects of the heart of the stable homotopy category over $k$).

To describe the latter, we introduce Voevodsky's $(-1)$-construction.
 Given a homotopy sheaf $F$, one defines a new homotopy sheaf $F_{-1}$, whose sections
 over a smooth $k$-scheme $X$ fit into the following split short exact sequence:\footnote{Alternatively,
 one can show that $F_{-1}=\uHom(\GG,F)$ using the internal Hom in $\HI(k)$.}
\begin{equation}\label{eq:(-1)-construction}
0 \rightarrow F(X) \xrightarrow{p^*} F(\GGx X) \rightarrow F_{-1}(X) \rightarrow 0.
\end{equation}
One then defines a homotopy modules as a
 $\ZZ$-graded homotopy sheaf $F_*=(F_n)_{n \in \ZZ}$ equipped
 with a sequence of isomorphisms $(F_{n+1})_{-1} \simeq F_n$ for $n \in \ZZ$.
 When $\E$ is a motivic spectrum, we denote by $\uH_{n,*}(\E)$ its $n$-th homology object,
 viewed as a homotopy module, where the symbol $*$ refers
 to the internal grading.

A cornerstone of the theory, Voevodsky’s Cancellation Theorem (see \cite{VoeCancel})
 asserts that the functor $\Sigma^\infty$ is fully faithful.
 Consequently, we will use the same notation for both $M(X)$ and $\Sigma^\infty M(X)$.
 Moreover, it follows that $\Sigma^\infty$ induces a canonical fully faithful functor
 on the respective hearts:
$$
\sigma^\infty:\HI(k) \rightarrow \HM(k).
$$

As in the effective case, the $\infty$-category of \emph{geometric motives} $\iDMgm(k)$ is the sub-$\infty$-category of $\iDM(k)$
spanned by finite sums, extensions, shifts and direct summands of motives of the form $M(X)(n)$ for $X$ smooth, $n \in \ZZ$.
 Again, being geometric is equivalent to being compact in the homotopy category $\DM(k)$.
\end{num}

\begin{num}\textit{Duality and Chow motives}.\label{num:recall-dual}
Let $X$ be a smooth $k$-scheme. We let $h(X):=M(X)^\vee=\uHom(M(X),\un_k)$ be its canonical dual,
 the \emph{cohomological motive} associated with $X$.\footnote{In terms of the six functors formalism,
 one has $h(X)=p_*(\un_k)$, where $p:X \rightarrow \Spec(k)$ is the structural map.}

If $X/k$ is smooth projective of dimension $d$, it follows from \cite[Th. 5.23]{Deg8}
 that $M(X)$ is rigid in $\iDM(k)$ with dual $h(X)$.
 Moreover, there exists a canonical isomorphism $h(X) \simeq \M(X)(-d)[-2d]$.
 For a general smooth $k$-scheme $X$ of dimension $d$ (which can even be assumed to be non-smooth but of finite type),
 one also obtains that $M(X)$ is rigid in the $e$-localized category $\DM(k)[1/e]$,
 with dual $h(X)$. Moreover, one gets an isomorphism $h(X) \simeq M^c(X)(-d)[-2d]$, where $M^c(X)$
 is the associated motive with compact support in $\DM(k)[1/e]$
 (see \cite[Chap. 5, 4.3.7]{VSF} if $e=1$, and \cite[Prop. 8.10, Rem. 7.4]{CD5} in general).\footnote{In terms of six operations,
 one has $M^c(X)=f_*f^!(\un_k)$, \emph{loc. cit.} 8.10.}

Our reference for Chow motives will be \cite{MNP}. One gets a contravariant monoidal fully faithful embedding:
$$
\Chow(k)^{op} \longrightarrow \DMgm(k)=\Ho\iDMgm(k)
$$
with source the pseudo-abelian category of Chow motives over $k$, under which the Chow motive $h(X)=(X,\Delta_X,0)$
 (\emph{op. cit.} \textsection 2.2.1)
 of a smooth projective $k$-scheme $X$
 is mapped to the (homological) motivic complex $\M(X)$.\footnote{This is coherent with the fact $\M(X)=h(X)^\vee$
 as recalled above.}
\end{num}

\begin{rem}\label{rem:DM-bases}
The construction of motivic complexes, effective and stable,
 has been extended to arbitrary noetherian finite dimensional bases $S$.
 This is based on a strategy started by Voevodsky in \cite{VoeSimpl},
 using the extension of finite correspondences due to Suslin and Voevodsky.
 The explicit construction has been written down in \cite[Part III]{CD3},
 which gives appropriate monoidal model categories, and therefore 
 presentable stable monoidal $\infty$-categories $\iDMe(S)$ and $\iDM(S)$.
 It follows the same pattern as above, starting from the derived $\infty$-category
 $\iDer(\Shtr(S))$ of sheaves with transfers over $S$.
 As an example, the objects of $\DMe(S)$ can similarly be described as 
 $\AA^1$-local complexes in $\iDer(\Shtr(S))$.\footnote{One should be careful
 that the simpler description of $\AA^1$-local objects obtained over a perfect field
 using \Cref{thm:Voevodsky-A1} fails over positive dimensional bases,
 as shown by Joseph Ayoub in \cite{AyoubC}.}
 At this time, it is known that $\iDM(S)$ admits all of the six functors, but the localization property
 for closed immersions has only been established in certain restricted cases (see \cite[11.4]{CD3} for details).

On the other hand,
 an alternative strategy to define an appropriate category of stable motivic complexes
 was outlined by Oliver Röndigs and Paul Arne Østvær in their work \cite{ROmod}.
 Indeed, in \emph{loc. cit.} they proved that for a field $k$ of characteristic $0$,
 the triangulated category $\DM(k)$ is equivalent to the category of  modules over the \emph{Eilenberg-MacLane
 motivic spectrum} defined by Voevodsky (\cite{VoeOpen}). Since the latter object was also defined
 over an arbitrary base, one could try to look at modules over some arbitrary base $S$ to define
 motivic complexes over $S$. This strategy can be made effective, but one still runs into
 a major problem in order to get the six functors formalism: one has to prove that the
 motivic Eilenberg MacLane is stable under pullbacks, which is the last conjecture,
 conjecture 17, of the paper \cite{VoeOpen}.

This conjecture was solved in \cite[16.1.7]{CD3} for rational coefficients and a geometrically unibranch
 base scheme $S$. More generally, one can show that the rationalization of $\iDM(S)$ is equivalent
 to the category of modules over an appropriately defined rational motivic Eilenberg-MacLane spectrum
 (\cite[14.2.9]{CD3}).
 Integrally, Voevodsky's conjecture is known for pullbacks of regular $k$-schemes where $k$ is a fixed
 base field; see \cite[Cor. 3.6]{CD5}. This result enabled the authors to realize the program
 of Röndigs and Østvær in their work. Furthermore, by restricting attention to $k$-schemes, inverting the characteristic exponent
 of $k$, and replacing the Nisnevich topology with the $\cdh$-topology, it is proved in \emph{loc. cit.}
 that $\cdh$-local motivic complexes do satisfy the localization property,
 are equipped with the six functors formalism,
 and are equivalent to the category of modules over the appropriately defined Eilenberg-MacLane motivic ring spectrum.

Finally, the most successful strategy to get a working theory of stable motivic complexes
 over arbitrary bases has been developed by Markus Spitzweck in \cite{Spit}, still adopting
 the general strategy of Röndigs and Østvær. Drawing on Bloch's higher Chow groups
 and their extension over Dedekind domains by Marc Levine, Spitzweck constructs
 a suitable Eilenberg-MacLane motivic spectrum over a Dedekind scheme.
 By taking pullbacks in the stable motivic homotopy category, one then obtains the desired $E_\infty$-ring spectrum.

Modules over this $E_\infty$-ring spectrum inherit all the expected properties
 of the six functors formalism from the motivic stable homotopy category
 as established by Ayoub in \cite{Ayoub1}. With rational coefficients, Spitzweck's
 theory coincides with that of \cite{CD3} and that of \cite{AyoubEt}. 
\end{rem}

\subsection{A quick review of motivic cohomology}

\begin{num}\textit{Definition via motivic complexes}.\label{num:cohm-integral}
Let $k$ be a field.
 Let $X$ be a smooth $k$-scheme, and $(n,i) \in \ZZ^2$ be a pair of integers.
 Following Voevodsky (\cite[chap. 5]{VSF}),
 one can define the motivic cohomology of $X$ in degree $n$ and twists $i$ as
 morphisms in the category $\DM(k)$:
$$
\Hmot^{n,i}(X)=\Hom_{\DM(k)}(\M(X),\un(i)[n]).
$$
It follows from the so-called continuity property (see \cite[11.1.24, 11.1.25]{CD3}) that
 these cohomology groups commute with projective limits of schemes whenever the transition morphisms
 of an essentially affine pro-scheme $(X_\lambda)_{\lambda \in \Lambda}$ are dominant (see \emph{loc. cit.} for details).
 So when $X=\lim_{\lambda \in \Lambda} X_\lambda$ we get under the preceding assumptions:
\begin{equation}\label{eq:continuity}
\Hmot^{n,i}(X) \simeq \colim_{\lambda \in \Lambda} \Hmot^{n,i}(X_\lambda).
\end{equation}
These motivic cohomology groups can be computed with the Nisnevich cohomology of the $i$-th motivic complex $\un(i)=\ZZ(i)=\sus(\Ztr(\GG^{\wedge,i}))$:\footnote{Here, $\GG^{\wedge,i}$
 denotes the smash product of the pointed scheme $(\GG,1)$,
 and $\Ztr(\GG^{\wedge,i})$ denotes the associated sheaf with transfers.
 Equivalently, $\Ztr(\GG^{\wedge,i})$ is the cokernel of the map:
$$
\oplus_{i=1}^n \Ztr(\GG^{i-1}) \rightarrow \Ztr(\GG^{n})
$$
obtained by considering all the possible inclusions.}
$$
\Hmot^{n,i}(X) \simeq H^n_\nis(X,\ZZ(i)).
$$

This isomorphism follows from Voevodsky's theory (recalled above) when $k$ is perfect.
 In the general case, $k$ is essentially smooth over its prime field $F$, so $X$ is also essentially smooth over $F$,
 and one to reduce to the case of $F$ using the continuity property.
 Along with the (quasi-)isomorphisms of motivic complexes
$$
\un(i) \simeq \begin{cases}
\ZZ & i=0, \\
\GG[-1] & i=1,
\end{cases}
$$
this leads to the following computations of the motivic cohomology groups:
\begin{equation}\label{eq:cohm1}
\Hmot^{n,i}(X)=\begin{cases}
\ZZ^{\pi_0(X)} & i=n=0, \\
\cO_X(X)^\times &  i=n=1, \\
\Pic(X) & (n,i)=(2,1), \\
0 & \text{$i<0$ or  $n>2i$ or $n-i>\dim(X)$.}
\end{cases}
\end{equation}
One further gets for any integer $n \geq 0$ the following isomorphisms:
\begin{align}
\label{eq:cohm2} \Hmot^{2n,n}(X) & \simeq \CH^n(X) \\
\label{eq:cohm3} \Hmot^{n,n}(E) & \simeq \KM n(E)
\end{align}
with respectively the classical Chow groups of codimension $n$ algebraic cycles
 of $X$ modulo rational equivalence, and the $n$-th Milnor K-theory of any function field $E/k$.

Recall finally that it is known due to a theorem of Voevodsky (\cite{VoeHC})
 that motivic cohomology
 coincides with Bloch's higher Chow groups,\footnote{Bloch's definition, motivated by
 earlier results of Kratzer and Soulé, and by the ongoing excitement that emerged from
 the picture laid down by Lichtenbaum and Beilinson, appeared around 10 years before 
 Voevodsky's definition. For the record,
 it was the desire to correct a mistake in the proof of the localization property
 for higher Chow groups that motivated Suslin to introduce Suslin homology in a
 talk in the CIRM at Luminy in 1987. This seminal definition was a major contribution 
 that shaped Voevodsky's theory of motivic complexes a few years later.
 The error was subsequently fixed by Bloch \cite{BlochMov} and Levine \cite{LevineHC}.}
 according to the following formula:
\begin{equation}\label{eq:cohm&higher-Chow}
\Hmot^{n,i}(X) \simeq \CH^i(X,2n-i).
\end{equation}
\end{num}

\begin{num}\textit{Rational coefficients}.\label{sec:rat-mot-coh}
 Recall that the rational part of the above
 motivic cohomology groups can be compared with Quillen's higher algebraic K-theory
 according to the following fundamental formula:
\begin{equation}\label{eq:rat-Hmot&Kth}
\Hmot^{n,i}(X)_\QQ \simeq K_{2i-n}^{(i)}(X)
\end{equation}
where we have denoted by $K_{2i-n}^{(i)}(X)$ the $i$-th graded parts for the $\gamma$-filtration
 on rational K-theory $K_{2i-n}(X)_\QQ$. As recalled in the introduction,
 the right-hand side was actually 
 taken as a definition of rational motivic cohomology in \cite{Soule} and \cite{BeiHR1}.
 For this reason, we also use the notation due to Riou:
$$
\HB^{n,i}(X)=\Hmot^{n,i}(X)_\QQ.
$$
In the case of a smooth $k$-scheme, the above formula is a consequence of Voevodsky's isomorphism
 \eqref{eq:cohm&higher-Chow}, and the corresponding isomorphism for rational higher Chow groups 
 due to Bloch (\cite{Bloch}).
 A more general proof, valid for an arbitrary regular scheme $X$ (without a base field),
 is given in \cite[14.2.14, 16.1.4]{CD3}.\footnote{For singular schemes $X$,
 several authors have proposed suitable definitions for an integral version of the left-hand side
 of \eqref{eq:rat-Hmot&Kth},
 whose rational part coincides with Soulé's definition.
 This is still an active area of research.}
\end{num}

\section{Review and complement on generic motives}\label{sec:generic-mot}

\subsection{Generic motives and points for the homotopy $t$-structures}

We will use the following convenient definition.
\begin{df}
A $k$-scheme $\cX$ is said to be pro-smooth (resp. essentially smooth) if it is the projective limit
 of a pro-scheme $(X_i)_{i \in I}$ such that $X_i$ is a $k$-scheme of finite type,
 and the transition maps are étale (resp. open immersions) and affine.

 We let $\Sme_k$ be the category of pro-smooth $k$-schemes.
\end{df}
According to \cite[8.13.2]{EGA4}, the category of pro-smooth (resp. essentially
 smooth) $k$-schemes can be viewed as a full subcategory of $\pro\Sm$.
 In fact, the functor $P$ that assigns to a pro-scheme $(X_i)_{i \in I}$ as above 
 the projective limit $\cX=\lim_{i \in I} X_i$ seen as an object of $\Sme_k$ is an equivalence
 of categories. We will say that $(X_i)_{i \in I}$ is a presentation of $\cX$,
 keeping in mind that the choice of a specific presentation is irrelevant. 
 To formulate the next results, we consider an arbitrary quasi-inverse to the functor $P$:
\begin{equation}
\nu:\Sme_k \rightarrow \pro\Sm_k.
\end{equation} 

\begin{ex}\label{ex:essentially-smooth}
\begin{enumerate}[leftmargin=*]
\item Given any smooth scheme $X$, and any point $x \in  X$, the localization $X_{(x)}=\Spec(\cO_{X,x})$
 and henselization $X^h_{(x)}=\Spec(\cO^h_{X,x})$
 of $X$ at $x$ are respectively essentially smooth and pro-smooth over $k$.
\item Our main example will come from function fields $E$ over $k$: $\Spec(E)$
 is essentially smooth over $k$.

As in \cite{Deg5}, one can use a canonical presentation of $\Spec(E)$ by considering
 the set $\mathcal M(E)$ of smooth sub-$k$-algebras $A \subset E$ whose fraction field is $E$.
 Ordered by inclusion, this set is filtered and one has:
$$
\Spec(E)=\lim_{A \in \mathcal M(E)^{op}} \Spec(A).
$$

One gets another presentation by considering a \emph{smooth model} of $E$:
 that is a connected smooth $k$-scheme $X$ with a chosen $k$-isomorphism 
 $\phi:E \simeq \kappa(X)$
 This leads to a canonical isomorphism:
$$
\Spec(E)=\lim_{U \subset X} U
$$
where $U$ ranges over the non-empty open subsets of $X$.
\end{enumerate}
\end{ex}

\begin{num}\label{num:pro-motives_ess-sm}
The composite $\infty$-functor:
$$
\Sm_k \xrightarrow{\M} \iDMgm(k) \rightarrow \pro\iDMgm(k), \quad X \mapsto \M(X)
$$
admits a right Kan extension $\bM$ along the inclusion $\Sm_k \rightarrow \Sme_k$, in the $\infty$-categorical sense
 (see \cite[\textsection 6.4]{CisHC}).
 This follows formally from the fact $\pro\iDMgm(k)$ is complete\footnote{\emph{i.e.} admits all small limits, which follows by construction}
 as an $\infty$-category,
 but one gets a direct construction by using the following commutative diagram:
$$
\xymatrix@R=10pt@C=30pt{
& \Sm_k\ar^\M[r]\ar[ld]\ar[d] & \iDMgm(k)\ar[d] \\
\Sme_k\ar@/_8pt/_{\bM}[rr]\ar^{\nu}[r] & \pro\Sm_k\ar^-{\pro \M}[r] & \pro\iDMgm(k).
}
$$
Indeed, following our convention for pro-objects (see \Cref{num:notation-plimit}),
 we have the following explicit formula:
$$
\bM(\cX)=\plim{i \in I} \M(X_i)
$$
for any presentation $(X_i)_{i \in I}$ of the essentially smooth $k$-scheme $\cX$.
 Following \cite[Def. 3.3.1]{Deg5}, we highlight the following pro-motives:
\end{num}
\begin{df}\label{df:genmot}
Using the above notation, one defines the \emph{generic motive} associated with a function field $E/k$
 as the pro-motive $\bM(E):=\bM(\Spec(E))$.

The $\infty$-category $\iDMgen(k)$ of generic motives (over $k$) is defined as
 the full sub-$\infty$-category
 of $\pro\iDMgm(k)$ whose objects are of the form $\bM(E)\tw n$ for any function field $E/k$
 and any integer $n \in \ZZ$.
\end{df}

\begin{rem}
One should be mindful that the above definition of generic motives is substantially different from that
 in \emph{loc. cit.} as it retains the higher homotopy information, contained in the mapping
 spaces of the $\infty$-category $\iDMgen(k)$.
 However, we will see in \Cref{rem:compare-old&new_DMgen} that the homotopy category of the latter
 coincides with the category introduced in \emph{loc. cit.}
\end{rem}

\begin{num}
The reason for focusing on these particular pro-objects
 is that they form a conservative family of \emph{points} for the $t$-$\infty$-category $\iDM(k)$,
 equipped with its homotopy $t$-structure.
Before making this assertion precise in the following corollary, we introduce a new notation:
 given an $R$-linear presheaf $F$ on $\Sm_k$,
 we denote by $\hat F$ its \emph{left} Kan extension along the inclusion $\Sm_k \rightarrow \Sme_k$.
 Concretely, for any essentially smooth $k$-scheme $\cX$, we have:
$$
\hat F(\cX) \simeq \colim_{i \in I} F(X_i)
$$
for any presentation $(X_i)_{i \in I}$ of $\cX$.
 When $F_*$ is a homotopy module (see \Cref{num:recall_DM}), we simply denote by $\hat F_*$
 the extension to $\Sme$ just defined. In particular, when $\E$ is a motivic spectrum,
 with associated $n$-th homology object $\uH_{n,*}(\E)$ seen as a homotopy module
 (see again \Cref{num:recall_DM}), we will denote by $\hat \uH_{n,*}(\E)$ the extension
 to $\Sme$ defined above.
\end{num}
\begin{prop}\label{prop:pts_DM}
Let $\cO$ be an essentially smooth semi-local $k$-algebra,
 $(n,i) \in \ZZ^2$ a pair of integers,
 and $\E$ a stable motivic complex over $k$. Then there exists a canonical isomorphism:
$$
\Hom_{\pro\iDM(k)}(\bM(\cO)\tw i[n],\E) \simeq \hat \uH_{n,-i}(\E)(\cO)
$$
where on the left-hand side, $\bM(\cO)=\bM(\Spec(\cO))$ and $\E$ is seen as a constant pro-object.
\end{prop}
\begin{proof}
First, recall that $\uH_{n,*}(\E\tw{-i})=\uH_{n,*-i}(\E)$. This reduces to the case $i=0$.
 Next, using the cancellation theorem (recalled in \Cref{num:recall_DM}), one can compute
 the relevant group morphism in $\iDMe(k)$.
 We then apply the hypercohomology spectral sequence for the homotopy $t$-structure,
 extended to the essentially smooth $k$-scheme $\mathcal X=\Spec(\cO)$, which has
 the following form:
$$
E_2^{p,q}=H^p(\mathcal X,\uH_{-q,0}(\E)) \Rightarrow \E^{p+q}(\mathcal X).
$$
The $E_2$-term is the cohomology with coefficients in a homotopy module,
 and this coincides with the Chow groups with coefficients in the associated
 cycle module, according to\cite[3.10, 3.11]{Deg9}.
 One deduces from the Gersten conjecture for Chow groups with coefficients,
 as proved in \cite[Th. 6.1]{Rost}, that the $E_2$-term of the above spectral sequence
 is concentrated on the line $p=0$. This concludes the proof.
\end{proof}

\begin{ex}\label{ex:gen-morph&motiv-coh}
Let $X$ be a smooth projective $k$-scheme of dimension $d$.
 Using duality (see \Cref{num:recall-dual}) and the continuity property
 \eqref{eq:continuity}, one immediately deduces the following computation:
$$
\hat \uH_{n,-i}(\M(X))(E)=\Hom_{\pro\iDM(k)}(\bM(E)\tw i[n],\M(X))=\Hmot^{2d-n-i,d-i}(X_E)
$$
where $X_E=X \times_k E$.
\end{ex}

\begin{num}
We now turn to the promised interpretation of generic motives
 as ``points'' for the homotopy $t$-structure on $\iDM(k)$.
 The term points is used here by analogy with the more classical notion
 of points in topos theory. Indeed, a point of a topos can be seen as a fiber functor,
 an exact functor to the punctual topos.

In our setting, we will leverage the fact (see \Cref{rem:exactness-map})
 that any stable $\infty$-category
 admits mapping spaces with values in the $\infty$-category of spectra $\iSp$,
 which we see as the analog of the punctual topos.\footnote{In fact,
 the $\infty$-category $\iSp$ is the initial presentable stable $\infty$-category.}
 Then the exactness property of fiber functors is replaced by the $t$-exactness property,
 with respect to the standard (Postnikov) $t$-structure on spectra.
\end{num}
\begin{cor}\label{cor:gmot-pts}
For any function field $E/k$ and integer $i\in \ZZ$, the $\infty$-functor
$$
\iDM(k) \rightarrow \iSp, \E \mapsto \Map_{\pro\iDM(k)}(\bM(E)\tw i,\E)
$$
is $t$-exact, with respect to the homotopy $t$-structure on the left and the
 canonical $t$-structure on the right.
 Moreover, the family of these $t$-exact $\infty$-functors indexed by a function field $E/k$ and an integer $i \in \ZZ$,
 is conservative.\footnote{\emph{i.e.} a morphism in $\iDM(k)$ is an isomorphism
 if and only if its image by all the functors of the family considered is an isomorphism.
 See also \Cref{num:functors-icat}.}
\end{cor}
This corollary is a compact form of the following two assertions:
\begin{itemize}
\item A motivic spectrum $\E$ is non-negative (resp. negative) for the homotopy $t$-structure
 if and only if for any $(E,i)$ as above, the spectrum $\Map(\bM(E)\tw i,\E)$ is non-negative (resp. negative).
\item A morphism of motivic spectra $f:\E \rightarrow \F$ is an isomorphism
 if and only if for any $(E,i)$ as above, the map $\Map(\bM(E)\tw i,f)$ is an isomorphism
 in the $\infty$-category of spectra.
\end{itemize}
\begin{proof}
Indeed, \Cref{prop:pts_DM} applied to the case $\cO=E$ gives the formula:
$$
\pi_n \Map_{\pro\iDM(k)}(\bM(E)\tw i,\E)=\Hom_{\pro\iDM(k)}(\bM(\cO)\tw i[n],\E)
 =\hat \uH_{n,-i}(\E)(\cO).
$$
Then the corollary follows from the definition of the homotopy $t$-structure
 in \Cref{num:recall_DM}, and one of the main results of Voevodsky's theory of motivic complexes:
 the fact that function fields form a conservative family for homotopy sheaves
 (see \cite[Prop. 4.7]{Deg7}).
\end{proof}

\begin{rem}\label{rem:functor-Rost-transform}
In more classical terms, one gets that the family of exact functors with values in the category
 of abelian groups:
\begin{equation}\label{eq:functor-Rost-transform}
\HM(k) \rightarrow \mathrm{Ab}, F_* \mapsto \Hom_{\pro\HM(k)}\big(\uH_0(\bM(E))\tw i,F_*\big)=\hat F_{-i}(E)
\end{equation}
is conservative. In fact, this assertion can be deduced directly from Voevodsky's result \cite[Prop. 4.7]{Deg7}.
\end{rem}

\begin{num}
Beilinson has already discussed the conjectural \emph{projective} nature of generic motives
 with respect to the conjectural motivic $t$-structure (see \cite[Cor. 3.3]{BeiGen}).

We will now formulate a similar, non-conjectural, property of generic motives
 with respect to the homotopy $t$-structure.
 Let us first recall from \cite[\textsection 6.1]{GLdB} that,
 given an arbitrary stable $\infty$-category $\T$ equipped with a $t$-structure,
 an object $K$ of $\T$ is called \emph{derived $t$-projective} if the following two conditions
 hold:
\begin{enumerate}[label=(\roman*)]
\item For all $L$ in $\T$, $\Hom_\T(K,L)\simeq \Hom_{\T^\hrt}(H_0(K),H_0(L))$.
\item The object $H_0(K)$ is projective in the abelian category $\T^\hrt$.
\end{enumerate}
To address the case of generic motives, one needs an appropriate reformulation.
\end{num}
\begin{df}
Let $\T$ be an $\infty$-category equipped with a $t$-structure (see \Cref{df:infty-t-struct}),
  and let $\mathcal K$ be a pro-object of $\T$.
 The canonical homological $\infty$-functor $H_0:\T \rightarrow \T^\hrt$,\footnote{According to our usual
 abuse of notation, the target is really the nerve of the category $\T^\hrt$. Of course,
 this $\infty$-functor is equivalent to the more classical one: $H_0:\Ho(\T) \rightarrow \Ho(\T)^\hrt$,
 using \Cref{ex:associated-htp}(2) and the identification $\T^\hrt=\Ho(\T)^\hrt$.}
 induces a functor on pro-categories 
$$
\pro\T \rightarrow \pro{\T^{\hrt}}
$$
that we will abusively denote by $H_0$.\footnote{Note that according to \Cref{sec:pro-infty},
 the right-hand side is the ordinary category of pro-objects of $\T^\hrt$, which is therefore abelian.
 However, we will not use this fact in the sequel.}

One says that $\mathcal K$ is \emph{derived pro-$t$-projective} if the following conditions hold:
\begin{enumerate}[label=(\roman*)]
\item For all $\E$ in $\T$, $\Hom_{\pro\T}(\mathcal K,\E)\simeq \Hom_{\pro\T^\hrt}(H_0(\mathcal K),H_0(\E))$.
\item The functor $\T^\hrt \rightarrow \ab, A \mapsto \Hom_{\pro\T^\hrt}(H_0(\mathcal K),A)$
 is exact.
\end{enumerate}
\end{df}
The second condition sometimes appears in the literature under the name ``pro-projective''.
 The next proposition gives us the desired property of generic motives.
\begin{prop}\label{prop:gen-mot_projective}
Let $E/k$ be a function field and $i \in \ZZ$ an integer.
 Then the generic motive $\bM(E)\tw i$ is derived pro-$t$-projective
 with respect to the homotopy $t$-structure on $\iDM(k)$.
\end{prop}
\begin{proof}
We prove that the pro-motive $\bM(E)\{i\}$ satisfies the conditions
 of the preceding definition. In fact, point (i) follows from \Cref{prop:pts_DM},
 while point (ii) follows from \Cref{rem:functor-Rost-transform}. 
\end{proof}

\subsection{Categorical variants and presentations of generic motives}\label{sec:present-DMgen}

In this section, we will show how one can compare various constructions of generic motives.
 We start with a key lemma.
\begin{lm}\label{lm:compute_HOM_DMgen}
Let $E$, $F$ be function fields over $k$, and $n \in \ZZ$.
 Let $X$, $Y$ be smooth $k$-models of $E$ and $F$, respectively.
 Then there exists a canonical isomorphism, where $V$ and $U$ range over non-empty open subsets:
\begin{align*}
&\pi_0 \left(\underset{V \subset Y}\lim\  \underset{U \subset X}\colim \Map_{\iDM(k)}(\M(U),\M(V)\tw n)\right) \\
 & \qquad \simeq  \underset{V \subset Y}\lim \underset{U \subset X}\colim\left(\pi_0 \Map_{\iDM(k)}(\M(U),\M(V)\tw n)\right)
\end{align*}
where we emphasize that on the left-hand side the colimit and limit are taken within the $\infty$-category
 of spaces, while on the right-hand side, they are taken in ordinary categorical sense, in the category of
 abelian groups.
\end{lm}
\begin{proof}
Since filtered colimits are exact, one first deduces the following isomorphism for any $V \subset Y$:
$$
\pi_0 \left(\underset{U \subset X}\colim \Map_{\iDM(k)}(\M(U),\M(V)\tw n)\right) 
 \simeq \underset{U \subset X}\colim\left(\pi_0 \Map_{\iDM(k)}(\M(U),\M(V)\tw n)\right).
$$
Applying \Cref{prop:pts_DM}, one further gets:
$$
\underset{U \subset X}\colim\left(\pi_0 \Map_{\iDM(k)}(\M(U),\M(V)\tw n)\right)
 \simeq \hat \uH_{0,n}(V)(E).
$$
But if $j:W \subset V$ is a non-empty open subset, it follows from \cite[Cor 3.22]{Deg7}
 that the morphism of homotopy invariant sheaves with transfers $\uH_0(W) \rightarrow \uH_0(V)$
 is an epimorphism. Voevodsky's cancellation theorem implies 
 that the same is true for $\uH_{0,n}(W)(E) \rightarrow \uH_{0,n}(V)(E)$.\footnote{We need the cancellation
 theorem only in the case $n>0$ as Voevodsky's $(-1)$-functor --- see \eqref{eq:(-1)-construction} --- is exact.}
 This implies that the inductive system:
 $\left(\uH_{0,n}(V)(E)\right)_{V \subset Y}$ satisfies the Mittag-Leffler condition.
 In particular, one deduces (see e.g. \cite[Th. B]{Goerss}):
$$
\pi_0 \left(\underset{V \subset Y}\lim\ \hat \uH_{0,n}(V)(E)\right)
 \simeq \underset{V \subset Y}\lim \left(\hat \uH_{0,n}(V)(E)\right)
$$ 
and this concludes the proof.
\end{proof}

\begin{num}
For the next statement, we consider the following composition of $\infty$-functors:
$$
\iDMgen(k) \hookrightarrow \pro\iDMgm(k) \rightarrow \pro{\left(\Ho \iDMgm(k)\right)}=\pro{\DMgm(k)}
$$
where the last map is obtained by applying the pro-object construction to the natural map
 $\iDMgm(k) \rightarrow \Ho\iDMgm(k)$. It follows from abstract non-sense (\Cref{ex:associated-htp}(2))
 that this defines an (ordinary) exact functor:
\begin{equation}\label{eq:compare-iDMgen}
\Ho \iDMgen(k) \rightarrow \pro{\DMgm(k)}.
\end{equation}
\end{num}
We now deduce the following result, which allows us to compare 
 \Cref{df:genmot} with the older definition of generic motives used in \cite{Deg5}.
\begin{prop}\label{rem:compare-old&new_DMgen}
The functor \eqref{eq:compare-iDMgen} is fully faithful, with essential image given by
 the ordinary category of generic motives defined in \cite[Def. 3.3.1]{Deg5}.

In other words, the ordinary category of generic motives defined in \emph{loc. cit.}
 is equivalent to the homotopy category of the $\infty$-category of generic motives $\iDMgen(k)$ defined here.
\end{prop}
\begin{proof}
One should be careful that this is not a tautology.
 Indeed, upon unwinding definitions and using Formula \eqref{eq:Map-pro},
 we find that this is precisely the content of the preceding lemma.
\end{proof}

We also remark that, by combining the preceding result with \Cref{prop:pts_DM},
 one can compute homotopy classes of morphisms of generic motives
 inside an abelian category, as expressed in the following proposition.\footnote{For ordinary generic motives,
 this was already remarked in \cite[Cor. 3.4.7]{Deg5}.}
\begin{prop}\label{thm:higher-morph-genmot}
The functor
$$
\Ho \iDMgen(k) \rightarrow \pro\HM(k), \bM(E)\tw n \mapsto \uH_{0,*+n}(\bM(E))=:\uH_{0,*}(\bM(E))\{n\}
$$
is fully faithful.
\end{prop}

\begin{num}
We conclude this subsection by showing that,
 at the cost of inverting the characteristic exponent $e$ of the base field $k$,
 it is possible to consider ind-(cohomological motives) for modeling generic motives.

Recall that the presentable stable $\infty$-category $\DM(k)$ is generated by its compact objects,
 which are precisely the geometric motives (as recalled in \Cref{num:recall-dual}).
 This observation leads to an interpretation of
 the ``big'' category of motives $\DM(k)$ as ind-geometric objects:\footnote{This is well-known,
 but see \cite[Proposition 5.3.5.11]{LurieHTT} for a very general proof.}
\begin{equation}\label{eq:DM=ind-DMgm}
\DM(k) \simeq \ind\DMgm(k).
\end{equation}
Working within a sufficiently large universe,
 we formally deduce the following equivalence of $\infty$-categories
 (see \Cref{rem:ind-pro}):
$$
(\pro\iDMgm(k))^{\op}=\ind\iDMgm(k)=\iDM(k).
$$
Thus, in principle, one can view generic motives as certain motivic spectra,
 provided one works in the opposite category.
 We can make this perspective more concrete by inverting
 the characteristic exponent $e$ of $k$.
 Indeed, as recalled in \Cref{num:recall-dual}, every object of $\DMgm(k)[1/e]$ is rigid,
 and for any smooth $k$-scheme $X$, the dual of $M(X)$ can be computed as the cohomological motive $h(X)$.
 Based on these observations, we obtain the following proposition:
\end{num}
\begin{prop}\label{prop:gen_mot&coh_mot}
The canonical (contravariant) $\infty$-functor:
$$
\big(\iDMgme(k)[1/e]\big)^{op} \rightarrow \iDM(k)[1/e], \big(\bM(E)\{i\}\big) \mapsto \underset{X \in \mathcal M(E)^{op}}{\colim} \big(\h(X)\{-i\}\big)
$$
where $X=\Spec(A)$ ranges over the smooth affine models of $E$ (\Cref{ex:essentially-smooth}(2))
 and the colimit is taken in the $\infty$-category $\iDM(k)$, is fully faithful.
\end{prop}
\begin{proof}
The proof reduces to the following computation of mapping spaces:
\begin{align*}
\Map\big(\colim_Y \h(Y)\tw{-i},\colim_X \h(X)\big)
&\simeq \lim_Y \Map\big(\h(Y)\tw{-i},\colim_X \h(X)\big) \\
&\simeq \lim_Y \colim_X\Map\big(\h(Y)\tw{-i},\h(X)\big) \\
&\simeq \lim_Y \colim_X\Map\big(\M(X),\M(Y)\tw{i}\big) \\
& \simeq \Map\big(\bM(E),\bM(F)\tw i\big)
\end{align*}
where $X$ and $Y$ range over the smooth affine models of some function fields $E$ and $F$, respectively.
 The first isomorphism is formal, the second one follows as $\h(Y)\tw{-i}$ is compact,
 the third one follows by duality, and the last one follows from \eqref{eq:Map-pro}.
\end{proof}

\begin{num}\textit{Ind-Chow motives as (co)models for generic motives}.
The preceding proposition shows that, by working in the opposite category and inverting the characteristic exponent $e$,
 one can replace generic motives by the following motivic spectrum:\footnote{In fact, one can reproduce,
 in a dual fashion, the definitions of the preceding section.
 In particular, there exists a right Kan extension $\bh$ of the functor $\h:\Sm_k \rightarrow \DM(k)$
 along the inclusion $\nu:\Sm_k \rightarrow \Sme_k$:
$$
\xymatrix@C=24pt@R=12pt{
\Sm_k^{op}\ar_\nu[d]\ar^-\h[r] & \DM(k) \\
\Sme_k^{op}\ar_-{\bh}[ru] &
}
$$
such that $\bh(\cX)=\hocolim_{i \in I} \h(X_i)$.}
$$
\bh(E)\{i\}=\colim_{X \in \mathcal M(E)^{op}} \h(X)\tw i
$$
where the colimit is computed in the $\infty$-category $\iDM(k)[1/e]$.

The sub-$\infty$-category of $\iDM(k)[1/e]$ is then equivalent,
 according to the preceding proposition, to the opposite of the $\infty$-category of generic motives
 after inverting $e$.
\end{num}

\begin{rem}
As remarked by the referee, the article \cite{BeiGen} uses the opposite category
 of Voevodsky's triangulated category $\DMgm(k,\QQ)$. This can lead to confusion
 since it is not canonically triangulated. However, the above proposition shows
 that, due to the good properties of Voevodsky's (rational) motives,
 this is not a significant problem. Indeed, in the context of \cite{BeiGen},
 one can work directly in the triangulated category $\DM(k,\QQ)$,
 and use the motivic spectrum
 $\bh(K)$ in place of the ind-object $\mathcal C(\eta)$ in section 3.1 of \emph{loc. cit.}
 without altering the essential content of Beilinson's discussion.
\end{rem}

\subsection{Morphisms of generic motives and cycle modules}\label{sec:funct-gen-mot}

\begin{num}
Generic motives enjoy a rich functorial structure,
 closely mirroring that of Milnor K-theory as developed by Rost in the framework of cycle modules
 \cite[Def 1.1]{Rost}.
 In fact, this functoriality can be described by a category $\efld_k$
 whose objects are pairs $(E,i)$ consisting of a function field $E$ and an integer $i \in \ZZ$,
 and morphisms are described by generators and relations: see the appendix, \Cref{df:efld}.
 Then the aforementioned functoriality can be expressed through the existence of a contravariant functor
 (see \cite[Th. 5.1.1]{Deg5}):
\begin{equation}\label{eq:morph_gen-mot}
\rho:(\efld_k)^{\op} \rightarrow \Ho \iDMgen(k), (E,i) \mapsto \bM(E)\{-i\}.
\end{equation}
We will use this functor in the next section
 so we briefly recall its definition on the generators of the morphisms,
 and refer to \emph{loc. cit.} for the proof of the needed relations.
 We use notation (D1*), (R1a*), and so on, to indicate the data and relations
 on morphisms of generic motives that correspond to the functor $\rho$, according
 to the notation of \Cref{df:efld}.
\end{num}

\begin{num}
We start with the first two types of functorialities satisfied by generic motives:
\begin{enumerate}[leftmargin=11pt, itemindent=2em]
\item[(D1*)] \textbf{Covariant functoriality in function fields}.--
 This functoriality simply arises from the fact that generic motives are obtained by restriction and $\GG$-twists
 of the $\infty$-functor $\bM:\Sme_k \rightarrow \iDMgm(k)$. 
 In particular, an arbitrary morphism $\varphi:E \rightarrow L$ of function fields over $k$
 induces a canonical map:
$$
\varphi^*:\bM(L)\tw i \rightarrow \bM(E)\tw i
$$
which actually defines an $\infty$-functor from the category function fields over $k$
 to $\iDMgen(k)$.
\item[(D2*)] \textbf{Transfers: contravariance in function fields}.-- By construction, Voevodsky's homological motives
 are functorial in the category of finite correspondences $\Smc_k$.
 This gives rise to an $\infty$-functor:
$$
\M:\Smc_k \rightarrow \pro\iDMgm(k).
$$
The definition of finite correspondences can be extended to essentially smooth $k$-schemes $\cX$ and $\cY$.
 The corresponding abelian group $c(\cX,\cY)$ of finite correspondences from $\cX$ to $\cY$ is composed
 of algebraic cycles in $\cX \times_k \cY$ whose support is finite and equidimensional over $\cX$.
 Composition is defined as in \cite[Def. 1.16]{Deg7} and has the same properties.
 Thus, one can define a category $\Smec_k$ of essentially smooth $k$-schemes
 with finite correspondences as morphisms.

Using the continuity property of \cite[Prop. 1.24]{Deg7}, one deduces that the preceding functor
 admits a right Kan extension:
\begin{equation}\label{eq:gencor-DM}
\bM:\Smec_k  \rightarrow \pro\iDMgm(k).
\end{equation}
This immediately provides the desired functoriality. A finite extension of function fields
 $\varphi:E \rightarrow L$ corresponds to a finite surjective (and hence equidimensional) morphism
 $f:\cY=\Spec(L) \rightarrow \Spec(E)=\cX$. Taking the transpose of its graph defines
 a finite correspondence $\tra f \in c(\cX,\cY)$ and therefore a \emph{transfer} (or \emph{Gysin}) map:
$$
\varphi_!:\bM(E)\tw i \rightarrow \bM(L)\tw i.
$$
Note that this construction defines an $\infty$-functor,
 from the category of function fields over $k$, equipped with finite morphisms to $\iDMgen(k)$.
\end{enumerate}
\end{num}

\begin{rem}
Note that the functorial properties (D1*) and (D2*),
 together with the relations (R1a*), (R1b*), (R1c*) they satisfy, can be obtained using
 the category $\Smec$. In fact, one can consider the full subcategory $\Smgen_k$
 of $\Smec$ whose objects are the spectra of function fields over $k$.
 Then, by restricting \eqref{eq:gencor-DM}, one obtains an $\infty$-functor
$$
\Smgen_k \rightarrow \iDMgen(k), \Spec(E) \mapsto \bM(E)
$$
which gives a highly structured version of transfers on generic motives.
 Moreover, note that this functor precisely recovers the data (D1*) and (D2*),
 as well as the relations (R1*), of the functor $\rho$.
\end{rem}

\begin{num}\label{num:D3D4}
 Let us now describe the two other types of functoriality for generic motives:
\begin{enumerate}[leftmargin=11pt, itemindent=2em]
\item[(D3*)] \textbf{Action of Milnor K-theory}.--
Let $E/k$ be a function field, and $n \geq 0$ be an integer.
 One deduces from \eqref{eq:cohm3} the following isomorphism:
\begin{equation}\label{eq:gen_mot&KM}
\Hom_{\iDMgen(k)}(\bM(E),\un\tw n) \simeq \Hmot^{n,n}(E) \simeq K_n^M(E).
\end{equation}
Thus, one can associate to any symbol $\sigma=\{u_1,\hdots,u_n\} \in K_n^M(E)$
 a canonical map $\mu(\sigma):\bM(E) \rightarrow \un\tw n$.

Using the diagonal map $\delta:\Spec(E) \rightarrow \Spec(E) \times_k \Spec(E)$,
 or equivalently, the multiplication map $ E \otimes_k E \rightarrow E$, 
 one obtains the required morphism of generic motives:
$$
\gamma_\sigma:\bM(E) \xrightarrow{\delta_*} \bM(E) \otimes \bM(E) \xrightarrow{\Id \otimes \mu(\sigma)} \bM(E)\tw n.
$$
\item[(D4*)] \textbf{Residues}.-- Let $(E,v)$ be a valued function field over $k$.
 Let $\cO_v$ and $\kappa(v)$ be the associated ring of integers and residue field respectively.
 This induces a closed immersion between essentially smooth $k$-schemes
 $i:\cZ=\Spec(\kappa(v)) \rightarrow \Spec(\cO_v)=\cX$, with open complement $j:\cU=\Spec(E) \rightarrow \Spec(\cO_v)$.
 The Gysin triangle extends to this general setting and yields a (split) homotopy exact sequence in
 the stable $\infty$-category $\pro\iDMgme(k)$:
$$
\bM(\kappa_v)\tw 1 \xrightarrow{\partial_v} \bM(E) \xrightarrow{j_*} \bM(\cO_v) \xrightarrow{i^!} \bM(\kappa_v)(1)[2]
$$
The map $\partial_v$ is referred to as the \emph{residue map} associated with $(E,v)$.
\end{enumerate}
This concludes the description of the functor $\rho$ of \eqref{eq:morph_gen-mot} on morphisms.
 We refer the reader to \cite{Deg5} for the detailed proof of the relations.
\end{num}

\begin{num}
Let $\E$ be a motivic complex in $\iDM(k)$.
 Seen as a constant pro-object, it represents a functor $\bh_\E:\Hom_{\pro\iDM(k)}(-,\E)$,
 that can be restricted to the category of generic motives.
 Composing further with the functor $\rho$, one obtains a functor:
$$
\hat \E:\efld_k \xrightarrow{\rho^{\op}} \iDMgen(k)^{\op} \xrightarrow{\bh_\E} \ab
$$
which is, by definition, a cycle premodule. We call it the \emph{Rost transform} of $\E$.
 One can check that it is in fact a cycle module. Furthermore,
 according to \cite[Th. 3.7]{Deg9}, one gets:
\end{num}
\begin{thm}\label{thm:deglise}
For any perfect field $k$, the functor
$$
\mathcal R:\HM(k)=\iDM(k)^\heartsuit \rightarrow \MCycl_k, \E \mapsto \hat \E,
$$
 where $\MCycl_k$ is Rost's category of cycle modules,
is an equivalence of categories, with a quasi-inverse given by the functor that associates
 to a cycle module $M$ its $0$-th Chow group functor $\underline H M:=A^0(-,M)$.

Moreover, if $k$ is non-perfect of characteristic $p>0$, then the pair of functors
 $(\mathcal R[1/p],\underline H[1/p])$ obtained after inverting $p$, is an equivalence
 of categories.
\end{thm}
Note the last assertion was not proved in \emph{loc. cit.} but it easily follows since,
 after inverting $p$, both categories $\HM(k)[1/p]$ and $\MCycl_k[1/p]$ become invariant
 under purely inseparable extensions of the base field $k$.

\begin{rem}\label{rem:gen-mot=reconstruction}
According to \Cref{prop:pts_DM}, for any motivic complexe $\E$, one obtains:
$$
\hat \E(K,n)=\Hom_{\pro\iDM(k)}(\bM(K)\tw n,\E) \simeq \uH_{0,-n}(\E)(K).
$$
In particular, the functor $\E \mapsto \hat \E$ factors through
 the homological functor $\uH_{0,*}$ associated with the homotopy $t$-structure on $\iDM(k)$.

Moreover, the abelian group $\hat \E(K,n)$ can be interpreted as the fiber
 of the Nisnevich sheaf $\uH_{0,-n}(\E)$ at the point (in the topos-theoretic sense) of the Nisnevich site
 $\Sm_k$ corresponding to the function field $K$. 
 As shown in \Cref{cor:gmot-pts},
 generic motives can be viewed as points of the $\infty$-$t$-category $\iDM(k)$,
 equipped with its homotopy $t$-structure.
 The \emph{Rost transform functor} $\mathcal R$ can thus be interpreted as
 the restriction of homotopy modules to this category of points.
 Remarkably, restricting a homotopy module $M$
 to this enriched category --- via the functoriality encoded by $\rho$ ---
 is sufficient to fully reconstruct $M$, as well as morphisms of homotopy modules.

Rost morphisms $(\text D*)$ can be interpreted as \emph{specialization maps}
 between these various points, but one should be careful they are not isomorphisms.
 The category $\iDMgen(k)$ is not a groupoid.
 From this perspective, cycle modules function as generalized local systems,
 with homotopy modules acting as their sheaf-theoretic counterparts.
\end{rem} 
\section{Computations of generic motives}\label{sec:comput-gmot}

Throughout this section, we adopt the notation for motivic complexes
 and generic motives as recalled in the previous two sections.
 Most notably, we make use of the functoriality of generic motives
 as described by the functor \eqref{eq:morph_gen-mot}.

\subsection{Initial remarks about generic motives}

\begin{num}\label{num:hope_PHD}
In \cite[\textsection 9.3.2]{Deg1},
 the question was raised whether morphisms of type (D*) generate
 all morphisms between generic motives. More precisely, is the functor $\rho$, defined in the previous section \eqref{eq:morph_gen-mot},
 full and/or faithful ?

This question has a negative answer (see \Cref{rem:gen-gmot}),
 but there are still interesting cases where the functor $\rho$ does induce an isomorphism on morphisms.
\end{num}
\begin{prop}
Let $(E,n)$ and $(F,m)$ be objects in $\efld_k$ such that $F/k$ is finite.
 Then the following map is an isomorphism:
\begin{align*}
\bigoplus_{x \in \Spec(E \otimes_k F)} K_{m-n}^M(\kappa(x))
 & \rightarrow \Hom_{\iDMgen(k)}(\bM(E)\tw n,\bM(F)\tw m) \\
 \sigma_x & \mapsto \psi_x^* \circ \gamma_{\sigma_x} \circ \varphi_{x!}
\end{align*}
where $\varphi_x:E \rightarrow \kappa(x)$ and $\psi_x:F \rightarrow \kappa(x)$
 are the canonical morphisms induced by the prime ideal $x \subset E \otimes_k F$.

In particular, the map:
$$
\Hom_{\efld_k}\big((E,n),(F,m)\big) \rightarrow \Hom_{\iDMgen(k)}\big(\bM(E)\tw n,\bM(F)\tw m\big)
$$
is an isomorphism.
\end{prop}
Thus, in the preceding case, only (D1), norm maps (D2) and multiplication by symbols (D3) are sufficient
 to describe all morphisms of generic motives.
\begin{proof}
The proposition is obtained from the following computation:
\begin{align*}
\Hom_{\iDMgen(k)}\big(\bM(E)\tw n,\bM(F)\tw m\big)
& \stackrel{(1)}\simeq \Hom_{\pro\iDM(k)}\big(\bM(E) \otimes \bM(F),\un\tw{m-n}\big) \\
& \stackrel{}\simeq \Hmot^{m-n,m-n}(E \otimes_k F) \\
& \simeq \bigoplus_{x \in \Spec(E \otimes_k F)} K_{m-n}^M(\kappa(x))
\end{align*}
where the first two isomorphisms follow as in \Cref{ex:gen-morph&motiv-coh},
 particularly using that $\bM(F)=\M(F)$ is rigid, with strong dual $\M(F)^*=\M(F)$,
 and the last one follows from \eqref{eq:cohm3}.
 The assertion regarding the description of the isomorphism follows
 from the construction of data (D1), (D2), (D3) as recalled above.
\end{proof}

The presentation of the category $\efld_k$ given in \Cref{prop:efld_Rost_present}
 implies a vanishing of morphisms in $\efld_k$ in certain degrees. It is notable 
 that this vanishing also holds in the category of generic motives, at least
 after inverting the characteristic exponent $e$ of $k$.
\begin{prop}
For all function fields $E$ and $F$ over $k$,
 and for all pairs of integers $(n,m) \in \ZZ^2$
 such that $n-m>\dtr_k(F)$, one has:
$$
\Hom_{\iDMgen(k)}(\bM(E)\tw n,\bM(F)\tw m)[1/e]=0.
$$
\end{prop}
\begin{proof}
We may assume that $n=0$ to simplify the notation. Put $d=\dtr_k(F)$.
Let $X$ be a smooth model of $F$ over $k$. According to \Cref{ex:essentially-smooth}(2)
 and \Cref{lm:compute_HOM_DMgen}, one gets:
$$
\Hom_{\iDMgen(k)}(\bM(E),\bM(F)\tw m)
 = \underset{U \subset X}\lim \Hom_{\pro\iDMgm(k)}(\bM(E),\M(U)\tw m)
$$
where the projective limit runs over the non-empty open subschemes of $X$.
 By applying duality, which holds under our assumptions (see \Cref{num:recall-dual}),
 one gets:
\begin{align*}
\Hom_{\pro\iDMgm(k)}&(\bM(E),\M(U)\tw m)[1/e] \\
 & \simeq \Hom_{\pro\iDMgm(k)}(\bM(E) \otimes \M^c(U),\un(m+d)[m+2d])[1/e]
\end{align*}
Moreover, the motive with compact support $\M^c(U)$ is
 effective, as noted in \Cref{num:recall-dual}.
 Thus, by the cancellation theorem, the left hand-side vanishes
 whenever $m+d<0$ as required.
\end{proof}

\subsection{Rational curves}\label{sec:rat-curve}

\begin{num}
Being (formal) projective limits, generic motives are inherently large objects.
 This is confirmed by the following computation, which we recall from \cite{Deg5}. 
\end{num}
\begin{prop}\label{prop:gmot-rational1}
Let $F/k$ be a function field, and consider
 the canonical inclusion $\varphi:F \rightarrow F(t)$.

Given any closed point $x \in \AA^1_{F,\{0\}}$, with residue field $\kappa_x$, we consider
 the following maps:
\begin{itemize}
\item $\varphi_x:\kappa_x \rightarrow \kappa_x(t)$ the canonical inclusion,
\item $\psi_x:F(t) \rightarrow \kappa_x(t)$ the finite morphism induced by the finite field extension $\kappa_x/F$,
\item  $\partial_x:\bM(\kappa_x)\tw 1 \rightarrow \bM(F(t))$ the residue map (D4*)
 associated to the valuation on $F(t)$ corresponding to the closed point $x \in \AA^1_F$.
\end{itemize}
Then there exists a canonical isomorphism:
$$
\bM(F(t)) \xrightarrow{(\varphi^*,\phi)}  \bM(F) \oplus \pprod{x \in \AA^1_{F,{(0)}}} \bM(\kappa_x)\tw 1
$$
arising from the following split homotopy exact sequence in $\iDMgen(k)$:
$$
0 \rightarrow \pprod{x \in \AA^1_{F,\{0\}}} \bM(\kappa_x)\{1\}
 \xrightarrow{d=\prod_x \partial_x} \bM(F(t)) \xrightarrow{\varphi^*} \bM(F) \rightarrow 0,
$$
where the splitting $\phi=(\phi_x)_{x \in \AA^1_{F,\{0\}}}$ of the map $d$ is
 defined for each point $x$ as follows:
$$
\phi_x:\bM(F(t)) \xrightarrow{\psi_{x!}} \bM(\kappa_x(t))
 \xrightarrow{\gamma_{(t-x)}} \bM(\kappa_x(t))\{1\}
 \xrightarrow{\varphi_x^*} \bM(\kappa_x)\{1\}.
$$
\end{prop}
We refer the reader to \cite[6.1.1]{Deg5} for the proof.

\begin{num}
As a corollary of the preceding proposition, we obtain a canonical splitting
 $s_t:\bM(F) \rightarrow \bM(F(t))$ of the map $\varphi^*:\bM(F(t)) \rightarrow \bM(F)$,
 which is uniquely determined by the following commutative diagram:
$$
\xymatrix@R=14pt@C=34pt{
\bM(F(t))\ar^-{1-d \circ \phi}[r]\ar_-{\varphi^*}[d] & \bM(F(t)). \\
\bM(F)\ar_{s_t}[ru] &
}
$$
Indeed, this splitting holds more generally.
\end{num}
\begin{cor}\label{cor:inclusion-splits}
For any extension of function fields $\varphi:E \rightarrow L$,
 the map $\varphi^*:\bM(L) \rightarrow \bM(E)$ is a split epimorphism
 after $\QQ$-localization.
\end{cor}
\begin{proof}
The case where $L/E$ is finite follows from the degree formula\footnote{This formula is for example
 a consequence of (R2c) and the fact that the norm map associated with $\varphi$ acts as multiplication
 by $d$ on the degree $0$ part of Milnor K-theory}:
$$
\varphi^* \varphi_!=d.Id
$$
where $d$ denotes the degree of $L/E$.

In the general case, $L$ is a finite extension of a purely
 transcendental extension $E(t_1,\ldots,t_n)$. We are reduced to the purely
 transcendental case which holds according to the paragraph preceding the statement.
\end{proof}

The following proposition contradicts the hope of \Cref{num:hope_PHD}:
 the functor $\rho$ of \eqref{eq:morph_gen-mot} cannot be full in general.
\begin{prop}\label{prop:gmot_rational}
Let $F/k$ be a function field and $n$ be an integer.

Then there exists an isomorphism of abelian groups:
$$
\Hom_{\iDMgen(k)}\big(\bM(F),\bM(k(t))\tw n\big)
 \xleftarrow{\ \sim\ } \KM n(F) \oplus
 \prod_{\substack{x \in \AA^1_{k,(0)} \\ y \in \Spec(F \otimes_k \kappa_x)}} \KM{n+1}(\kappa_y)
$$
which to a symbol $\sigma \in \KM n(F)$ associates the morphism:
$$
\bM(F) \xrightarrow{\gamma_\sigma} \bM(F)\tw n
 \xrightarrow{\varphi_F^*} \bM(k)\tw n
 \xrightarrow{s_t} \bM(k(t))\tw n
$$
and to the data of a point $x \in \AA^1_{k,(0)}$,
 of a composite extension field $\kappa_y$ of $F$ and $\kappa_x$
 over $k$
$$
F \xrightarrow{\psi_y} \kappa_y \xleftarrow{\varphi_y} \kappa_x
$$
and a symbol $\tau \in \KM{n+1}(\kappa_y)$ associates the morphism:
$$
\bM(F) \xrightarrow{\psi_{y!}} \bM(\kappa_y) \xrightarrow{\gamma_\tau} \bM(\kappa_y)\tw{n+1} 
 \xrightarrow{\varphi_{y}^*} \bM(\kappa_x)\tw{n+1} \xrightarrow{\partial_x} \bM(k(t))\tw n
$$
\end{prop}
\begin{proof}
This follows from the preceding proposition and duality:
\begin{align*}
&\Hom_{\iDMgen(k)}(\bM(F),\bM(k(t)\tw n) \\
 & \simeq \Hom_{\iDMgen(k)}(\bM(F),\un\tw n)
 \oplus 
 \prod_{x \in \AA^1_{k,(0)}} \underset{=\Hom_{\iDMgen(k)}(\bM(F) \otimes \bM(\kappa_x),\un\tw n)}{\underbrace{\Hom_{\iDMgen(k)}(\bM(F),\bM(\kappa_x)\tw n)}}
\end{align*}
One finally concludes using the isomorphism \eqref{eq:gen_mot&KM} and noting that $\kappa_x/k$ is finite separable:
$$
\Spec(F) \times_k \Spec(\kappa_x)=\Spec(F \otimes_k \kappa_x)=\bigsqcup_{y \in \Spec(F \otimes_k \kappa_x)} \Spec(\kappa_y).
$$
\end{proof}

\begin{rem}\label{rem:gen-gmot}
The preceding proposition shows that the functor $\rho$ is not full,
 negatively answering a question in \cite[9.3.2]{Deg1}.
 This follows from the fact that generic motives are pro-objects.
 In particular, a more accurate formulation of the conjecture describing morphisms
 of generic motives could be formulated as follows.
\end{rem}
\begin{cjt}\label{cjt:gen-gmot}
There exists a natural topology on morphisms of generic motives
 such that morphisms of type (Dn*), for $1 \leq n\leq 4$, are topological
 generators. Relations among these topological generators are described
 by relations (Rm*) as spelled out in \Cref{df:efld}.
\end{cjt}

{\subsection{Curves of positive genus}\label{sec:positive-genus}

\begin{num}\label{num:mot-curves}
Let $K$ be a function field over $k$ of transcendence degree $1$.
 It corresponds to a smooth projective curve $\bar C/k$.
 We fix a closed point $x_0 \in \bar C_{(0)}$,
 and consider the corresponding affine curve $C=\bar C-\{x_0\}$.
 This is a smooth affine model of $K$.

The Jacobian $J$ of $\bar C/k$ is defined as the $k$-scheme
 representing the connected component of the Picard functor
 (see \cite[Th. Def. 27.137]{GW2}).
 On the other hand, the proper pointed $k$-scheme $(\bar C,x_0)$
 admits a unique Albanese scheme $(A,\alpha:\bar C \rightarrow A)$
 such that $\alpha$ maps $x_0$ to the zero element of $A$.
 In fact, letting $J^\vee$ be the dual abelian variety of $J$ over $k$,
 the base point $x_0$ defines a canonical map
 $\bar C \rightarrow J^\vee$, which satisfies the universal property
 of the Albanese scheme of $(\bar C,x_0)$, and we can just put $A=J^\vee$
 (see \cite[Rem. 27.225]{GW2} for a clear account).

One considers the following composite map:
$$
C \xrightarrow j \bar C \xrightarrow \alpha A.
$$
We let $\uA$ be the sheaf of abelian groups represented by $A$ on $\Sm_k$.
 It is $\AA^1$-invariant and admits transfers (see e.g. \cite{Orgo, SS}),
 thus making it an object of $\HI(k)$.
 The preceding map induces a morphism of sheaves with transfers
$$
\Ztr(C) \rightarrow \uA
$$
which gives us a map $\pi_1:\M(C) \rightarrow \uA$ in $\iDMe(k)$.
 We define $\pi_0=p_*:\M(C) \rightarrow \un$ as the map induced by
 the projection.
 
According to a theorem of Suslin and Voevodsky (see \cite[Chap. 5, 3.4.2]{VSF}), the following map in $\iDMe(k)$
\begin{equation}\label{eq:SV-curves}
\M(C) \xrightarrow{(\pi_0,\pi_1)} \un \oplus \underline A
\end{equation}
is an isomorphism. From this, one deduces the following computation of generic motives.
\end{num}
\begin{prop}\label{prop:gmot_tr1}
Consider the notation introduced above,
 and let $\pi_i^K$ be the composition of $\pi_i$ with the canonical map
 $\bM(K) \rightarrow \M(C)$. 

Then we have a homotopy exact sequence
 (see \Cref{num:stable}) in $\pro\iDMgm(k)$:
$$
\pprod{x \in C_{(0)}} \bM(\kappa_x)\tw 1 \xrightarrow{\prod_x \partial_x} \bM(K)
 \xrightarrow{(\pi_0^K,\pi_1^K)} \un \oplus \uA,
$$
where we have denoted by $\partial_x$ the residue map (\Cref{num:D3D4}(D4*))
 associated with the valuation on $K$ corresponding to the closed point $x \in C$.

The boundary map $\partial:\un \oplus \uA \rightarrow \pprod{x \in C_{(0)}} \bM(\kappa_x)(1)[2]$
 of this exact sequence is trivial on the first factor.
 As an element of the abelian group:
$$
\prod_{x \in C_{(0)}} \Hom(\uA,\bM(\kappa_x)(1)[2]) \simeq \prod_{x \in C_{(0)}} \Pic^0(\bar C_{\kappa_x})
$$
it corresponds to the element $(x-x_0)_{x \in C_{(0)}}$ where each $(x-x_0)$ is seen
 as a $0$-cycle of degree $0$ on $\bar C_{\kappa_x}$.
\end{prop}
Therefore, as soon as the field $K$ is not rational over $k$,
 the boundary map $\partial$ is non trivial, and the above homotopy exact sequence is non-split.
 This means in some sense that the motive $\bM(K)$ is ``maximally mixed''.
\begin{proof}
Let $Z \subset C$ be a closed subscheme. As $k$ is perfect, $Z$ is smooth and one deduces
 a homotopy exact sequence in $\iDMgme(k)$:
$$
\M(Z)\tw 1 \xrightarrow{d^Z} \M(C-Z) \xrightarrow{j^Z_*} \M(C) \xrightarrow{i_Z^!} \M(Z)(1)[2]
$$
Note that by additivity $\M(Z)=\prod_{x \in Z} \bM(\kappa_x)$.
 Thus, taking the projective limit of these homotopy exact sequences over the open complement $U=C-Z$
 yields a homotopy exact sequence in $\pro\iDMgme(k)$ which has the required form up to using
 the mentioned additivity and the isomorphism \eqref{eq:SV-curves}.
 The statement regarding the boundary map $\partial$, which is indeed the projective limit of the Gysin morphisms $i_Z^!$,
 thus follows.
\end{proof}

\begin{num}
In \cite{Deg9}, we associate to any homotopy sheaf $F$ over $k$ a homotopy module
 $F_*=\sigma^\infty(F)$ (see also \Cref{num:recall_DM}) with the formula for $n \geq 0$:
\begin{align*}
F_n&=F \otimes \GG^{\otimes,n}=\uH_0(F\tw n) \\
F_{-n}&=\uHom(\GG^{\otimes,n},F)
\end{align*}
where the tensor product and internal Hom are taken with respect to the canonical monoidal structure
 on $\HI(k)$.\footnote{Recall that $F_{-n}$ is also obtained by iterating Voevodsky's $(-1)$-construction,
 recalled in \eqref{eq:(-1)-construction}. See \cite[1.13]{Deg9} for details.}

This construction can be described more precisely when $F=\uA$ is the homotopy sheaf associated with an abelian variety (see \Cref{num:mot-curves}). Let us recall this description due to Kahn and Yamazaki.
\end{num}
\begin{thm}\label{thm:abelian-var-HI}
Let $A$ be an abelian variety over $k$, and $\uA_*$ be the associated homotopy module as defined above.

Then for any $n>0$ and any function field $E$, one has:
\begin{align*}
\uA_{-n}(E)&=0 \\
\uA_n(E)&=\KSM n(E;A)
\end{align*}
where $\KSM n(E;A)=K(E,A,\GG,\hdots,\GG)$ is the Somekawa K-theory associated with the semi-abelian variety
 $A \times \GG^n$ (see \cite{Somekawa}, or rather \cite[Def. 5.1]{KY} for the correct sign).
\end{thm}
\begin{proof}
For the first isomorphism, we only need to prove that $\uA_{-1}=0$.
 According to \Cref{cor:gmot-pts},
 it is only necessary to check that for any function field $E/k$, one has $\uA_{-1}(E)=0$.
 This follows from the exact sequence \eqref{eq:(-1)-construction}
 and the classical fact that the projection map $p$ induces an isomorphism
 $A(E) \xrightarrow{p^*} A(\GGx E)$, since $A$ is an abelian $k$-scheme.

The second computation follows from the main theorem of \cite{KY},
 stated as (1.1) in the introduction.
\end{proof}

\begin{rem}
The vanishing $\uA_{-1}=0$ is equivalent to say that the sheaf $\uA$ is a ``birational
 invariant'': we refer the interested reader to \cite{KS} for more on this notion
 (see \cite[4.1]{KS} for the case of abelian schemes).
\end{rem}

We can now state the following corollary of \Cref{prop:gmot_tr1}.
\begin{cor}
Consider the notations of the aforementioned proposition,
 and let $E$ be an arbitrary function field.
 Then for any integer $n \in \ZZ$, there exists a short exact sequence
 of abelian groups:
$$
\prod_{\substack{x \in C_{(0)} \\ y \in \Spec(E \otimes_k \kappa_x)}} \KM{n+1}(\kappa_y)
 \rightarrow \Hom_{\iDMgen(k)}\big(\bM(E),\bM(K)\tw n\big)
 \rightarrow \KM n (E) \oplus \KSM n (E;A) \rightarrow 0
$$
where the first map sends a symbol $\tau \in \KM{n+1}(\kappa_y)$
 to the morphism
$$
\bM(E) \xrightarrow{\psi_{y!}} \bM(\kappa_y) \xrightarrow{\gamma_\tau} \bM(\kappa_y)\tw{n+1} 
 \xrightarrow{\varphi_{y}^*} \bM(\kappa_x)\tw{n+1} \xrightarrow{\partial_x} \bM(K)\tw n
$$
given the canonical morphisms $E \xrightarrow{\psi_y} \kappa_y \xleftarrow{\varphi_y} \kappa_x$,
 and the residue map $\partial_x$ associated with the valuation on $K$ corresponding
 to the point $x$ as defined in \Cref{num:D3D4}(D4*).

In particular, when $n=-1$, one gets an isomorphism:
$$
\Hom_{\iDMgen(k)}\big(\bM(E)\tw 1,\bM(K)\big) \simeq \prod_{x \in C_{F,(0)}} \ZZ.
$$
\end{cor}
\begin{proof}
One applies the functor $\Hom_{\pro\iDMgm(k)}(\bM(E),-)$ to the homotopy exact sequence
 of \Cref{prop:gmot_tr1}, and proceeds as in the proof of \Cref{prop:gmot_rational}
 for the first term of the above exact sequence, we use the preceding proposition
 to compute the third one, and then apply the vanishing 
 of motivic cohomology $\Hmot^{n+1,n}(L)$ for any field $L$ for the surjectivity assertion.
\end{proof}

We end this section by a remark about duality.
\begin{lm}\label{lm:dual-abvar}
Consider the notation of \Cref{num:mot-curves}.
 Then the motive $\uA$ is rigid in $\iDM(k)$ and there exists a canonical isomorphism:
$$
\uA^\vee \simeq \uA(-1)[-2].
$$
\end{lm}
\begin{proof}
According to \Cref{num:recall-dual}, the motive $\M(C)$ is rigid, with dual given by $\M^c(C)(-1)[-2]$.
 The homotopy exact sequence of the motive with compact support
$$
\M^c(\{x_0\}) \rightarrow \M^c(\bar C) \rightarrow \M^c(C)
$$
splits, giving a canonical isomorphism $\M^c(C)=\coKer(\un \rightarrow \M(C))=\uA \oplus \un(1)[2]$.
 Thus, according to \eqref{eq:SV-curves}, one gets an isomorphism
$$
\un \oplus \uA^\vee \simeq \uA(-1)[-2] \oplus \un.
$$
One deduces the desired isomorphism as the factor $\un$ is eliminated by the projection
 $p:C \rightarrow \spec k$.
\end{proof}

\begin{rem}
Recall that $\un(1)[1]=\GG$. Therefore, the preceding formula reflects the duality of abelian varieties.
 Note in particular that, as $A$ is the Albanese variety of a pointed smooth projective curve,
 it admits a polarization which explains the preceding formula.
 For a general abelian variety $A$, one would expect $\uA^\vee=\underline{(A^*)}(-1)[-2]$
 where $A^*$ is the dual abelian variety.
 This formula was proved for \'etale motives (and consequently for $\DM(k,\QQ)$)
 by Kahn and Barbieri-Viale in \cite[4.5.3]{KBV}.
\end{rem}

\subsection{Higher dimension}\label{sec:surfaces}

Given the preliminary sections, the reader will not be surprised that generic motives
 are gigantic in general.
 We can at least illustrate this assertion by the following result which partially extends \Cref{prop:gmot_rational}.
\begin{prop}\label{prop:gmot-higher-trdeg}
Let $\varphi:K \rightarrow L$ be an extension of function fields over $k$ of transcendence degree $n>0$.
 For simplicity, we consider the pro-motive $T=\pprod{x \in \AA^1_{K,(0)}} \bM(\kappa_x)\tw 1$.
 
Fix a transcendental basis $(t_1,...,t_n)$ of $L/K$,
 and let $K_i=K(t_1,...,t_i) \subset L$. We consider the finite morphism
 $\psi:K_n \rightarrow L$ and for any $i \leq j$, we let $\varphi_{i,j}:K_i \rightarrow K_j$
 be the obvious inclusion.
 
 Then the following map is a split epimorphism:
$$
\bM(L) \xrightarrow{(\varphi^*,\pi_i^n \circ \psi^*)} \bM(K) \oplus \bigoplus_{i=1}^n T\tw i
$$
where the projector $\pi_i^n:\bM(K_n) \rightarrow T\{i\}$ is defined inductively using the formulas:
\begin{align*}
\pi_1^n&=\phi \circ \varphi_{1,n}^* \\
i>1: \pi_i^n&=\pi_{i-1}^{n-1} \circ \partial_n
\end{align*}
where $\phi:\bM(K_1) \rightarrow T\tw 1$ is the projection map from \Cref{prop:gmot-rational1},
 and $\partial_n:\bM(K_n) \rightarrow \bM(K_{n-1})\tw 1$ is the residue map
 arising from the valuation on $K_n$ associated with the variable $t_{n}$.
\end{prop}
\begin{proof}
Using \Cref{cor:inclusion-splits} (or the degree formula directly), $\psi^*$ is a split epimorphism,
 so that we are reduced to the purely transcendental case $L=K_n$.
 Then it follows by induction on $n>0$: the case $n=1$ is precisely \Cref{prop:gmot-rational1}.
 Then we get a composite epimorphism:
\begin{align*}
\bM(K_n) &\rightarrow \bM(K_{n-1}) \oplus \bM(K_{n-1})\tw 1 & \text{\Cref{prop:gmot-rational1} for $K_n/K_{n-1}$} \\
 &\rightarrow \bM(K_1) \oplus  \bM(K_{n-1})\tw 1 & \text{\Cref{cor:inclusion-splits}} \\
 &\rightarrow \bM(K) \oplus T\tw 1 \oplus \bM(K_{n-1})\tw 1 & \text{\Cref{prop:gmot-rational1}}
\end{align*}
and we conclude by induction using the $(n-1)$-case (up to discarding the first factor).
 The formulas for the projectors follow from the induction.
\end{proof}

\begin{rem}
The preceding proposition provides only a preliminary decomposition of the various pro-motives contained within a given generic motive.
Despite the analysis in the preceding section, the formula in \Cref{prop:gmot-rational1} is not sufficient to explicitly compute
 the structure of the motive of a purely transcendental extension of degree greater than one.
Nevertheless, it reveals that generic motives of transcendence degree $n$ typically form 
non-trivial extensions of lower-dimensional motives.
In particular, these motives are not mere direct sums or simple products of lower-dimensional components but involve intricate extension structures,
 reflecting the increasing complexity of their cohomological and motivic information.
\end{rem}

\begin{num}\textit{Surfaces}. We now assume that $k$ is algebraically closed.
 We will illustrate the preceding remark in transcendence degree $2$.
 Let $S$ be a geometrically connected smooth projective surface over $k$.
 We let $K$ be its function field.

Recall that Murre has shown in \cite{Murre1} the existence of a Chow-K\"unneth decomposition
 of the cohomological motive of $S$:
$$
\h(S)=\un \oplus \h^1(S) \oplus \h^2(S) \oplus \h^3(S) \oplus \un(-2)[-4],
$$
which we will consider in the homotopy category $\DM(k,\QQ)$.
Moreover, according to \cite{KMP},
 this decomposition can be refined: the component $\h^2(S)$ can be split into:
$$
\h^2(S)=\h^2_{alg}(S) \oplus \h^2_{tr}(S),
$$
representing the algebraic and transcendental parts of $\h^2$.
 Note that by definition, the algebraic part is of the form:
$\h^2_{alg}=\rho.\un(-1)[-2]$ where $\rho$ is the Picard number of the surface $S$. 
 Beware that in \emph{loc. cit.}, the motive $\h^2_{tr}(S)$ is denoted by $t_2(S)$.
 Recall that it is shown in Corollary 14.8.11 of \emph{loc. cit.}
 that the functor $S \mapsto \h^2_{tr}(S)$ is a birational invariant of $S$.\footnote{As recalled to me by
 Jan Nagel, this property can also be shown directly by using the blow-up formula
 for Chow motives combined with the fact any birational morphism of surfaces in
 characteristic $0$ can be written as a sequence of blow-ups of closed points.}
 In particular, it depends only on the function field $K$ of $S$. 

Dually, we obtain a decomposition of the homological motive in the homotopy category $\DM(k,\QQ)$:
$$
\M(S)=\un \oplus \M_1(S) \oplus \M_2^{alg}(S) \oplus \M_2^{tr}(S) \oplus \M_3(S) \oplus \un(2)[4].
$$
If $A$ denotes the Albanese scheme associated with $S$, we have an identification $\M_1(S)=\uA$,
 and by duality: $\M_3(S)=\uA^\vee(2)[4] \simeq \uA(1)[2]$ (apply \Cref{lm:dual-abvar}).
 Let us summarize the above discussion.
\end{num}
\begin{lm}
Consider the above notation.
 Then there exists a direct factor $\M_2^{tr}(S) \subset \M(S)$ in $\DMgm(k,\QQ)$
 called the \emph{transcendental part} of the homological motive, and a (refined) Chow-K\"unneth
 decomposition in $\DMgm(k,\QQ)$:
$$
\M(S)=\un \oplus \uA \oplus \rho.\un(1)[2] \oplus \M_2^{tr}(S) \oplus \uA(1)[2] \oplus \un(2)[4].
$$
\end{lm}

\begin{num}
To go further, we will need to use the coniveau filtration viewed internally
 within pro-motives, as in \cite{Deg11} or \cite{PawD}.\footnote{We follow the notation
 of the latter which seems better suited.}
 We continue to consider a geometrically connected smooth projective surface $S$ over $k$,
 with function field $K$.

Recall that a flag in $S$ consists of a decreasing sequence $(Z^p)_{p \geq 0}$ of closed subschemes
 of $S$ such that $\codim_S(Z^p)\geq p$.
 The set $\Flag(S)$ of flags in $S$, ordered by term-wise inclusion, is filtered.
 We define an increasing filtration of $S$ by pro-schemes,
 the \emph{coniveau filtration}, by setting:
$$
S^{\leq p}:=\plim{Z^* \in \Flag(S)} \big(S-Z^{p+1}\big).
$$
By applying the functor $\M$ component-wise,
 we then get a canonical homotopy exact sequence in $\pro\iDMgm(k)$:
\begin{equation}\label{eq:coniv-exact-cpl}
\M(S^{\leq p-1}) \rightarrow \M(S^{\leq p}) \rightarrow \M(S^{=p})
\end{equation}
where the last term stands for the homotopy cofiber. Moreover, it follows
 from \cite[Lem. 1.14]{Deg11} that there exists a canonical isomorphism:
$$
\M(S^{=p}) \simeq \pprod{x \in S^{(p)}} \bM(\kappa_x)(p)[2p].
$$
We can then extend the computation of \Cref{prop:gmot_tr1}
 in the transcendence 2 case as follows.
\end{num}
\begin{prop}\label{prop:gen-mot-surf}
Consider the above notation.
 Then there exists homotopy exact sequences in the stable $\infty$-category $\pro\iDMgm(k,\QQ)$:
\begin{gather*}
\pprod{x \in S^{(1)}} \bM(\kappa_x)\tw 1 \xrightarrow{\prod_x \partial_x} \bM(K) \rightarrow \bM(S^{\leq 1}) \\
\pprod{s \in S^{(2)}} \bM(\kappa_s)\tw 2 \rightarrow \bM(S^{\leq 1}) \rightarrow \un \oplus \uA \oplus \rho.\un(1)[2] \oplus \M_2^{tr}(S) \oplus \uA(1)[2] \oplus \un(2)[4]
\end{gather*}
where the $\partial_x$ is the residue map (see \Cref{num:D3D4}(D4*))
 associated with the valuation on $K$ corresponding to the divisor with generic point $x \in S^{(1)}$.
\end{prop}
In particular, the motive of the transcendental degree $2$ function field $K/k$ can be computed in terms
 of two successive extensions
 of generic motives associated with fields of lower transcendental degrees on the one hand,
 and of pure motives, including the transcendental motive of $S$, on the other hand.
\begin{proof}
This follows immediately from the exact sequence \eqref{eq:coniv-exact-cpl} and the preceding discussion,
 once one notices that $\M(S^{\leq 2})=\M(S)$ and $\M(S^{\leq 0})=\M(S^{=0})=\bM(K)$.
\end{proof} 
\section{Motivic cohomology of fields}\label{sec:cohm-fields-main}

\begin{num}
The structure of the generic motive of a function field $K/k$ is reflected in motivic cohomology,
 as follows from the formula:
\begin{equation}\label{eq:cohm-gmot}
\Hmot^{n,i}(K)=\Hom_{\pro\iDM(k)}(\bM(K),\un(i)[n]).
\end{equation}
(Apply the continuity property \eqref{eq:continuity}.)
 We will illustrate this fact in this section.

The latter group has several explicit presentations. For the non-trivial case $i>0$, we consider the non-positively graded homological complex
 made of free abelian groups
$$
Z_0(\GG^{\wedge,i})=c_K(\Delta^0,\GG^{\wedge,i}) \xrightarrow{\ d_0\ } \cdots \xrightarrow{d_{n-1}} c_K(\Delta^n,\GG^{\wedge,i}) \xrightarrow{d_n} c_K(\Delta^{n+1},\GG^{\wedge,i}) \rightarrow \hdots
$$
where all schemes and products are taken implicitly over the field $K$,
 $\GG^{\wedge,i}$ is the $i$-th smash product of the pointed scheme $(\GG,1)$ (see also \Cref{num:cohm-integral} for this notation).
 Thus the group $c_K(\Delta^n,\GG^{\wedge,i})$ is the free abelian group generated
 by integral closed subschemes $Z \subset \Delta^n \times \GG^{\wedge,i}$ which are finite dominant over $\Delta^n_K$ and which is not contained
 in a closed subscheme of the form $\Delta^n \times (\GG^{\wedge,j} \times \{1\} \times \GG^{\wedge,i-j-1})$.
 The differentials are given by the alternating sum of the operation of taking intersection along faces $\Delta^{n-1}_i \times \GG^{\wedge,n}$,
 operations which are always well-defined and can be computed using Serre's Tor formula.

Thus, the motivic cohomology of the field $K$ can be expressed as follows
\begin{equation}\label{eq:cohm-presentation}
\Hmot^{n,i}=\Ker(d_{i-n})/\Img(d_{i-n-1}).
\end{equation}
As both abelian groups $\Ker(d_{i-n})$ and $\Img(d_{i-n-1})$ are free,
 this is therefore an \emph{explicit presentation} of motivic cohomology.
 While this presentation may not be practical for explicit computations,
 it provides a bound for the size of motivic cohomology of fields
 (see \Cref{prop:cohm-uncount} for more).\footnote{One gets a similar presentation by considering Bloch's higher Chow groups.}
\end{num}

\subsection{Borel's computations and $\lambda$-structures}\label{sec:Borel}

In this subsection we recall the following fundamental computation of motivic cohomology. This is
 mainly due to Borel if one neglects weights, which were determined by Beilinson using his
 equally fundamental theory of (higher) regulators. Beilinson's main argument was explained and corrected
 in detail by Burgos in \cite{Burgos}. We use the notation of \Cref{sec:rat-mot-coh} for rational motivic cohomology.
\begin{thm}\label{thm:BorelB}
Let $K$ be a number field, and let $r_1$ and $r_2$ denote the number of real and complex embeddings of $K$, respectively.
 Then for any pair of integers $(n,i)$, one has:
$$
\HB^{n,i}(K)=\begin{cases}
\QQ & n=i=0, \\
K^* \otimes_\ZZ \QQ & n=i=1, \\
\QQ^{r_2} & n=1, i>1, i \text{ even}, \\
\QQ^{r_1+r_2} & n=1, i>1, i \text{ odd}, \\
0 & \text{otherwise.}
\end{cases}
$$
If $\cO_K$ is the ring of integers of $K$, one gets:
$$
\HB^{n,i}(\cO_K)=\begin{cases}
\QQ & n=i=0, \\
\cO_K^* \otimes_\ZZ \QQ & n=i=1, \\
\QQ^{r_2} & n=1, i>1, i \text{ even}, \\
\QQ^{r_1+r_2} & n=1, i>1, i \text{ odd}, \\
\Pic(\cO_K)_\QQ & n=2, i=1, \\
0 & \text{otherwise.}
\end{cases}
$$
\end{thm}
We will give a proof of this theorem, which is essentially a reduction to Borel's computation
 of the stable real cohomology of arithmetic groups (\cite{Borel}).
 For clarity, we first recall the arguments in low cohomological degrees,
 which are now well-established results.

The computations for $i \leq 1$ follow from \eqref{eq:cohm1}.
 The same holds for the vanishing $\HB^{n,i}(K)$
 (resp. $\HB^{n,i}(\cO_K)$) for $n-i>0$ (resp. $n-i>1$).
  We can consider the coniveau spectral sequences for the $1$-dimensional regular scheme $X=\Spec(\cO_K)$
 with coefficients in $n$-th twisted rational motivic cohomology and rational K-theory:\footnote{The
 second one is a direct sum of copies of the first one, but it is convenient
 to use both forms.}
\begin{align*}
E_{1,\mathcyr B}^{p,q}=\oplus_{x \in X^{(p)}} \HB^{q-p,n-p}(\kappa_x) &\Rightarrow \HB^{p+q,n}(\cO_K), \\
E_{1,K}^{p,q}=\oplus_{x \in X^{(p)}} K_{-p-q}(\kappa_x)_\QQ &\Rightarrow K_{-p-q}(\cO_K).
\end{align*}
According to Quillen's computation of the K-theory of a finite field $\FF_q$ (see \cite{QuiFF}),
 one gets:
\begin{align*}
\HB^{n,i}(\FF_q)=0 & \Leftrightarrow (n,i) \neq (0,0), \\
K_n(\FF_q)_\QQ=0 & \Leftrightarrow n \neq 0.
\end{align*}
Therefore, the above spectral sequence implies:
\begin{align*}
\HB^{n,i}(\cO_K) \xrightarrow \sim \HB^{n,i}(K) & \text{ if } i\geq 2, \\
K_n(\cO_K)_\QQ \rightarrow K_n(K)_\QQ & \text{ is } 
\begin{cases}
mono  & n=1 \\
iso & n \geq 2.
\end{cases}
\end{align*}
Note in particular that $\HB^{3,2}(\cO_K)$ must be zero according to the first isomorphism.
 As a summary, we have obtained the lower $K$-groups as follows:
\begin{align*}
K_0(K)_\QQ&=\HB^{0,0}(K)=\QQ \\
  K_0(\cO_K)_\QQ&=\HB^{0,0}(\cO_K) \oplus \HB^{2,1}(\cO_K)=\QQ \oplus \Pic(\cO_K)_\QQ \\
K_1(K)_\QQ&=\HB^{1,1}(K)=K^\times \otimes_\ZZ \QQ \\
  K_1(\cO_K)_\QQ&=\HB^{1,1}(\cO_K)=\cO_K^\times \otimes_\ZZ \QQ
\end{align*}
Thus, it remains to compute $K_n(\cO_K)_\QQ$ for $n \geq 2$,
 and to determine its graduation for the $\gamma$-filtration to get
 the above calculation.
 As already mentioned, the first task was achieved by Borel
 in his famous paper \cite{Borel}. The second task follows from the comparison of the Borel
 and the Beilinson regulators, as detailed in \cite{Burgos}
 (see also \cite{Rapop} for an earlier discussion).
 In the remainder of this section, we will explicitly recall the computations of Borel
 and show how to give a direct argument for the determination of the $\gamma$-filtration.

\begin{num}\textit{Rational K-theory and homology of groups}.\footnote{These
 facts are well-known, but we recall them for completeness,
 and for introducing our notation.}
Given any commutative ring $A$ and an integer $n>1$, one gets the following
 classical computation:
$$
K_n(A) \simeq \pi_n\big(\BGL(A)^+\big) \simeq \pi_n\big(\BSL(A)^+\big)
$$
where $\GL(A)$ (resp. $\SL(A)$) is the infinite general (resp. special)
 linear group associated with $A$, and $X^+$ is Quillen
 $+$-construction applied to a space $X$.\footnote{See \Cref{ex:presentable-htp}
 for a quick recall in terms of $\infty$-categories.}
The first isomorphism is Quillen's comparison of the plus and Q constructions
 (\cite{GraysonPQ}, \cite{Godfard} with more details).
 To get the second isomorphism, one considers the obvious exact sequence of
 groups:
$$
1 \rightarrow \SL(A) \rightarrow \GL(A)
 \xrightarrow{\det} A^\times \rightarrow 1.
$$
After applying the classifying space construction,
 followed by the plus construction, 
 one gets a fibration sequence in the homotopy category:\footnote{Note
 that the $+$-construction respects fiber sequences,
 or equivalently homotopy fibers, as it is a localization functor.}
$$
\BSL(A)^+ \rightarrow \BGL(A)^+ \rightarrow \mathrm B(A^\times)^+.
$$
Note that $B(A^\times)^+=B(A^\times)$ since $A^\times$ is an abelian group.
 Thus the associated exact sequence of homotopy groups allows us to conclude.

Next we recall that $\BSL(A)^+$ is an H-space.\footnote{In fact,
 the ``+=Q'' theorem of Quillen shows that it is a loop space.}
 Consequently, the Milnor-Moore/Cartan-Serre theorem implies that the Hurewicz
 map
$$
\pi_n(\BSL(A)^+) \rightarrow H_n^{\text{sing}}(\BSL(A),\ZZ)
$$
induces an isomorphism after tensoring with $\QQ$
\begin{equation}\label{eq:rational_Kth&hlg}
K_n(A)_\QQ \xrightarrow \sim \Prim_n H_*(\SL(A),\QQ)
\end{equation}
where $\Prim_n$ denotes the primitive part in degree $n$ of the Hopf algebra:
$$
H_*(SL(A),\QQ)=H_*^{\text{sing}}(\BSL(A),\QQ)\simeq H_*^{\text{sing}}(\BSL(A)^+,\QQ).
$$
\end{num}

\begin{num}
We next fix an integer $r>0$ and consider the arithmetic group $\Gamma_r=\SL_r(\cO_K)$,
 which can be seen as a subgroup of the real Lie group\footnote{The latter can also be seen as the real points
 of the algebraic $\QQ$-group scheme obtained by the Weil restriction of scalars of the $K$-group scheme $\SL_{r,K}$.}
$$
G_r=\SL_r(K \otimes_\QQ \RR)=\SL_r(\RR)^{r_1} \times \SL_r(\CC)^{r_2}.
$$
We denote by $\nu_r:\Gamma_r \rightarrow G_r$ the natural inclusion, where $\Gamma_r$ is seen with the discrete topology.
 Viewing $\RR$ as a discrete module over $G_r$ (resp. $\Gamma_r$),
 one can define the continuous cohomology $H^*_{ct}(-,\RR)$
 with coefficient in $\RR$ (see e.g. \cite[IX, \textsection 1]{BW}).
 As an application of the main result of \cite{Borel}, one gets the following fundamental computation.
\end{num}
\begin{thm}\label{thm:Borel}
The restriction map along $\nu_r$ in continuous cohomology induces an isomorphism for any $n \leq r/4$:
$$
H^n_{ct}(G_r,\RR) \xrightarrow{\nu_r^*} H^n_{ct}(\Gamma_r,\RR)=H^n(\Gamma_r,\RR).
$$
\end{thm}
\begin{proof}
As the proof is very involved, we indicate how to deduce the proof of the above theorem
 from \cite{Borel}, Theorem 7.5.

One considers the maximal compact open subgroup $K_r=\SO_r^{r_1} \times \SU_r^{r_2}$
 of $G_r$ and the associated symmetric space $X_r=K_r\backslash G_r$.
 The group $\Gamma_r$ acts properly discontinuously and almost freely on $X_r$ so that
 the topological quotient space $O_r=X_r/\Gamma_r$ is an orbifold (see e.g. \cite[Chap. 13]{Thurs}).
Moreover, almost by construction of an orbifold, the natural map
 $\B \Gamma_r:=X_r\sslash \Gamma_r \rightarrow O_r$ between the homotopy quotient to the 
 orbifold quotient is a $\QQ$-equivalence, in the sense of \Cref{ex:presentable-htp}.\footnote{In the original argument
 given by Borel, one applies Selberg's lemma (see \cite{Alp}) to get a finite index normal torsion-free subgroup
 $\Gamma'_r \subset \Gamma _r$.
 Then $\Gamma'_r$ acts freely on, $X_r$ so that $X_r/\Gamma'_r$ is a manifold and a model for $\B \Gamma'_r$.
 But now the canonical map $X_r/\Gamma'_r \rightarrow X_r/\Gamma_r=O_r$ is a finite branched cover, and therefore a $\QQ$-equivalence.
 Moreover, the canonical map $B\Gamma'_r \rightarrow B\Gamma_r$ is also a $\QQ$-equivalence
 (see e.g. \cite[Chap. III.7]{Brown}). One deduces by composition a natural $\QQ$-equivalence $B\Gamma_r \simeq O_r$.} \\
 The de Rham isomorphism theorem for orbifolds (see \cite{Satake}) implies that one can compute the singular cohomology of $O_r$ with real coefficients,
 and therefore that of $\B \Gamma_r$,
 as the cohomology of the differential graded $\RR$-algebra $\Omega^*(O_r) \simeq \Omega^*(X_r)^{\Gamma_r}$.
 On the other hand, according to the topological version of the Poincaré lemma, the continuous cohomology of $G_r$ 
 can be identified with the cohomology of the $\RR$-dga $\Omega^*(X_r)^{G_r}$. The Matsushima map considered
 in \cite[\textsection 3.1]{Borel} is nothing else than the inclusion map of real differential graded algebras
$$
j_{\Gamma_r}:\Omega^*(X_r)^{G_r} \rightarrow \Omega^*(X_r)^{\Gamma_r}.
$$
The latter is therefore a quasi-isomorphism in degree less than $r/4$ according to the conjunction of Th. 7.5
 and \textsection 9.5, formula (3) of \cite{Borel}.
 Finally, it is straightforward to verify that $j_{\Gamma_r}$ models the restriction map $\nu_r^*$ on continuous cohomology.
\end{proof}

\begin{cor}
Consider the inclusion of topological groups
$$
\nu:\SL(\cO_K) \rightarrow \SL(K \otimes_\QQ \RR).
$$
Then the pullback along $\nu$ induces an isomorphism:
$$
H^*_{ct}\big(\SL(K \otimes_\QQ \RR),\RR\big) \xrightarrow{\nu^*} H^*(\SL(\cO_K),\RR)
$$
of graded $\RR$-algebras, whose graded pieces are all finite dimensional over $\RR$.
\end{cor}
Indeed, for any $n \leq r/4$, $H^n_{ct}(G_r,\RR)$ has finite dimension over $\RR$,
 and does not depend on $r$.\footnote{This follows from the van Est isomorphism, which allows to identify
 the latter cohomology groups with the singular cohomology group $H^n_{\text{sing}}(G_r^u,\RR)$
 of the compact real form $G_r^u$ associated with the semi-simple real Lie group $G_r$.
 See below.}
 It follows that both the source and the target of $\nu_r^*$ satisfy the Mittag-Leffler
 condition, implying that $\nu^*$ can be identified with the projective limit
 of the $\nu_r^*$. Thus the preceding theorem allows us to conclude.

\begin{num}
Recall that according to the universal coefficient theorem, one gets a canonical isomorphism:
$$
H^n(\SL(\cO_K),\RR)=H_n(\SL(\cO_K),\RR)^\vee
$$
where $\vee$ denotes the dual of $\RR$-vector spaces.
 Since this $\RR$-vector space is finite dimensional according to the preceding corollary,
 one deduces a canonical isomorphism
$$
H_n(\SL(\cO_K),\RR) \rightarrow  H_n(\SL(\cO_K),\RR)^{\vee\vee} \simeq H^n(\SL(\cO_K),\RR)^\vee.
$$
We will define the following $\RR$-valued continuous homology\footnote{Due to finite dimensionality, it coincides
 with the other possible definition of continuous homology via derived functors.}:
$$
H_n^{ct}\big(\SL(K \otimes_\QQ \RR),\RR\big):=H^n_{ct}\big(\SL(K \otimes_\QQ \RR),\RR\big)^\vee
$$
Therefore, the preceding corollary can be formulated by saying that the restriction map $\nu$ induces
 an isomorphism of $\RR$-coalgebras:
$$
H_*(\SL(\cO_K),\RR) \xrightarrow{\nu_*} H_*^{ct}\big(\SL(K \otimes_\QQ \RR),\RR\big).
$$
Therefore, according to \eqref{eq:rational_Kth&hlg}, one gets the following computation:
\end{num}
\begin{cor}
Consider the above notation.
 Then for any integer $n \geq 2$, the Hurewicz map, composed with the above isomorphism $\nu_*$,
 induces an isomorphism of $\RR$-vector spaces:
$$
K_n(\cO_K) \otimes_\QQ \RR \xrightarrow \sim \Prim_n H_*^{ct}\big(\SL(K \otimes_\QQ \RR),\RR\big).
$$ 
\end{cor}
This is a nice formula, which already allows us to recover the rank of K-theory.
 However, we do not know how to determine the gamma filtration on the right-hand side directly,
 though it exists naturally (see below). Therefore, we must return to the known strategy
 and use the compact form of the ind-Lie group $\SL(K \otimes_\QQ \RR)$.

\begin{num}
We now recall the necessary computations from Lie groups and Lie algebras.
 With the ind-Lie group  $G=\SL(K \otimes_\QQ \RR)=\SL(\RR)^{r_1} \times \SL(\CC)^{r_2}$,
 one associates using Cartan duality the following objects:
\begin{center}
\begin{tabular}{|l|l|l|}
\hline
\emph{ind-object of}
&non ind-compact type & ind-compact dual \\
\hline
Lie groups & 
$G=\SL(\RR)^{r_1} \times \SL(\CC)^{r_2}$
 & $G^u=\SU^{r_1} \times (\SU \times \SU)^{r_2}$ \\
\hline
max. compact
 & $K=\SO^{r_1} \times \SU^{r_2}$ & \text{same} \\
\hline
sym. spaces
 & $X=(\SL(\RR)/\SO)^{r_1} \times (\SL(\CC)/\SU)^{r_2}$
 & $X^u=(\SU/\SO)^{r_1} \times \SU^{r_2}$ \\
\hline
\end{tabular}
\end{center}
The duality is defined via the associated ind-Lie algebras:
 starting with the Cartan decomposition $\mathfrak g=\mathfrak k \oplus \mathfrak p$
 associated with $G$ (and its Killing form),
 we consider the \emph{compact form} of $\mathfrak g$ as the Lie algebra
 $\mathfrak g^u=\mathfrak k \oplus \mathfrak p^u$ where
 $\mathfrak p^u=i.\mathfrak p \subset \mathfrak g \otimes_\RR \CC$.
 Then $\mathfrak g^u$ is the ind-Lie algebra of $G^u$.

As explained in the proof of \Cref{thm:Borel},
 the continuous cohomology of $G$ can be identified with the cohomology
 of the differential graded $\RR$-algebra $\Omega^*(X)^G$ of $G$-invariant
 differential forms on $X$.\footnote{Recall that the differentials of this complex
 are trivial, as the $G$-invariant differential forms are in fact harmonic.}
 Further, the classical Lie cohomology theory translates into the following composite
 of isomorphisms or homotopy equivalence:
$$
\Omega^*(X)^G \underset{(1)}{\xrightarrow \simeq} C^*(\mathfrak g,\mathfrak k;\RR)
 \underset{(2)}{\xrightarrow \simeq} C^*(\mathfrak g^u,\mathfrak k;\RR)
 \underset{(3)}{\xleftarrow \simeq} \Omega^*(X^u)^{G^u} 
 \underset{(4)}{\xrightarrow{\sim}} \Omega^*(X^u)
$$
where (1) and (3) are the isomorphisms obtained by evaluating
 a differential form on $X$ or $X_u$ at the base point (class of $\Id$),
 (2) is the isomorphism which, in degree $n$, is multiplication with $i^n$,
 and (4) is a homotopy equivalence --- the retraction of the obvious inclusion --- given by integration
 $\omega \mapsto \int_{G^u} g^*\omega dg$ where $dg$ is the Haar measure on $G^u$.

Using the de Rham isomorphism for the ind-smooth variety $X^u$, we finally
 get a composite isomorphism:
$$
\varphi:H^*_{ct}(G,\RR) \underset{(1234)}{\xrightarrow{\ \sim\ }} H^*_{dR}(X^u) \xrightarrow \sim H^*_{\text{sing}}(X^u,\RR).
$$
Let us pause the exposition with the following definition.
\end{num}
\begin{df}
For any integer $n \geq 2$, we have obtained an isomorphism:
$$
\nu_{\text{Bo}}:K_n(\cO_K) \otimes_\ZZ \RR \xrightarrow{\nu_*} \Prim_n H_*^{ct}\big(\SL(K \otimes_\QQ \RR),\RR\big)
 \xrightarrow{\varphi^\vee} \Prim_n H_*^{\text{sing}}\big((\SU/\SO)^{r_1} \times \SU^{r_2},\RR\big)
$$
called the \emph{Borel isomorphism} associated with the number field $K$ (in degree $n$).
\end{df}
Note that the Borel regulator is then obtained by the following composite map:
\begin{equation}\label{eq:Borel-regulator}
\rho_{\text{Bo}}:K_n(\cO_K)_\QQ \rightarrow K_n(\cO_K) \otimes_\ZZ \RR \xrightarrow{\nu_{\text{Bo}}} \Prim_n H_*^{\text{sing}}\big((\SU/\SO)^{r_1} \times \SU^{r_2},\RR\big).
\end{equation}

\begin{rem}
\begin{enumerate}
\item Beware that Burgos has explained with remarkable precision in \cite[\textsection 9.5]{Burgos}
 that one needs to renormalize
 the above regulator to compare it with the Beilinson regulator
 (up to multiplying by $2$ on the factors corresponding to the $r_1$ real places).
 In our approach, this problem is irrelevant.
\item The only non-canonical isomorphism involved in the Borel regulator
 is the isomorphism (2) mentioned above. However, it is noted that it is degree-wise
 a multiplication by a scalar within the complexified ind-Lie algebra $\mathfrak g_\CC$.
\end{enumerate}
\end{rem}

\begin{num}
We are now in a position to elucidate the behavior of the Borel regulator
 with respect to the canonical $\lambda$-structure on the left-hand side,
 and an appropriate computable $\lambda$-structure on the right-hand side.

To define the $\lambda$-operations, we use the method of Soul\'e in \cite[\textsection 1]{Soule}. 
 In our case, this leads us to consider the exterior power of matrices.
 After taking the obvious base on $\Lambda^r(\ZZ^n)$,
 we deduce a morphism of schemes
$$
\Lambda^r_n:\SL_n \rightarrow \SL_{\binom n r}
$$
and as $n$ goes to $\infty$, a morphism of ind-group schemes $\Lambda^r:\SL \rightarrow \SL$.
 According to the construction of \emph{loc. cit.}, this morphism evaluated at $\cO_K$
 induces by functoriality a morphism on $K_n(\cO_K)=\pi_n(\BSL(\cO_K)^+)$ which
 is nothing else than the operation $\lambda^r$ defined in \emph{loc. cit.}

By the same formula, we obtain  operations $\Lambda^r$ on $\SU$ and $\SO$ respectively,
 and therefore by functoriality an operation on the target of the Borel regulator $\rho_B$.
 It follows from the construction, as we have detailed, that $\rho_B$ commutes with the action
 just constructed on both its source and target.

It remains to determine the induced structure on
 $H_*^{\text{sing}}\big((\SU/\SO)^{r_1} \times \SU^{r_2},\RR\big)$
 and on its primitive part.

For this, we recall the theory of characteristic classes for notation.
 Given a topological group $\mathfrak G$,
 there exists a universal principal topological bundle $E\mathfrak G$
 classifying respectively the special\footnote{\emph{i.e.} with trivial determinant}
 complex and real vector bundles (see e.g. \cite{MilnorB}). One deduces the classical  fibration sequence:
$$
\mathfrak G \rightarrow \mathrm E\mathfrak G \xrightarrow p \B \mathfrak G
$$
where the classifying space $\B \mathfrak G$ is pointed by the trivial principal bundle.
 This implies the existence of the (classical) homotopy equivalence $\mathfrak G \simeq \Omega \B \mathfrak G$
 with target the indicated loop space.
 By adjunction, one deduces the so-called \emph{suspension map}
$$
s:\Sigma \mathfrak G \rightarrow \B \mathfrak G
$$
and the induced map in cohomology:
$$
H^{n+1}_{\text{sing}}(\B \mathfrak G,\RR) \xrightarrow{s^*} H^{n+1}_{\text{sing}}(\Sigma \mathfrak G,\RR) \simeq H^{n}_{\text{sing}}(\mathfrak G,\RR).
$$
Note that this map is natural with respect to the topological group $\mathfrak G$.
 In our case, where $\mathfrak G=\SU, \SO$ is an ind-Lie group, we deduce
 from the Serre spectral sequence associated with the fibration $p$,
 and the known cohomologies of $\mathfrak G$ and $B \mathfrak G$  (see e.g. \cite[Chap. 16]{Switzer}),
 that the suspension map is an isomorphism. By taking real duals, one deduces
 the following homological suspension isomorphisms which fit into the commutative diagram:
\begin{equation}\label{eq:suspension_iso}
\begin{split}
\xymatrix@R=12pt@C=30pt{
H_{n}^{\text{sing}}(\SU,\RR)\ar^-{s_*}_-\sim[r] & H_{n+1}^{\text{sing}}(\B \SU,\RR) \\
H_{n}^{\text{sing}}(\SO,\RR)\ar^-{s_*}_-\sim[r]\ar^{\iota_*}[u] & H_{n+1}^{\text{sing}}(\B \SO,\RR)\ar_{\iota_*}[u]
}
\end{split}
\end{equation}
where $\iota:\SO \rightarrow \SU$ is the natural inclusion.
 Now, the theory of characteristic classes (see again \cite[Chap. 16]{Switzer}) applies and gives:
\begin{align*}
H^{*}_{\text{sing}}(\B \SU,\RR)&=\RR[c_2,\hdots,c_i,\hdots] \\
H^{*}_{\text{sing}}(\B \SO,\RR)&=\RR[p_2,\hdots,p_j,\hdots]
\end{align*}
where the Chern (resp. Pontryagin) classes $c_i$ (resp. $p_j$) are in degree $2i$ (resp. $4j$)
 and in weight $i$ (resp. $j$) for the natural $\lambda$-structure, induced as above.
 One can compute $\iota^*$ via complexification of real vector bundles,
 and one deduces that 
$$
\iota^*(c_{i})=\begin{cases}
0 & i \text{ odd,} \\
(-1)^j(2j-1)!p_j+\text{lower degrees} & i=2j.
\end{cases}
$$
Therefore, the kernel of $\iota^*$ in rational cohomology is generated by the odd Chern classes.
 One should be careful also that $\iota^*$ divides weights by $2$.\footnote{All this comes from the relation
 $p(E)=c(E \otimes_\RR \CC).c(E \otimes_\RR \CC)$ for a real vector bundle $E$, between
 the total Pontryagin and Chern classes. Thus $\iota^*$ can be thought of as extracting a square root.}

Therefore the singular homology of $\B\SU$ is the free symmetric algebra generated by the $c_i^\vee$
  in even degree $2i$ and weight $i$ for $i \geq 2$. The operation of taking the cokernel of $\iota_*$
 kills the classes $c_i^\vee$ for $i$ even.
 According to the construction of $\lambda$-operations recalled above,
 the suspension isomorphisms in diagram \eqref{eq:suspension_iso} shift the degree by $+1$ and preserves the weight.
 Thus one finally obtains the following key computation:
\end{num}
\begin{prop}
The real singular homology with its weight graduation coming from the previously defined $\lambda$-structure
$$
 H_*^{\text{sing}}\big((\SU/\SO)^{r_1} \times \SU^{r_2},\RR\big)
$$
is the free exterior real algebra generated by $r_2$ classes of degree $2i-1$ and weight $i$
 for each $i \geq 2$ and by $r_1$ classes of degree  $2j-1$ and weight $j$ for each odd $j \geq 2$.
\end{prop}
This follows from the preceding discussion and the (dual) K\"unneth formula.
 Let us draw the explicit corollary, which completes the proof of \Cref{thm:Borel}.
\begin{cor}\label{cor:weights-cpt-dual}
Given a number field $K$ with $r_1$  real and $r_2$ complex places,
 a basis for the K-theory graded $\QQ$-vector space $K_*(\cO_K)_\QQ$ restricted to degrees $*\geq 2$
 is given by:
\begin{itemize}
\item $r_2$ elements in degree $2i-1$ and weight $i$ for each $i \geq 2$,
\item $r_1$ elements in degree $2j-1$ and weight $j$ for each odd $j \geq 2$.
\end{itemize}
\end{cor}

\subsection{The effect of transcendental elements}\label{sec:cohm-fields}

We will now deduce from the computations of generic motives
 obtained in the preceding section
 some information about motivic cohomology of fields. 
 All subsequent results are based on formula \eqref{eq:cohm-gmot}.
 First, as a corollary of \Cref{prop:gmot-rational1} and \Cref{thm:Borel},
 we get the following computation:
\begin{prop}\label{prop:cohm-rat-curve}
Let $K$ be a number field. Then for any integers $(n,i) \in \ZZ^2$,
$$
\HB^{n,i}(K(t))=\begin{cases}
\QQ & n=i=0 \\
K^\times \oplus Z_0(\AA^1_K)_\QQ & n=i=1 \\
\HB^{1,i}(K) & n=1, i>1 \\
\bigoplus_{x \in \AA^1_{K,(0)}} \HB^{1,i-1}(\kappa_x) & n=2, i \geq 2 \\
0 & \text{otherwise.}
\end{cases}
$$
We have denoted by $Z_0(-)$ the group of algebraic $0$-cycles.
\end{prop}
As there are infinitely many points $x \in \AA^1_{K,(0)}$ whose residue field $\kappa_x$
 has at least one complex place, one deduces from \Cref{thm:Borel} that
 for all $i \geq 2$, the abelian group $\HB^{n,i}(K(t))$ has infinite rank.

This picture can be confirmed in higher transcendence degree.
 Indeed, applying now \Cref{prop:gmot-higher-trdeg}, one obtains the following proposition.
\begin{prop}\label{prop:cohm-inf-rk}
Let $K$ be a field of characteristic $0$
 such that $\dtr K\geq d>0$.

Then for any $n \in [2,d+1]$ and any $i \geq n$,
 the abelian group $\HB^{n,i}(K)$ has infinite rank.
\end{prop}
\begin{proof}
Indeed, by assumption, $K$ contains the field of rational functions $\QQ(t_1,...,t_d)$.
 According to \Cref{cor:inclusion-splits}, one can assume that $K=\QQ(t_1,,...,t_d)$.
 According to \Cref{prop:gmot-higher-trdeg}, it follows that
 $\HB^{n,i}(K)$ contains $\bigoplus_{x \in \AA^1_{\QQ,(0)}} \HB^{1,i-n+1}(\kappa_x)$
 which is infinite according to \Cref{thm:Borel} and the fact there exists infinity many
 finite monogeneous extensions of $\QQ$ with at least one complex place.
\end{proof}

\begin{rem}
According to the splitting $\varphi$ given in \Cref{prop:gmot-rational1},
 it is straightforward to construct infinitely many free families of elements of $\HB^{n,i}(\QQ(t_1,...,t_n))$
 under the assumptions of the previous proposition.
 As a function field $K/k$ of transcendence degree $d=n$ is a finite extension of $\QQ(t_1,...,t_n)$
 one deduces infinitely many free families of $\HB^{n,i}(K)$.
 We leave these constructions as an exercise to the reader.\footnote{As a hint,
 consider products of symbols extracted from some transcendental basis of $K/\QQ$,
 elements coming from the Borel regulator on an appropriate number field $E$
 and eventually applying restriction or corestriction along finite field extensions on motivic cohomology.
 All these operations can be realized geometrically as well by taking appropriate smooth models.}
\end{rem}

One can further extend this argument to obtain non-finiteness results of motivic cohomology.
\begin{prop}\label{prop:cohm-uncount}
Let $K$ be an extension field of $\QQ$ of infinite transcendence degree.
 Then for any pair of integers $(n,i)$ with $i \geq n \geq 2$, the abelian group
 $\HB^{n,i}(K)$ has rank equal to $\card(K)$.
\end{prop}
It is noteworthy that this result extends an earlier theorem by Springer,
 concerning the Milnor K-theory case (\emph{i.e.} $n=i$): see \cite[5.10]{BT}
 (and below for more discussion). 
\begin{proof}
According to formula \eqref{eq:cohm-presentation}, we get $\card(\HB^{n,i}(K)) \leq \card(K)$.

Next, we consider the opposite inequality.
 By assumption, there exists a subfield $F \subset K$ and a finite morphism
 $F(t_1,\ldots,t_n) \rightarrow K$. Applying \Cref{cor:inclusion-splits}, it suffices to treat the case
 of $K=F(t_1,\ldots,t_n)$. According to \Cref{prop:gmot-higher-trdeg},
 it follows that $\HB^{n,i}(K)$ contains, as a direct factor, the abelian group
 $$\oplus_{x \in \AA^1_{F,{(0)}}} \HB^{1,i-n+1}(\kappa_x).$$
 Applying again \Cref{cor:inclusion-splits}, the latter groups contains for instance
 the abelian group
 $$\oplus_{x \in F} \HB^{1,i-n+1}(\QQ[i])$$ whose rank is bigger or equal to $\card(F)=\card(K)$.
\end{proof}

\begin{num}
As a conclusion, there is a range, controlled by the transcendence degree over $\QQ$,
 in which all motivic cohomology groups of fields have infinite rank.
 Nevertheless, a compelling conjecture arises, which we now state.
 We will state it directly in arbitrary characteristic,
 even though our focus has been on the characteristic $0$ case.
 Recall that the Kronecker dimension of a field $K$ is given by
$$
\delta(K)=\begin{cases}
\dtr(K/\QQ)+1 & \car(K)=0 \\
\dtr(K/\FF_p) & \car(K)=p>0.
\end{cases}
$$
\end{num}
\begin{cjt}\label{conj:vanishing}
Let $K$ be a field. Then,
$$
\HB^{n,i}(K)=0
$$
unless $n=i=0$ or $n \in [1,\delta(K)]$.
\end{cjt}
There is little evidence for this conjecture.
 Nevertheless, it holds for fields $K$ and $K(t)$ such that $\delta(K) \leq 1$:
 as we have seen, the case of number fields follows from \Cref{thm:Borel} and \Cref{prop:cohm-rat-curve}.
 The case of characteristic $p>0$ is a consequence of \cite{QuiFF} and \cite{Harder} (see below).
 We will conclude this section with some further evidence.

\begin{num}\textit{The Beilinson-Soulé conjecture}.
Actually the above conjecture in the range $n<1$ is already well-known
 as the Beilinson-Soulé conjecture, in its ``strong'' form --- this terminology appears in the introduction of \cite{LevineAT}.
 By opposition, the ``weak'' form of the \emph{Beilinson-Soulé conjecture} asserts the vanishing of motivic cohomology
 in degrees $<0$ (see e.g. \cite[\textsection 3, Be2]{SV}).\footnote{The integral version is equivalent to the rational
 one: see \cite[Lem. 24]{Kahn}.} In the sequel, we will simply use the terminology \emph{Beilinson-Soulé conjecture} for the strong form of the conjecture.
\end{num}

\begin{rem}
We observed in \Cref{prop:cohm-rat-curve} that the line $n=1$ on motivic cohomology of fields 
 exhibits special behavior: it does not seem to grow with transcendental elements. In general, one conjectures
 following Beilinson
 that for any $i>1$, $\HB^{1,i}(K)=\HB^{1,i}(K_0)$ where $K_0 \subset K$ is the field of constants (\emph{i.e.} the algebraic closure of the prime subfield of $K$).
 This is sometimes called the \emph{rigidity} conjecture.
 For more discussion on this and the Belinson-Soulé conjecture, the reader is referred to
 \cite[\textsection 4.3.4, Lem. 24]{Kahn}.\footnote{Note that the conjectures 16 and 17 of \emph{loc. cit.} are now established as a theorem
 by Voevodsky, with significant contributions from Rost.}
\end{rem}

\begin{num}\textit{The Milnor K-theory part}.
The above conjecture in the special case $n=i$ was already formulated in \cite[Question after the proof of 5.10]{BT}.
 Even in this special case, nothing is known beyond the case $\delta(K)\leq 1$ already mentioned.
\end{num}

\begin{num}
Recall that the Beilinson-Soulé conjecture is \emph{equivalent} to the existence of a motivic $t$-structure
 on (rational) Tate motives over $K$ with a weight filtration such that $\un(n)$ is of weight $-2n$
 (see \cite[Th. 4.2]{LevineAT}).\footnote{In fact,
 the assertion that $\un(n)$ lies in the heart of the motivic $t$-structure and has weight $(-2n)$ for all $n \in \ZZ$
 is equivalent to the strong form of the Beilinson-Soulé conjecture for $K$.}

In this respect, the novelty of our question concerns the vanishing above the Kronecker dimension
 of the field. Moreover, under the Beilinson-Soulé conjecture,
 the vanishing in this particular range is equivalent to the assertion that the category of Tate motives
 has cohomological dimension $\delta(K)$ (see again \cite[Cor. 4.3]{LevineAT} in the case $\delta(K)=1$).
\end{num}

\begin{num}\textit{The positive characteristic case}.
When $K$ is a field of characteristic $p>0$, the above conjecture is actually a consequence of two well-known conjectures.
 In fact, a stronger vanishing of motivic cohomology is anticipated.
 The Beilinson-Parshin conjecture implies that $\HB^{n,i}(K)=0$ if $n\neq i$ (see \cite[Th. 3.4]{Geisser}).
 Thus, the previous conjecture would follow from the Beilinson-Parshin conjecture and a positive resolution
 of the question of Bass and Tate mentioned above.
\end{num}

\subsection{Motivic cohomology with coefficients in abelian varieties}

We finish this paper by mentioning an interesting cohomology theory that
 comes from Voevodsky's theory of motivic complexes.
\begin{df}
Let $k$ be a perfect field, and $A$ be an abelian variety over $k$.
 For an integer $i \in \ZZ$, we consider the motivic complex in $\DM(k)$:
$$
\underline A(i)=\Sigma^\infty(\underline A) \otimes \un(i).
$$
Then, for any smooth $k$-scheme of finite type, we define the motivic cohomology of $X$
 with coefficients in $A$, degree $n$ and twist $i$, as the abelian group:
$$
H^{n,i}(X,\uA)=\Hom_{\DM(k)}\big(\M(X),\uA(i)[n]\big).
$$
\end{df}
This new cohomology theory inherits all the properties
 of the motive $\M(X)$:
\begin{itemize}
\item  It is contravariant with respect to arbitrary morphisms in $X$,
 and covariant up to twists and shift with respect to proper smoothable morphisms
 (use Gysin morphisms).
\item It satisfies the localization property with respect to a smooth closed pair $(X,Z)$ (use the Gysin triangle).
\item It satisfies the projective bundle formula as well as the blow-up formula with smooth center.
\end{itemize}
According to \Cref{thm:abelian-var-HI}, $\uA(-i)=0$ for $i>0$.
 In addition, one deduces from the cancellation theorem
 that for $i\geq0$:
$$
H^{n,i}(X,\uA)) \simeq H^n_\nis(X,\uA(i))
$$
where the left-hand side is the Nisnevich cohomology of the motivic complexes $\uA \otimes \un(i) \in \DMe(k)$.
 According to the property of the derived tensor product in $\DMe(k)$, one deduces
 that $\uA(i)$ is a complex of Nisnevich sheaves concentrated in cohomological degree $\leq i$.
 One deduces:
$$
H^{n,i}(X,\uA)=\begin{cases}
\uA(X) & n=i=0, \\
0 & \text{$i<0$ or $n>2i$ or $n-i>\dim(X)$.}
\end{cases}
$$
Finally, one gets for any function field $K/k$ the following computation:
$$
H^{n,n}(K,\uA) \simeq \KSM n (K;A)
$$
with the notation of \Cref{thm:abelian-var-HI}.
 Moreover, it follows from \Cref{thm:deglise} that $\KSM *(-;A) \simeq \widehat{\uA}_*$ is a cycle module over $k$.
 One deduces a differential from Rost's cycle complex with coefficients in the latter cycle module:
$$
\bigoplus_{y \in X^{(n-1)}} \KSM 1 (\kappa_y;A) \xrightarrow{\mathrm{div}_A} \bigoplus_{x \in X^{(n)}} A(\kappa_x).
$$
\begin{df}
Consider the above notation.
 One defines the $n$-codimensional algebraic cycles of $X$ with coefficients in $A$ as the abelian group:
$$
Z^n(X;A)=\bigoplus_{x \in X^{(n)}} A(\kappa_x).
$$
The map $\mathrm{div}_A$ defined above is called the divisor class map. One defines the Chow group of $X$
 in codimension $n$ and coefficients in $A$ as the cokernel:
$$
\CH^n(X;A):=\coKer(\mathrm{div}_A)=Z^n(X;A)/\sim_A.
$$
\end{df}
Here $\sim_A$ is the obvious equivalence relation,
 that we call rational equivalence. 
 As a consequence of existence of the coniveau spectral sequence together with the previous computation,
 one gets:
\begin{prop}
Let $X$ be a smooth $k$-scheme and $A$ an abelian variety.
 Then for any integer $n \geq 0$, there exists a canonical isomorphism:
$$
H^{2n,n}(X,\uA) \simeq \CH^n(X;A).
$$
\end{prop}

\begin{rem}
\begin{enumerate}
\item By considering motives of the form $\uA \otimes \underline B$ in $\DMe(k)$,
 one can define exterior products of motivic cohomology with coefficients in abelian varieties.
\item Up to inverting the characteristic exponent of $k$,
 the preceding definition can be extended to singular $k$-schemes
 in a contravariant way. The corresponding extended cohomology theory satisfies $\cdh$-descent.
\item One can define a compactly supported version by using motives with compact support.
\end{enumerate}
\end{rem}

\begin{num}
The preceding definition is useful to state the computation of motivic cohomology
 of function fields given in \Cref{prop:gmot_tr1}.
\end{num}
\begin{prop}
Let $K$ be a number field,
 and $\bar C/K$ be a smooth projective curve with a base point $x_0 \in \bar C_{(0)}$.
 We let $L$ be the function field of $\bar C$, and $A$ be the jacobian of the pointed curve $(\bar C,x_0)$.

Then, there are canonical morphisms
$$
\pi_{L/K}^{n,i}:\Hmot^{n-2,i-1}(K,\uA) \rightarrow \Hmot^{n,i}(L).
$$
which are isomorphisms for any pair $(n,i)$ such that $i \geq n$, $n \in \ZZ-[1,3]$.

For $i>1$, the following map is an isomorphism:
$$
\Hmot^{-1,i-1}(K,\uA) \oplus \Hmot^{1,i}(K) \xrightarrow{(\pi^{1,i}_{L/K},\varphi^*)} \Hmot^{1,i}(L)
$$
Moreover, for $i \geq 3$, the following sequence of abelian groups is exact:
$$
0 \rightarrow \Hmot^{0,i-1}(K,\uA) \xrightarrow{\pi_{L/K}^{2,i}} \Hmot^{2,i}(L)
 \xrightarrow{\sum_x \partial_x}
\bigoplus_{x \in C_{(0)}} \Hmot^{1,i-1}(\kappa_x) \rightarrow \Hmot^{1,i-1}(K,\uA)
 \xrightarrow{\pi_{L/K}^{3,i}} \Hmot^{3,i}(L) \rightarrow 0.
$$
\end{prop}
\begin{proof}
This follows directly from the homotopy exact sequence of \Cref{prop:gmot_tr1},
 with the duality result of \Cref{lm:dual-abvar}, and using the
 vanishing of motivic cohomology of the fields $K$ and $\kappa_x$ from \Cref{thm:Borel}.
\end{proof}

\begin{rem}
One can therefore interpret the \Cref{conj:vanishing} for the function field $L$
 in terms of properties of the cohomology groups $\Hmot^{n,i}(K,\uA)$,
 which are attached to an abelian variety over a number field.
 These cohomology groups may be more accessible.
\end{rem} 

\section*{Appendix}

\section{Reference guide on $\infty$-Categories}\label{sec:infty-cat}

\begin{num} 
The theory of $\infty$-categories provides a powerful foundation
 for both homotopy theory and homological algebra.
It enhances both derived and homotopy categories, offering a unified and flexible framework.
 The beauty of the theory is that it allows the direct transport of categorical concepts
 into the world of $\infty$-categories --- concepts that traditionally required derived constructions.
 In particular, stable $\infty$-categories provide an alternative framework
 to triangulated categories, in which one can resolve the traditional
 non-functoriality of cones --- a source of many complications.

We provide in this appendix a comprehensive overview the theory.
 Our goal is to offer a concise reference guide that equips the reader with a working knowledge of $\infty$-categories,
 sufficient to follow the constructions involved in the $\infty$-categorical presentation of Voevodsky's theory.
 At the same time, we include references for those interested in exploring the theory in greater depth.
 A first point of entry is the survey paper of Denis-Charles Cisinski
 \cite{CisB}, which is itself based on Jacob Lurie's fundamental treatise
 \cite{LurieHTT}.
 Concerning (symmetric) monoidal structures on $\infty$-categories,
 a good survey paper is \cite{GrothI}, based on Lurie's (unpublished) treatise \cite{LurieHA}.
 When needed, we also refer directly to Lurie’s original texts.

A final comment about the foundational issues of the theory.
 As in most works on higher homotopy theory,
 we assume the existence of an ambient universe in which all our
 ($\infty$-)categories reside. We will additionally need to have
 at least three universes: a smaller one to get a working notion
 of colimits and limits, and a bigger one to be able to consider
 objects such as the $\infty$-category of functors, or of $\infty$-categories
 (that belong to the initial universe). Throughout the remainder of the text,
 we will no longer explicitly mention these conventions when it is clear
 from the context which universe is being referred to.
\end{num}

\subsection{$\infty$-Categories and mapping spaces}

\begin{num}\textbf{Higher morphisms in $\infty$-categories.}
An $\infty$-category $\iC$ consists not only of objects and morphisms (sometimes called $1$-morphisms),
 but also of $n$-morphisms with $n>1$, which are to be thought of
 as equivalences between $(n-1)$-morphisms. 

There are several equivalent presentations of these objects, called \emph{models}.
 In these notes, we adopt Joyal's model of quasi-categories, which is both very practical and efficient.
 One defines an $\infty$-category 
 as a simplicial set $\iC$ that satisfies an additional property called
 the weak Kan condition: the canonical morphism of simplicial sets $\iC \rightarrow *$
 satisfies the right lifting property
 with respect to all inclusions $\Lambda_n^k \rightarrow \Delta^n$ of a horn for all $0<k<n$.\footnote{This
 property is depicted by the following diagram:
$$
\xymatrix@=10pt{
\Lambda_n^k\ar[r]\ar[d] & \iC\ar[d] \\
\Delta^n\ar[r]\ar@{-->}[ru] & {}*
}
$$
where solid arrows is the data of a commutative square,
 and the dashed arrow is required to exist so that all the diagram is commutative.
 See \cite[Def. 3.1, Rem. 3.2]{CisB} for more details.} 

The weak Kan condition ensures that composition of $1$-morphisms exists (\cite[Rem. 4.2]{CisB}),
 but is in general non-unique.
 It is well-defined up to a \emph{contractible} set of choices
 (see \cite[Rem. 8.13]{CisB}). This ambiguity in $\infty$-categories
 can be puzzling at first, but it is the core specificity of $\infty$-category theory.
 In any case, on defines the \emph{identity morphism} of an object $X$ in $\iC$
 as the degenerate $1$-simplex $\Id_X=s_0^0(X)$ with vertices $X$.
\end{num}

\begin{ex}\textbf{Nerve functor.}\label{ex:nerve}
Let $\catC$ be an ordinary category.
 One associates to it a simplicial set $\Nrv \catC$, called the \emph{nerve} of $\catC$:
 the vertices are the objects of $\catC$ and the $n$-simplices are collections of composable
 morphisms $X_0 \xrightarrow{f_0} X_1 \rightarrow \hdots \xrightarrow{f_n} X_n$.
 Then $\Nrv \catC$ does satisfy the weak Kan condition (see \cite[Prop. 2.1]{CisB}).

As another example, one can check that the simplicial set $\Delta^1$ is an $\infty$-category.
 Indeed, it is the nerve of the category with two objects $0$ and $1$ and only
 one non-identity arrow $0 \rightarrow 1$.
\end{ex}

\begin{ex}\textbf{Opposite $\infty$-category.}\label{ex:opposite-infty}
Let $\iC$ be an $\infty$-category.
 One defines the \emph{opposite $\infty$-category of $\iC$} as the simplicial set $\iC^{op}$
 with $n$-simplices given by $(\iC^{op})_n=\iC_n$ and with degeneracies $\tilde d_n^i$ and 
 faces $\tilde s_n^i$ given in terms of degeneracies and faces of $\iC$ by the formula:
$$
\tilde d_n^i:=d_n^{n-i}, \tilde s_n^i:=s_n^{n-i}.
$$
As an exercise, the reader can check that a simplicial set $\iC$
 is an $\infty$-category if and only if the simplicial set $\iC^{\op}$ defined
 above is so.
 Similarly, $\Nrv(\catC^{\op})=(\Nrv \catC)^{op}$.
\end{ex}

\begin{ex}\textbf{Localizations and derived categories.}\label{ex:loc-icat1}
Let $\catC$ be an ordinary category and let $W$ be a set of morphisms of $\catC$.
 Then there exists an $\infty$-category denoted by $\Nrv C[W^{-1}]$
 and called the $\infty$-categorical $W$-localization of $\catC$,
 with a morphism of simplicial sets $\Nrv \catC \rightarrow \Nrv C[W^{-1}]$,
 such that morphisms of $\catC$ become invertible
 (isomorphism in the sense to be defined shortly in \Cref{num:associated-htp})
 in the target $\infty$-category (see \cite[Th. 9.3]{CisB} for an even more general
 statement).\footnote{The most efficient construction is to use the
 so-called hammock localization of Dwyer and Kan, which to $(\catC,W)$
 associates a canonical simplicial category, and then to apply the simplicial
 nerve functor (\cite[Def. 1.1.5.5]{LurieHTT}).}

This construction allows to enhance the derived category of any abelian
 category $\catA$ to an $\infty$-category that we will denote by $\iDer(\catA)$,
 and call it the derived $\infty$-category associated with $\catA$:
 one defines $\iDer(\catA)$ as the $\infty$-categorical localization of
 $\Nrv \Comp(\catA)$, the nerve of the category of complexes in $\catA$, along
 quasi-isomorphisms.\footnote{Note for completeness that there is an alternative construction, called the \emph{dg-nerve},
 which uses the canonical differential graded (dg)
 enhancement of $\Comp(\catA)$ to define an $\infty$-category $\Nrv_{dg} \Comp(\catA)$. It is equivalent
 to the localization of the $\Nrv \Comp(\catA)$ with respect to homotopy equivalences
 of complexes, so that after localization along quasi-isomorphism, it gives
 an $\infty$-category which is quasi-equivalent to the one we denoted by $\iDer(\catA)$.
 See \cite[1.3.1.6]{LurieHA} for the construction, which is originally due to Hinich and Schechtman who call it
 the \emph{Sugawara functor}.}

Another fundamental example is the $\infty$-categorical localization
 of the category of topological spaces (or what amounts to the same, simplicial sets)
 along the weak (homotopy) equivalences.
 We will denote this $\infty$-category by $\iS$ and call it the
 $\infty$-category of spaces.\footnote{A common criticism against this terminology:
 the word space is over-used. Other possible terminologies are the $\infty$-category of
 \emph{$\infty$-groupoids}, or of \emph{anima}.}
\end{ex}

\begin{ex}\textbf{Spanned $\infty$-categories.}
Given an $\infty$-category $\iC$ and a set of objects $E \subset \iC_0$,
 one can always consider the sub-simplicial set $\iC_E$ of $\iC$
 whose $n$-simplices are given by the elements of $\iC$ whose iterated $n$-th degeneracies
 belongs to the subset $E$. It is not difficult to check that $\iC_E$ satisfies
 the weak Kan condition.\footnote{Use that $\iC$ satisfies this condition and
 that the simplicial sets $\Lambda_n^k$ and $\Delta^n$ have the same vertices.}

We call $\iC_E$ the full sub-$\infty$-category of $\iC$
 spanned by the set of objects $E$.
 This allows us, for example, to define the bounded (resp. bounded below/above)
 derived $\infty$-category $\iDer^\epsilon(\catA)$ where
 $\epsilon=b$ (resp. $+$, $-$).
\end{ex}

\begin{num}\textbf{Associated homotopy category.}\label{num:associated-htp}
Let $\iC$ be an $\infty$-category.
 Given morphisms $f,g:X \rightarrow Y$ in $\iC$, one says that $f$ and $g$ are \emph{homotopic}
 if there exists a $2$-morphism $H$ with degeneracies given by $d_2^2(H)=f$, $d_2^1(H)=g$, $d_2^0(H)=\Id_Y$.
 This is depicted as follows:
$$
\xymatrix@C=20pt@R=14pt{
& Y\ar^{\Id_Y}[rd] & \\
X\ar_f[rr]\ar^g[ru] & \ar@{=>}|/-4pt/H[u]& Y
}
$$
It can be shown that this defines an equivalence relation on the set of morphisms
 for $X$ to $Y$ in $\iC$. Its quotient set is called the set of homotopy classes of morphisms
 from $X$ to $Y$.
 Moreover, one can define a category $\Ho(\iC)$ called the \emph{associated homotopy category},
 whose objects are those of $\iC$, and morphisms are given by the homotopy classes of $1$-morphisms
 (see \cite[Prop. 4.3]{CisB}).  

We can now start to transpose the language of category theory:
 we say that a morphism $f$ in an $\infty$-category $\iC$ is an \emph{isomorphism}
 if it is an isomorphism in $\Ho(\iC)$.
\end{num}

\begin{ex}\label{ex:associated-htp}
\begin{enumerate}[leftmargin=*]
\item The homotopy category associated with the nerve $\Nrv \catC$ of an ordinary category is
 equivalent to the original category $\catC$.
\item Let $\iC$ be an $\infty$-category and $\catD$ be an ordinary category.
 Then, giving an $\infty$-functor $\iC \rightarrow \Nrv \catD$ is equivalent
 to giving an ordinary functor $\Ho(\iC) \rightarrow \catD$. This is a useful exercice
 to become familiar with the above definitions. See also \cite[Proposition 1.2.3.1]{LurieHTT}.
\item The homotopy category associated with the $\infty$-derived
 category of an abelian category $\catA$ is the usual derived category:
 $\Ho \iDer(\catA)=\Der(\catA)$.\footnote{The equality symbol
 is justified by the universal property of a localized category.}
\item An \emph{$\infty$-groupoid} is an $\infty$-category $\iC$ in which all morphisms are isomorphisms.
 This property is equivalent to the \emph{Kan condition} (\cite[Th. 6.3]{CisB}) {on the} corresponding simplicial set $\iC$.
\end{enumerate}
\end{ex}

\begin{num}\textbf{Mapping spaces in $\infty$-categories.}\label{num:mapping-icat}
There is a practical way to package 
 the information encoded by the higher morphisms of an $\infty$-category $\iC$.
 First recall that the category of simplicial sets $\sS$ admits an internal Hom, usually called
 the mapping space, that we will simply denote by $\uHom_{\sS}$ to avoid possible confusion.

Let $X$ and $Y$ be objects of $\iC$. One defines the \emph{mapping space} $\Map_{\iC}(X,Y)$ 
 of morphisms from $X$ to $Y$ as the sub-simplicial set of $\uHom(\Delta^1,\iC)$ generated by
 vertices $f:\Delta^1 \rightarrow \iC$ such that $f(0)=X$ and $f(1)=Y$.
 This is not immediately obvious, but one also gets a bijection compatible with compositions
 (see \cite[1.2.3.9]{LurieHTT}):
$$
\Hom_{\Ho\iC}(X,Y) \simeq \pi_0\big(\Map_{\iC}(X,Y)\big).
$$
Further, a nice feature of the theory is that the simplicial set
 $\Map_{\iC}(X,Y)$ satisfies the (usual) Kan condition; in other words,
 it is a \emph{Kan complex}.\footnote{Or equivalently a fibrant object for the Quillen model structure on simplicial sets.}
 This is a foundational property, originally proved by Joyal,
 and a consequence of the weak Kan property of $\iC$ (\cite[Th. 7.1]{CisB}).\footnote{In other words,
 $\Map_{\iC}(X,Y)$ can be considered as an $\infty$-groupoid,
 corresponding to the fact that $n$-morphisms for $n>1$ are to be considered as equivalences.}
\end{num}

\begin{rem}\label{rem:simplicial-cat}
One can associate to the $\infty$-category $\iC$ a \emph{simplicial category} $\iC_{rect}$
 (\emph{i.e.} a category enriched in simplicial sets) in such a way that the above
 Kan complex $\Map_{\iC}(X,Y)$ is weakly equivalent to the enriched morphisms in $\iC_{rect}$.
 More precisely, there is an equivalence between the homotopy category of $\infty$-categories
 and that of simplicial categories --- each equipped with an appropriate model
 category structure --- arising from a Quillen equivalence of model categories;
 see \cite[\textsection 12]{CisB}.
\end{rem}

\begin{ex}
Let $\catA$ be an abelian category. Given two objects $A$ and $B$ of $\catA$,
 and an integer $n\geq 0$, one deduces from \Cref{ex:associated-htp}(3) a bijection:
$$
\pi_0\big(\Map_{\iDer(\catA)}(A,B[n])\big) = \Ext_{\catA}^n(A,B).
$$
Note in particular that the set on the left-hand side
 acquires an abelian group structure.
 We will see in \Cref{num:stable} that this group structure can be obtained via an intrinsic property
 of $\infty$-categories, that of being \emph{stable}.
\end{ex}

\begin{ex}\label{ex:final-initial}
Let $\iC$ be an $\infty$-category.
 An object $X$ of $\iC$ will be said \emph{initial} (resp. \emph{final})
 if for any object $Y$ of $\iC$, the mapping space $\Map_{\iC}(X,Y)$ (resp. $\Map_{\iC}(Y,X)$)
 is contractible.

The sub-$\infty$-category $\iC_0$ of $\iC$ spanned by initial (resp. final) object
 is contractible in the sense that the canonical functor $\iC_0 \rightarrow *$
 to the final $\infty$-category is an equivalence.\footnote{To be shortly defined in \Cref{num:functors-icat};
 in fact, one readily sees from the definition that it is fully faithful and essentially surjective.}
 In other words,
 the space of choices of such objects is contractible.
\end{ex}

\subsection{$\infty$-Functors, equivalences and adjunctions}

\begin{num}\textbf{Functors and natural transformations.}\label{num:functors-icat}
A functor between two $\infty$-categories is a morphism of simplicial sets
 $F:\iC \rightarrow \iD$. Then it automatically induces a functor\footnote{As an exercise,
 the reader can check it respects the homotopy relation on morphisms defined in \Cref{num:associated-htp}.}
 of the associated homotopy categories:
$$
\Ho(F):\Ho(\iC) \rightarrow \Ho(\iD).
$$
Moreover, given any two objects $X$, $Y$ of $\iC$, one gets a canonical
 morphism of mapping spaces, which we denote by the same letter:\footnote{In fact,
 as noted in \Cref{rem:simplicial-cat}, the functor $F$ can be \emph{rectified}:
 one can find simplicial categories modeling $\iC$ and $\iD$, along with a functor between them
 that corresponds to $F$. This follows from the
 Quillen equivalence mentioned in the previous remark.}
\begin{equation}\label{eq:functorial-Map}
F:\Map_{\iC}(X,Y) \rightarrow \Map_{\iC}\big(F(X),F(Y)\big).
\end{equation}

Recall that there exists an internal Hom functor in the category of simplicial sets,
 that we will denote by $\uHom_{\sS}$.
 It can be checked that, because $\iD$ satisfies the weak Kan condition,
 the simplicial set $\uHom_{\sS}(\iC,\iD)$ satisfies the weak Kan condition (see \cite[Prop. 5.1]{CisB}).
 In other words, it is an $\infty$-category called
 the \emph{$\infty$-category of functors between $\iC$ and $\iD$},
 and simply denoted by $\Fun(\iC,\iD)$.

A \emph{natural transformation} of functors between $\infty$-categories is by definition a
 ($1$-)morphism in the $\infty$-category $\Fun(\iC,\iD)$.
 We will say that $F$ is:
\begin{itemize}
\item \emph{conservative} if for any morphism $f:X \rightarrow Y$ of $\iC$,
 $f$ is an isomorphism if and only if $F(f)$ is an isomorphism.
 This amounts to asking that the induced functor $\Ho(F):\Ho(\iC) \rightarrow \Ho(\iD)$
 is conservative \emph{i.e.} preserves and detects isomorphisms,
\item \emph{fully faithful} if for any objects $X$, $Y$ in $\iC$,
 the map \eqref{eq:functorial-Map} is a weak equivalence,
\item \emph{essentially surjective} if the functor $\Ho(F)$ is essentially surjective.
\item An \emph{equivalence of $\infty$-categories} if it is fully faithful and essentially
 surjective.
\end{itemize}
Note that one can define the \emph{essential image} of $F$ as the subcategory of $\iD$
 spanned by the objects $X$ of $\iD$ such that there exists an isomorphism 
 of the form $X \rightarrow F(Y)$. Then the $\infty$-functor $F$ is essentially surjective
 if its essential image is exactly $\iD$.
\end{num}

\begin{ex}\textbf{Comma categories.}\label{ex:comma} 
Let $\iC$ be an $\infty$-category and $X$ be an object of $\iC$.

Seeing $\Delta^1$ as an $\infty$-category (\Cref{ex:nerve}),
 we can consider the $\infty$-category $\Fun(\Delta^1,\iC)$ 
 of arrows in $\iC$. By evaluating at $0$ and $1$ respectively,
 one obtains the source and target $\infty$-functors: 
$$
S,T:\Fun(\Delta^1,\iC) \rightarrow \iC
$$ 
One defines the \emph{comma $\infty$-category} $\iC/X$ (resp. $X/\iC$) of objects over (resp. under) $X$
 as the sub-$\infty$-category spanned by objects $f$ of $\Fun(\Delta^1,\iC)$
 such that $T(f)=X$ (resp. $S(f)=X$).
\end{ex}

\begin{ex}\label{ex:loc-icat2}
Using \Cref{rem:simplicial-cat}, it is possible to define the $\infty$-category $\iCat$
 of (small) $\infty$-categories.\footnote{Formally, it is the $\infty$-categorical localization
 of the category of $\infty$-categories (as a full sub-category of the category of simplicial sets)
 along equivalences of $\infty$-categories.}

We can now give a more precise formulation of \Cref{ex:loc-icat1}.
 Given a pair $(\catC,W)$ as in that example,
 we consider the comma $\infty$-category $\iC/\iCat$ --- just defined ---
 of $\infty$-categories under $\iC$. We let $(\iC/\iCat)_W$ be the sub-$\infty$-category
 spanned by objects $F:\iC \rightarrow \iD$ such that for any $f \in W$, $F(f)$ is an isomorphism in $\iD$.

It is now a theorem (see \cite[7.1.3]{CisHC}) that the $\infty$-category $(\iC/\iCat)_W$ admits an initial object
 $\pi:\iC \rightarrow \iC[W^{-1}]$. This defines the $W$-localisation via a universal property,
 and shows that it is unique (up to a contractible set of choices).
\end{ex}

\begin{num}
An \emph{adjunction of $\infty$-categories} is a pair of functors:
$$
F:\iC \leftrightarrows \iD:G
$$
together with a natural transformation $\epsilon:F \circ G \rightarrow \Id_\iD$
 such that the composite map:
$$
\Map_{\iC}(X,G(Y)) \xrightarrow{F} \Map_{\iC}(F(X),F \circ G(Y))
 \xrightarrow{\epsilon_*} \Map_{\iC}(F(X),Y)
$$
is a weak equivalence.

As usual, we also say that $F$ (resp. $G$) is a left (resp. right) adjoint to $G$
 (resp. $F$).\footnote{In fact,
 it can be shown that if $F$ admits a left/right adjoint $G$, then $G$ is
 unique up to a contractible set of choices (\cite[Rem. 5.2.2.2]{LurieHTT}). In particular,
 we will speak of \emph{the} left/right adjoint of $F$.}
 In the literature on $\infty$-categories, one also frequently finds the notation $F \dashv G$.
 The map $\epsilon$ is called the \emph{co-unit} of the adjunction. In this situation,
 one can also define the unit of the adjunction as a natural transformation
 $\Id_{\iC} \rightarrow G \circ F$ with the expected properties (see \cite[Th. 10.7]{CisB}).
 According to this definition, it is clear that such an adjunction induces
 an adjunction of the associated homotopy categories:
$$
\Ho(F):\Ho(\iC) \leftrightarrows \Ho(\iD):\Ho(G).
$$
As in the ordinary categorical case, one obtains the following useful facts:
\begin{enumerate}
\item Let $(F,G)$ be adjoint $ \infty$-functors.
 Then the $\infty$-functor $F$ (resp. $G$) is fully faithful
 if and only if the unit map
 $\Id_{\iC} \rightarrow GF$ (resp. co-unit map $GF \rightarrow \Id_{\iD}$)
 is an isomorphism.\footnote{This is obvious by using the associated homotopy
 category.}
\item An $\infty$-functor $F$ is an equivalence if it admits a right (resp. left)
 adjoint such that both $F$ and $G$ are fully faithful (see \cite[Th. 7.7]{CisB}).
\end{enumerate}
\end{num}

\begin{ex}\label{ex:left-loc}
In practice, there is a much better behaved notion of localizations
 of $\infty$-categories than the general one of \Cref{ex:loc-icat2}.

Given an $\infty$-category $\iC$, a \emph{left localization} of $\iC$
 is an $\infty$-category $\iD$ together with an $\infty$-functor
 $\pi:\iC \rightarrow \iD$ which admits a fully faithful right adjoint $\nu$.\footnote{We follow the terminology
 of \cite[Def. 15.6]{CisB}, inspired by Bousfield's well-known notion of left localization of model categories.}

In fact, $(\iD,\pi)$ is then a localization of $\iC$,
 in the sense of \Cref{ex:loc-icat2}, with respect to the so-called
 \emph{$\pi$-equivalences}: the morphisms $f$ in $\iC$ such that $\pi(f)$ is an isomorphism.
 Moreover, the essential image of $\nu$ is spanned by the so-called \emph{$\pi$-local objects} of $\iC$,
 that is the objects $X$ of $\iC$ such that for any $f:A \rightarrow B$, the induced map
 of mapping spaces
$$
f_*=\Map_\iC(f,X):\Map_\iC(A,X) \rightarrow \Map_{\iC}(B,X)
$$
is a weak equivalence.
 Finally, the composite functor $L=\nu \circ \pi$ is called the associated
 localization functor.
\end{ex}

\subsection{Limits, colimits and presentable $\infty$-categories}

\begin{num}\textbf{Main properties of presentable $\infty$-categories}.\label{num:presentabl-ppties}
The purpose of this subsection is to introduce the notion of \emph{presentable $\infty$-categories}.
 In practice, this is the most important property of an $\infty$-category.
 To motivate this notion and assist the reader looking for a quick overview,
 we begin by listing the key properties of these particular $\infty$-categories
 --- at the cost of anticipating some definitions, particularly that of presentability itself 
 which will be given in \Cref{df:presentable}.
 We also provide references to the literature where each property is carefully established.
\begin{enumerate}
\item A presentable $\infty$-category $\iC$ admits all colimits and limits (see \cite[5.5.1.1, 5.5.2.4]{LurieHTT}).
\item Let $F:\iC\rightarrow \iD$ be a functor between presentable $\infty$-categories.
 Then $F$ admits a right (resp. left) adjoint if and only if it commutes with colimits
 (resp. with limits and is accessible). (See \cite[Cor. 5.5.2.9]{LurieHTT}.)
\item Let $\iC$ be a presentable $\infty$-category,
 and $W$ be a small set of morphisms of $\iC$.
 We introduce the following terminology:
\begin{itemize}
\item An object $P$ of $\iC$ is called \emph{$W$-local} if for any $W$-equivalence $f$,
 the induced map of spaces $\Map_\iC(f,P)$ is a weak equivalence.
\item A morphism $f$ of $\iC$ is called a \emph{$W$-equivalence} if
 for any $W$-local object $P$, the map of spaces $\Map_\iC(f,P)$ is a weak equivalence.
 We let $\bar W \supset W$ be the set of $W$-equivalences.
\end{itemize}
 Then the $\infty$-categorical localization $\iC[\bar W^{-1}]$
 is presentable and the canonical functor $\pi:\iC \rightarrow \iC[\bar W^{-1}]$ 
 is an accessible left localization in the sense of \Cref{ex:left-loc} (see \cite[5.5.4.20]{LurieHTT}).

 Letting $\nu$ be the right adjoint of $\pi$,
 the composite functor $L_W:=\nu \circ \pi:\iC \rightarrow \iC$ is called the \emph{$W$-localization functor}.
 Its essential image is spanned by the $W$-local objects.
\end{enumerate}
The last construction is crucial in many applications of $\infty$-category theory,
 as we can see in the case of Voevodsky's motives.
\end{num}

\begin{num}\label{num:(co)lim-ifty}
Let $\iI$ be an $\infty$-category. An $\infty$-functor $F:\iI \rightarrow \iC$
 can be regarded as an $\iI$-diagram in $\iC$.
 There exists a unique $\infty$-functor $c:\iI \rightarrow *$ to the final $\infty$-category,
 which induces the constant $\iI$-diagram functor:
$$
ct_\iI:\iC=\Fun(*,\iC) \rightarrow \Fun(\iI,\iC).
$$
One says that \emph{colimits} (resp. \emph{limits}) \emph{indexed by $\iI$-diagrams} exist in $\iC$
 if the $\infty$-functor $ct_\iI$ admits a left (resp. right) adjoint
$$
\underset {\iI} \colim \text{ resp. } \lim_{\iI}:\Fun(\iI,\iC) \rightarrow \iC .
$$
As usual in category theory, if $\iI$ is a discrete set $E$,
 one refers to these as \emph{coproducts} (resp. \emph{products}) indexed by $E$.

It is not reasonable to require that a given $\infty$-category
 admits limits/colimits indexed by any $\infty$-category.
 For this reason, when considering limits/colimits, we will always implicitly
 assume that the indexing $\infty$-category $\iI$ belongs to a (fixed) smaller universe;
 one says that $\iI$ is \emph{small}. In other words, \textbf{we implicitly only consider
 small colimits/limits}.
 We will also say that an $\infty$-category is \emph{finite} if it is equivalent
 to a simplicial set with only finitely many non-degenerate simplices.

 One says that $\iC$ \emph{admits all (resp. finite) colimits/limits}
 if it admits limits/colimits indexed by any (resp. any finite) $\infty$-category.

Let $F:\iC\rightarrow \iD$ be a functor between $\infty$-categories
 which admits all/finite colimits/limits.
 Let us denote by $F_*:\Fun(\iI,\iC) \rightarrow \Fun(\iI,\iD)$
 the functor given by left composition with $F$.
 We will say that \emph{$F$ commutes with all colimits/limits (resp. finite colimits/limits)}
 if for any (resp. any finite) $\infty$-category $\iI$, the canonical natural transformations
 (obtained by adjunction)
$$
\underset{\iI}\colim \circ F_* \rightarrow F \circ \underset{\iI}\colim \text{\quad / \quad} F \circ \lim_\iI \rightarrow \lim_\iI \circ F_*
$$
are isomorphisms.

To mention the last set-theoretic condition required for the main definition of this subsection,
 we say that the $\infty$-functor $F$ is \emph{accessible}
 if there exists a regular cardinal $\kappa$ which belongs to the ambient universe
 such that $F$ commutes with $\kappa$-filtered colimits (see \cite[15.2]{CisB}).
 An \emph{accessible left localization} is a left localization 
 $\pi:\iC \rightarrow \iD$ as in \Cref{ex:left-loc} whose left adjoint $L$
 is accessible.
\end{num}

\begin{rem}\emph{Warning.}\label{rem:colim-infty-classical}
One should be careful that the colimit (resp. limit) of a diagram $F:\iI \rightarrow \iC$,
has nothing to do with the colimit (resp. limit) of the induced functor $\Ho(F):\Ho(\iI) \rightarrow \Ho(\iC)$
 on the homotopy categories.
 It is in fact much closer to what is called a homotopy colimit (resp. limit) in the framework of model
 categories (see \ref{num:model}). This terminology has largely fallen out of use
 in the modern $\infty$-categorical framework.

However, there are two notable exceptions to this warning.
First, if $\iC$ is the nerve of an ordinary category $\catC$,
 then colimits (resp. limits) in $\iC$ correspond exactly to colimits and limits
 in $\catC$, in the classical categorical sense.
 Second, when the indexing category
 $\iI$ has no non-identity morphisms, colimits (resp. limits) become coproducts (resp. products) and
 they agree when computed in both the $\infty$-category $\iC$
 and its homotopy $\infty$-category $\Ho(\iC)$.

In any case, it is usually clear from the context in what framework (ordinary, or $\infty$-categorical)
 one uses the terminology colimit (resp. limit). 
 \end{rem}

\begin{ex}\label{eq:pull-push}
Let $\iC$ be an $\infty$-category that admits finite limits (resp. colimits).
 Then one calls \emph{pullbacks} (resp. \emph{pushouts}) the limits (resp. colimits)
 indexed by the nerve of the category depicted as follows:
$$
\xymatrix@=10pt{
& \bullet\ar[d]  & \text{resp.} & \bullet \ar[d]\ar[r] & \bullet \\
\bullet\ar[r] & \bullet  & & \bullet &
}
$$
An $\infty$-category $\iC$ admits all (resp. finite) limits/colimits 
 if it admits all (resp. finite) products/coproducts and all pullbacks/pushouts
 (see \cite[11.8, 11.9]{CisB}).
\end{ex}

Recall that the opposite $\infty$-category has been defined in \Cref{ex:opposite-infty}.
 Given a small $\infty$-category $\iC_0$, the \emph{$\infty$-category of presheaves on $\iC_0$}
 is defined, as expected, as the $\infty$-category $\iPSh(\iC_0):=\Fun(\iC_0^{\op},\iS)$.
 We can now state the main definition of this subsection.
\begin{df}\label{df:presentable}
An $\infty$-category $\iC$ will be called \emph{presentable} if there exists
 a small $\infty$-category $\iC_0$ such that $\iC$
 is an accessible left localization of the $\infty$-category of presheaves on $\iC_0$
 (see also \cite[15.6, 15.9]{CisB}).

A morphism (also called \emph{left functor}) of presentable $\infty$-categories will be an $\infty$-functor
 which preserves colimits --- or equivalently admits a right adjoint, \Cref{num:presentabl-ppties}(1).
 We define the $\infty$-category $\iCatP$ of presentable $\infty$-categories
 as the non-full sub-$\infty$-category of $\iCat$ spanned by $\infty$-categories and their morphisms.
\end{df}

\begin{ex}\label{ex:presentable-htp}
The $\infty$-category of spaces $\iS$ (\Cref{ex:loc-icat1}) is presentable ---
 given the above definition, this is tautological as $\iS$ is the $\infty$-category of presheaves
 on the final $\infty$-category $*$.

As a further example, one can apply the localization construction explained in \Cref{num:presentabl-ppties}
 to describe the original construction of Bousfield localization.
 Let $R$ be a ring. One says that a morphism $f:X \rightarrow Y$ in $\iS$ is an $R$-equivalence
 if the induced map $f_*:H_*(X,R) \rightarrow H_*(Y,R)$ is an isomorphism.
 Then one gets an $R$-localization $\infty$-functor $L_R:\sS \rightarrow \sS$ whose essential image
 is spanned by $R$-local spaces, and which can be identified with $\iS[W_R^{-1}]$.
 As an example, $L_\ZZ(X)=X^+$ can be described by the Quillen $+$-construction (see \cite{HoyoisQ})
 and $L_\QQ(X)$ --- for a simply connected space --- can be described by either the Sullivan or the Quillen
 model (see \cite{Ivanov}).
\end{ex}

\begin{ex}\textbf{$\infty$-topos.}\label{num:infty-topos}
In fact, the main motivation of \cite{LurieHTT} was to introduce
 the $\infty$-categorical analogue of the theory developed in SGA4.

Abstractly, an $\infty$-topos is a particularly nice type of presentable $\infty$-category.
 Explicitly, an $\infty$-category $\iC$ is an \emph{$\infty$-topos} if
 there exists a small $\infty$-category $\iC_0$
 such that $\iC$ is an accessible left localization of the $\infty$-category $\iPSh(\iC_0)$
 and in addition, the canonical functor: $\pi:\iPSh(\iC_0) \rightarrow \iC$ commutes with \emph{finite} limits
 (see again \cite[15.9]{CisB}).

Let us be more concrete and consider a Grothendieck site $\mathrm S$,
 with a topology $t$ generated by a pre-topology:
 an ordinary category of geometric objects $X$, with a collection
 of covers $(p_i:V_i \rightarrow X)_{i \in I}$ satisfying the axioms 
 of \cite[II, Def. 1.3]{SGA4}.
 Then one can consider the localization of the $\infty$-category $\iPSh(\Nrv \mathrm S)$
 with respect to augmented \v Cech resolution associated with a cover $(p_i)_{i \in I}$ as above:
$$
\xymatrix@=30pt{
\hdots &
\bigsqcup_{(i,j) \in I^2} W_i \ar@<4pt>[r]\ar@<-4pt>[r]&
\bigsqcup_{i \in I} W_i\ar[l]\ar^-p[r] & W
}
$$
More precisely, this diagram defines by the Yoneda embedding a morphism in $\iPSh(\Nrv \mathrm S)$:
$$
\check C_\bullet(W/X) \xrightarrow p W
$$
 --- the source is indeed a simplicial presheaf --- and one localizes
 the $\infty$-category $\iPSh(\Nrv \mathrm S)$ with respect to these latter morphisms
 to get the $\infty$-category of sheaves $\iSh_t(\mathrm S)$.
 One formally deduces that the canonical functor
 $a_t=\pi:\iPSh(\mathrm S) \rightarrow \iSh_t(\mathrm S)$ is accessible
 and a left localization (as stated in \Cref{num:presentabl-ppties}).
 Using the properties of the pre-topology $t$, one deduces further that $a_t$
 also commutes with finite limits, as required.
 \end{ex}

\begin{ex}\label{ex:presentable-derived}
Another important example for us comes from abelian categories.
 Assume that $A$ is a Grothendieck abelian category.
 Then the associated derived $\infty$-category $\iDer(A)$ is presentable.
 This is not obvious given the description of \Cref{ex:loc-icat1},
 but see \cite[1.3.5.21]{LurieHA}.
\end{ex}

\begin{num}\textbf{The link with model categories}.\label{num:model}
 For a long time, model categories have been the most efficient way to describe
 both derived categories and homotopy categories from algebraic topology.
 The efficiency of $\infty$-categories comes from the language that their axiomatic
 allows one to develop, as exemplified previously. On the other hand, the link
 with model categories is very tight.

Let $\iM$ be a (closed) model category (\cite{QuiMod}), equipped with its three sets of morphisms:
 weak equivalences $W$, cofibrations $Cof$ and fibrations $Fib$.
 Then one associates to $\iM$ the $\infty$-category $\iM_\infty=\iM[W^{-1}]$ obtained by inverting
 its weak equivalences, as in \ref{ex:loc-icat1}.
 It can be shown that this association defines a Quillen equivalence between
 the so-called \emph{combinatorial} model categories and the presentable $\infty$-categories
 (see \cite{Pavlov}).
\end{num}

\subsection{Stable $\infty$-categories, triangulated categories and $t$-structures}

\begin{num}\label{num:infty-additive}
The theory of stable $\infty$-categories
 is a convenient replacement for that of triangulated categories.
 One of the appealing features is that being stable is a property,
 whereas being triangulated is a structure.

To facilitate the definition, let us introduce further notation
 in a given $\infty$-category $\iC$.
 A \emph{zero object} is an object which is both initial and final
 (\Cref{ex:final-initial}). Such an object, unique up to a contractible space of choices,
 is conventionally denoted by $0$. Note that a zero object allows us to define
 zero maps $A \rightarrow 0 \rightarrow B$ between arbitrary two objects,
 also denoted by $0$ following the usual abuse of notation.

One says that an $\infty$-category $\iC$ is \emph{additive} if its homotopy category is additive.
 This amounts to asking that $\iC$ admits finite producs and coproducts, a zero object,
 and that the for all objects $M$ and $N$, the canonical map
$$M \sqcup N \xrightarrow{\begin{pmatrix}\Id_M & 0 \\ 0 & \Id_N\end{pmatrix}} M \times N$$
 is an isomorphism.\footnote{These conditions can be checked indifferently in $\iC$
 or in $\Ho(\iC)$. See \Cref{rem:colim-infty-classical}.}

A \emph{commutative square} in an arbitrary $\infty$-category $\iC$
 is a functor $\Delta:\Box \rightarrow  \iC$
 where $\Box$ is the nerve of the obvious category. One writes suggestively
$$
\xymatrix@=10pt{
M\ar^f[r]\ar_h[d]\ar@{}|\Delta[rd] &N\ar^k[d] \\
P\ar_g[r] & Q.
}
$$ 
Assuming that $\iC$ admits finite limits and colimits,
 we say that the square $\Delta$ is \emph{cartesian} (resp. \emph{cocartesian})
 if the canonical morphism $M \rightarrow N \times_Q P$
 (resp. $N \sqcup_M P \rightarrow Q$) to the obvious pullback (resp. pushout),
 as defined in \Cref{eq:pull-push}, is an isomorphism.

The following definition of stability is striking in its simplicity
 (see \cite[Def. 1.1.1.9, Prop. 1.1.3.4]{LurieHA}).
\end{num}
\begin{df}
We say that an $\infty$-category $\iC$ is \emph{stable} if it admits finite limits and colimits,
 a zero object, and if any commutative square $\Delta$ of $\iC$ is cartesian if and only if it is cocartesian.
\end{df}

\begin{num}\label{num:stable}
In order to state the fundamental property of a stable $\infty$-category $\iC$,
 we introduce some terminology.

Let $M$ be an object of $\iC$. One defines the suspension (resp. loop) object $\Sigma M$ (resp. $\Omega M$)
associated with $M$ by the following cocartesian (resp. cartesian square):
$$
\xymatrix@=10pt{
0\ar[r]\ar[d] & M\ar[d] & \text{resp.} & \Omega M\ar[r]\ar[d] & M\ar[d] \\
M\ar[r] & \Sigma M && M\ar[r] & 0.
}
$$
Consider a sequence of composable maps in the $\infty$-category $\iC$:
\begin{equation}\tag{$\sigma$}
M \xrightarrow f N \xrightarrow g P
\end{equation}
that fits into a commutative square of the form: 
$$
\xymatrix@=10pt{
M\ar^f[r]\ar[d]\ar@{}|\Delta[rd] & N\ar^g[d] \\
0\ar[r] & P.
}
$$
Note that the existence of the commutative square $\Delta$
is equivalent to saying that $g \circ f = 0$ in the homotopy category.\footnote{Moreover,
specifying the commutative square $\Delta$
is equivalent to giving a composite $h$ of $g$ and $f$,
and a homotopy between $h$ and the zero map $0: M \rightarrow P$.}
One says that the sequence $(\sigma)$ is a \emph{homotopy exact sequence} if in addition,
the square $\Delta$ is cartesian --- or equivalently cocartesian.
One can observe at this point that in this case,
the square $\Delta$ is unique up to a contractible space of choices.
One also says that $P$ (resp. $M$) is the cone or cofiber (resp. fiber) of $f$ (resp. $g$).

According to the definition of the suspension object, 
there exists a unique map $\delta: P \rightarrow \Sigma M$
which fits into the following commutative diagram:
$$
\xymatrix@=10pt{
M\ar^f[rr]\ar[dd]\ar|/-2pt/{\Id_M}[rd] && N\ar|/-6pt/g[dd]\ar[rd] \\
& M\ar[rr]\ar[dd] && 0\ar[dd] \\
0\ar[rr]\ar[rd] && P\ar|/-4pt/{\delta}[rd] \\
& 0\ar[rr] && \Sigma M
}
$$
where the rear and front squares are cartesian.
It is called the \emph{boundary operator} associated with the exact sequence $(\sigma)$.

We now arrive at the main, fundamental theorem of the theory of stable $\infty$-categories,
due to Lurie (see \cite[1.1.2.14]{LurieHA}). Its proof is primarily (though not exclusively) based 
on the universal property of pullbacks in $\infty$-categories. It is particularly
satisfying that all four axioms of Verdier's triangulated categories,
including the choices of signs and the octahedral axiom,
are direct consequences of the simple definition of stable $\infty$-categories.
\end{num}
\begin{thm}\label{thm:stable->triangulated}
Let $\iC$ be a stable $\infty$-category.
 Then it is additive and its homotopy category $\Ho(\iC)$ has a unique triangulated structure
 such that the suspension functor is induced by $\Sigma$ and whose distinguished triangles
 are given by the image of the exact sequences with their boundary operator.
 \end{thm}
Note moreover that the suspension and loop operations
 on a stable $\infty$-category $\iC$ define an adjoint pair of auto-equivalences $(\Sigma,\Omega)$
 of $\iC$. In other words, the loop object functor corresponds to the desuspension.

\begin{rem}\label{rem:exactness-map}
Let $\iC$ be a stable $\infty$-category.
 The mapping space $\Map_{\iC}(M,N)$ is then an infinite loop-space,
 as demonstrated by the isomorphism
$$
\Map_{\iC}(M,N)=\Map_{\iC}(\Omega M,\Omega N)=\Omega\Map_{\iC}(\Omega M,N).
$$
In particular, $\pi_0\Map_{\iC}(M,N)$ is an abelian group,
 corresponding to the additivity of the category $\Ho(\iC)$.
 In fact, for stable $\infty$-category, one has a rectification procedure
 analogous to \Cref{rem:simplicial-cat} where one can replace a stable $\infty$-category with
 a category enriched over spectra in the classical sense of algebraic topology
 (see \cite[Rem. 4.8.2.20]{LurieHA}).
 In other words, the reader can freely assume that $\Map_{\iC}(M,N)$ is a
 spectrum.\footnote{We will make it precise when we consider a mapping space
 in a stable $\infty$-category as a spectrum.}

Given any exact sequence $M \rightarrow N \rightarrow P$ of $\iC$,
 and any object $Q$, one deduces fibration sequences of mapping spaces
 (by their exactness properties):
\begin{align*}
\Map_{\iC}(P,Q) &\rightarrow \Map_{\iC}(N,Q) \rightarrow \Map_{\iC}(M,Q) \\
\Map_{\iC}(Q,M) &\rightarrow \Map_{\iC}(Q,N) \rightarrow \Map_{\iC}(Q,P) 
\end{align*}
Applying the functor $\pi_0$, one recovers the long exact sequence corresponding to
 the triangulated structure of $\Ho(\iC)$.
\end{rem}

\begin{ex}\label{ex:GAb-stable-infty}
The derived $\infty$-category $\iDer(\catA)$ of a Grothendieck abelian category $\catA$ is stable
 (\cite[1.3.5.9]{LurieHTT}).
 Moreover, the triangulated structure on $\Ho(\iDer(\catA))$ coincides
 with the Verdier triangulated structure of $\Der(\catA)$ through the identification
 of \Cref{ex:loc-icat1}(2).\footnote{This follows from the description of $\iDer(\catA)$
 as the localization of the dg-nerve of the dg-category of chain complexes on $\catA$,
 \cite[1.3.5.13]{LurieHTT}.}
\end{ex}

\begin{num}\textbf{Exact functors}.
Let $F:\iC \rightarrow \iD$ be a functor between $\infty$-categories.
 Then the following conditions are equivalent (\cite[1.1.4.1]{LurieHTT}):
\begin{enumerate}
\item $F$ commutes with finite limits.
\item $F$ commutes with finite colimits.
\item $F$ respects exact sequences.
\end{enumerate}
If these conditions hold, one says that $F$ is exact.
\end{num}

\begin{ex}
Let $\iC$ be a stable and presentable $\infty$-category.

Let $W$ be a set of morphisms of $\iC$, and $\iC[W^{-1}]$ be the localization of $\iC$ at $W$.
 We have seen in \Cref{num:presentabl-ppties}, point (3), that this is automatically a left Bousfield
 localization so that we have an adjunction of $\infty$-categories:
$$
\pi:\iC \rightarrow \iC[W^{-1}]:\nu
$$
such that $\nu$ is fully faithful. Given the definition of $W$-local objects,
 and \Cref{rem:exactness-map}, one easily deduces that $W$-local objects are stable under
 extensions and suspensions. As $\nu$ is fully faithful, one deduces that
 $\iC[W^{-1}]$ is a stable $\infty$-category, and that both $\pi$ and $\nu$
 are exact functors.
 In particular, the $W$-localization functor $L_W$ is exact. 

We retain from these discussions that the localization of a presentable
 and stable $\infty$-category with respect to any set of morphisms is
 again a presentable and stable $\infty$-category, and is equivalent to the full
 sub-$\infty$-category spanned by the $W$-local objects.
\end{ex}

Thus the $\infty$-categories which are both presentable and stable enjoy very good properties.
\begin{df}\label{df:icat-stable}
We let $\iCatS$ be the sub-$\infty$-category of the $\infty$-category $\iCatP$
 spanned by those presentable $\infty$-categories that are both presentable and stable.
\end{df}
As morphisms of $\iCatP$, the so-called left functors, are required to commute with arbitrary limits,
 morphisms of presentable stable $\infty$-categories are in particular exact. 

\begin{df}\label{df:infty-t-struct}
A \emph{$t$-structure} on a stable $\infty$-category $\iC$ is the data of
 a pair of full sub-$\infty$-categories $(\iC_{\geq 0},\iC_{<0})$
 such that the pair $(\Ho\iC_{\geq 0},\Ho \iC_{<0})$ defines a $t$-structure
 on $\Ho(\iC)$.
\end{df}
In other words, a $t$-structure on an $\infty$-category is nothing else
 than a $t$-structure on its homotopy category. One sometimes calls
  a stable $\infty$-category equipped with a $t$-structure a $t$-$\infty$-category.

\begin{ex}\label{ex:t-structures-icat}
 These examples will be used in the text.
\begin{enumerate}
\item Given a Grothendieck abelian category $\catA$,
 one obviously gets a canonical $t$-structure on $\iDer(\catA)$, corresponding
 to the canonical $t$-structure on $\Der(\catA)$.
\item Let $\iC$ be a presentable stable $\infty$-category.
 Given a set of objects $P$ of $\iC$, we let $\langle P \rangle_+$ be the full sub-$\infty$-category of $\iC$
 that contains $P$ and is stable under extensions, positive suspensions and coproducts.
 Then there exists a unique $t$-structure $(\iC^P_{\geq 0},\iC^P_{<0})$ on $\iC$
 whose homologically positive objects are exactly $\iC^P_{\geq 0}=\langle P \rangle_+$.

 This is a classical construction,
 for example when the triangulated category $\Ho(\iC)$ is compactly generated
 (see \cite[Th. 1.2.6]{BD1}). The main point here
 is that one gets the homologically non-negative functor $\tau_{\geq 0}^P$
 as the right adjoint of the canonical functor $\nu_+:\langle P \rangle_+ \rightarrow \iC$,
 simply by applying \Cref{num:presentabl-ppties}(2).
\end{enumerate}
\end{ex}

\subsection{Monoidal $\infty$-categories}

\begin{num}
The definition of a symmetric monoidal category is much more involved
 that its counterpart for ordinary categories. This can be explained as a lot
 of the necessary structure is given by isomorphisms, which are to be described
 coherently with the higher structure of an $\infty$-category.
 However, such a description has been worked out in topology, via the general theory
 of \emph{operads} and more specifically of Segal's $\Gamma$-spaces.
 It leads to the more general notion of \emph{commutative algebra},\footnote{Note that a more precise terminology would be $E_\infty$-object.}
 in a general $\infty$-category with finite product $\iC$.

We first introduce the category $\Fin$ --- which is actually the opposite of the category $\Gamma$
 defined by Segal --- described as follows:
\begin{itemize}
\item objects are defined to be the sets $n_*:=\{0,...,n\}$ pointed by the element $0$, for an integer $n>0$,
\item morphisms are the pointed maps.
\end{itemize}
Important examples are defined by the so-called \emph{inert maps}, indexed by integers $1 \leq i \leq n$:
$$
\alpha_n^j:n_* \rightarrow 1_*, i \mapsto \delta_n^i.
$$
\end{num}

\begin{df}
A \emph{commutative algebra} in $\iC$ is an $\infty$-functor
$$
M^\otimes:\Nrv \Fin \rightarrow \iC
$$
such that for any integer $n>0$, the morphism
$$
M^\otimes(n_*) \xrightarrow{\tau_n=\prod_{1 \leq j \leq n} (\alpha_n^j)_*} \big(M^\otimes(1_*)\big)^n
$$
where the right-hand side denotes the $n$-fold product in $\iC$, is an isomorphism.

One defines the $\infty$-category of commutative algebras in $\iC$ as the
 sub-$\infty$-category of $\Fun(\Nrv \Fin,\iC)$ spanned by the commutative algebras.
\end{df}

\begin{num}\label{num:abusive-comm}
We can abusively identify such a commutative algebra with the object $M=M^{\otimes}(1_*)$ equipped
 with the multiplication maps
$$
\mu_n:M^n=\big(M^\otimes(1_*)\big)^n \xrightarrow{\tau_n^{-1}} M^\otimes(n_*) \xrightarrow{c_{n*}} M^\otimes(1_*)=M
$$
where $c_n:n_* \rightarrow 1_*$ is the pointed application which sends $1 \leq i\leq n$ to $1$.
 One can deduce from the universal property of products symmetry and associativity isomorphisms as expected.
\end{num}

\begin{ex}
\begin{enumerate}
\item Let $\mathrm{Set}$ be the category of sets.
 Then a commutative algebra in $\Nrv \mathrm{Set}$ is (equivalent to the data of) a commutative monoid.
 Similarly, a commutative algebra in the nerve of the category of abelian groups is a commutative ring.
\item Let $\mathscr Sp$ be the $\infty$-category of spectra from algebraic topology.\footnote{It can
 be obtained from the usual model category of spectra using \Cref{num:model}.
 Or more directly from the $\infty$-category $\iS$ of spaces by formal
 inversion of the suspension functor: this is known as the stabilization
 of $\iS$ (see \cite[\textsection 1.4.3]{LurieHA}).}
 Then a commutative algebra object in $\mathscr Sp$ is what is usually called an \emph{$E_\infty$-spectrum}.
\end{enumerate}
\end{ex}

We now arrive at the central definition of this subsection. We have taken the point of view of Joyal,
 to use $\infty$-functors rather than cofibered $\infty$-categories. Both point of views are equivalent:
 see \cite{GrothI}, Definition 4.4 and Remark 4.5.
\begin{df}
A \emph{symmetric monoidal} (resp. \emph{presentable and stable symmetric monoidal}) \emph{$\infty$-category} is a commutative
 algebra in the $\infty$-category $\iCat$ (resp. $\iCatS$).

We let $\iCatM$ (resp. $\iCatMS$) be the $\infty$-category consisting of these particular commutative algebras,
 as defined previously. A $1$-morphism in $\iCatM$ (resp. $\iCatMS$) will be called a monoidal left functor.
\end{df}

\begin{num}\label{num:explicit-prstmon-icat}\label{num:ifty-explicit-otimes}
Let us make explicit the above notion, in the case of a presentable and stable symmetric monoidal $\infty$-category $\iC^\otimes$.
 If we use the abusive description of \Cref{num:abusive-comm}, it corresponds to a 
 presentable and stable $\infty$-category $\iC=\iC^\otimes(1_*)$ equipped with $\infty$-functors,
 the $n$-fold multiplication maps:
$$
\mu_n:\iC^n \rightarrow \iC
$$
where the left-hand side is the $n$-fold product of presentable and stable $\infty$-categories.
 These multiplication maps are required to satisfy suitable commutativity and associativity axioms.

In particular, we can write $\otimes_{\iC}=\mu_2$ and call it the tensor product
 associated with $\iC^\otimes$. For any object $X$ in $\iC$, one can consider the
 $\infty$-functor:
$$
\Sigma_X:\iC \rightarrow \iC, Y \rightarrow X \otimes Y=\mu_2(X,Y).
$$
According to our choice of morphisms in $\iCatS$ (\Cref{df:icat-stable}), this is a left functor (\Cref{df:presentable})
 so that it automatically admits a right adjoint $\Omega_X=\uHom_{\iC}(X,-)$ --- according to \Cref{num:presentabl-ppties}(2).
 One deduces a bifunctor:
$$
\uHom:\iC \times \iC \rightarrow \iC
$$
which is exact.

As the associated homotopy category functor $\Ho$ commutes with products,
 one deduces that the associated homotopy category $\Ho(\iC)$ is symmetric monoidal, and moreover closed.
 Besides, starting from \Cref{thm:stable->triangulated}, $\Ho(\iC)$ is even a triangulated monoidal category.
\end{num}

\begin{ex}\label{ex:monoidal-icat}
An $\infty$-category $\iC_0$ which admits finite products
 can be endowed with a symmetric monoidal $\infty$-category structure,
 with associated tensor product given by the cartesian product $\times$.
 See \cite[\textsection 2.4.1]{LurieHA}.

Using the so-called \emph{Day convolution product}, Lurie further proved
 (see \cite[after Def. 2.1]{NS}, or directly \cite[Cor. 4.8.1.12]{LurieHA})
 that the $\infty$-category of presheaves $\iPSh(\iC_0)$ admits a symmetric monoidal structure
 such that the Yoneda embedding
$$
\gamma:\iC_0 \rightarrow \iPSh(\iC_0)
$$
is monoidal. Beware that this monoidal structure is in general different from the one coming from
 the cartesian product on $\iPSh(\iC_0)$, except when the $\infty$-category $\iC_0$ admits finite products.
\end{ex}

\begin{num}\textbf{Tensor invertible objects}.\label{num:otimes-inversion}
Let $\iC$ a presentable symmetric monoidal $\infty$-category and $X$ be an arbitrary object.
 One says that $X$ is $\otimes$-invertible if the $\infty$-functor $\Sigma_X$, defined in \Cref{num:ifty-explicit-otimes},
 is an equivalence of $\infty$-categories.

Marco Robalo has described in \cite[\textsection 2.1]{Robalo} 
 a universal procedure to $\otimes$-invert the object $X$.
 Indeed, he shows that the sub-$\infty$-category of the comma $\infty$-category $\iC/\iCatM$ (see \Cref{ex:comma})
 spanned by the monoidal $\infty$-functors $F:\iC \rightarrow \iD$ such that $F(X)$ is $\otimes$-invertible
 admits an initial object --- combine Proposition 2.1 and Proposition 2.9(1).
 We denote such an initial object (well-defined up to a contractible set of choices) by
 $\Sigma_X^\infty:\iC \rightarrow \iC[X^{-1}]$.

If, in addition, the object $X$ is symmetric (see \cite[Def. 2.16]{Robalo}),
 then the monoidal $\infty$-category $\iC[X^{-1}]$ can be described as the \emph{spectrum objects
 relative to $X$}; see \emph{loc. cit.} Corollary 2.22.
 In other words, $\iC[X^{-1}]$ is the homotopy colimit in the $\infty$-category $\iCatM$ 
  of the following tower of presentable monoidal $\infty$-categories:
$$
\iC \xrightarrow{\Sigma_X} \iC \xrightarrow{\Sigma_X} \hdots
$$
where $\Sigma_X$ was defined in \Cref{num:ifty-explicit-otimes}.
 In this construction, the assumption that $X$ is symmetric is used via \emph{loc. cit.}
 Theorem 2.14.

This implies in particular that whenever $\iC$ is stable and $X$ is symmetric,
 the $\infty$-category $\iC[X^{-1}]$ is stable and $\Sigma_X^\infty$ is exact.
\end{num}

\begin{ex}\label{ex:stabilization}
Following \cite[\textsection 1.4.2]{LurieHA},
 one can define the stabilization $\iC_{st}$ of a pointed $\infty$-category $\iC$ with finite limits
 by considering the limit of the (left) tower of pointed $\infty$-categories
$$
 \hdots  \xrightarrow{\Omega} \iC \xrightarrow{\Omega} \iC
$$
where $\Omega$ is the loop space functor defined in \Cref{num:stable}.
 If $\iC$ is a symmetric monoidal $\infty$-category, one obtains
 a symmetric monoidal $\infty$-category structure on $\iC_{st}$ by taking
 the colimit in the $\infty$-category $\iCatM$.

Building on \Cref{ex:monoidal-icat}, $\iC$ admits a symmetric monoidal $\infty$-categorical structure
 whose tensor product is given by the smash product. One can view the simplicial sphere $S^1$
 as an object of $\iC$ by the formula $S^1=\Sigma *$. Then one obtains the identification
 $\Omega=\Omega_{S^1}$ and $\Sigma=\Sigma_{S^1}$ (by using adjunctions properties).
 Note that $S^1$ is a symmetric object.\footnote{This follows from the corresponding property
 of the simplicial set $S^1$ in the $\infty$-category of spaces. Or one can directly deduce this
 from the symmetric monoidal structure on $\iC$.}
 
One can show that the stabilization $\iC_{st}$ defined above coincides
 with the monoidal $\infty$-category $\iC[(S^1)^{-1}]$ constructed previously.
 In other words, $\iC_{st}$ can be identified with the colimit of the (right) tower
 of presentable monoidal $\infty$-categories
$$
\iC \xrightarrow{\Sigma} \iC \xrightarrow{\Sigma} \hdots
$$
computed in the $\infty$-category of presentable monoidal $\infty$-categories.
 This follows from \cite[Rem. C.1.1.6]{LurieSAG}. Note that if one computes
 the previous colimit in the $\infty$-category of $\infty$-categories, 
 then one obtains a smaller $\infty$-category called the Spanier-Whitehead category.
 It is not stable but only \emph{pre-stable} (see \emph{loc. cit.}).
\end{ex}

\begin{num}\textbf{Monoidal structures and localizations.}\label{df:ifty-monoidal-loc}
Let $\iC$ be as in \Cref{num:explicit-prstmon-icat}.
 We consider a set of morphisms $W$ of $\iC$ and let us consider the canonical left functor $\pi:\iC \rightarrow \iC[W^{-1}]$.

Then the following conditions are equivalent (see \cite[Prop. 2.2.1.9]{LurieHA}):
\begin{enumerate}
\item There exists a symmetric monoidal $\infty$-category structure on $\iC[W^{-1}]$ such that the functor $\pi$
 extends to a symmetric monoidal left functor.
\item The $W$-local equivalences in $\iC$ are stable under tensor product:
 for any $W$-local equivalence $f:M \rightarrow N$, and any object $P$ of $\iC$, $f \otimes P$ is a $W$-local equivalence.
\end{enumerate}
In that case, we say that the $W$-localization is monoidal. By abuse of notation, we say that the functor $\pi$ is symmetric monoidal.
\end{num}

\begin{ex}
Let $\iSh_t(\mathrm S)$ be the $\infty$-topos of $t$-sheaves on a Grothendieck site as in \Cref{num:infty-topos}.
 One deduces from the description of \v Cech covers that $\iSh_t(\mathrm S)$ admits a symmetric monoidal structure
 such that $\infty$-functor $a:\iPSh(\mathrm S) \rightarrow \iSh(\mathrm S)$ is symmetric monoidal.
 Moreover, when the site $\mathrm S$ admits products, the corresponding tensor product on $\iSh(\mathrm S)$ is the cartesian product.
\end{ex}

\begin{num}\textbf{Monoidal model categories}.\label{num:model-tensor}
It is established in \cite[Prop. 2.3]{NS} that the $\infty$-category $\iM[W^{-1}]$
 associated to any simplicial, combinatorial, tractable, left proper and symmetric monoidal
 model category $\iM$ with weak equivalences $W$ admits a canonical symmetric monoidal $\infty$-categorical structure.

Moreover, it is proved in \cite[Prop. 2.4, Th. 2.8]{NS} that every presentable symmetric monoidal
 $\infty$-category $\iC$ arise in that way: there exists an equivalence $\iC \simeq \iM[W^{-1}]$.
 
In other words, one can freely use the construction of model categories to benefit
 from the language of $\infty$-category theory.
\end{num}

\subsection{Pro-objects in $\infty$-categories}\label{sec:pro-infty}

\begin{num}
Our reference for pro-objects is
 \cite[\textsection 5.2]{BHH}.\footnote{Which also proposes a detailed treatment 
 of set-theoretic issues inherent to the notions of limits and colimits.}

Given a presentable $\infty$-category $\iC$, there exists an $\infty$-category
 $\pro\iC$ which admits limits (see \Cref{num:(co)lim-ifty}),
 an $\infty$-functor $\iota:\iC \rightarrow \pro\iC$ which satisfy
 the following universal property: for any $\infty$-category $\iD$
 which admits all limits, the $\infty$-functor
$$
\iota^*:\Fun(\pro\iC,\iD) \rightarrow \Fun(\iC,\iD)
$$
is fully faithful and induces an equivalence of $\infty$-categories
 with the full sub-$\infty$-category of the left-hand side spanned
 by those functors which commute with limits
 (see again \Cref{num:(co)lim-ifty}).

We refer the reader to \cite[Th. 3.2.19]{BHH}.
 A simple construction of the $\infty$-category $\pro\iC$ is
 to consider the full sub-$\infty$-category of $\Fun(\iC,\iS)^{\op}$
 of functors which are accessible and commute
 with finite limits (see \cite[Prop. 3.2.18]{BHH}).
\end{num}

\begin{rem}\label{rem:ind-pro}
In fact, as we are implicitly working within universes,
 one can also rely on the formula:
$$
\pro\iC=\big(\ind(\iC^{\op})\big)^\op
$$
and use the treatment of ind-objects in \cite[\textsection 5.3]{LurieHTT}.
\end{rem}

\begin{ex}\label{ex:pro-ifty&ordinary}
Let $\catC$ be an (ordinary) category.
 Then one obtains the following identification:
$$
\pro(\Nrv \catC)=\Nrv\big(\pro\catC)
$$
using, on the right-hand side, the classical definition of pro-objects.
 Indeed, this follows from the universal properties of pro-objects
 in both the $\infty$-categorical framework (as above)
 and the ordinary categorical one (see \cite{SGA4}, \textsection 8.10, 8.13).
\end{ex}

\begin{num}\label{num:notation-plimit}
By the very construction, the $\infty$-category $\pro\iC$ associated
 with a presentable $\infty$-category admits cofiltered projective limits.
 In particular, given an $\infty$-functor $X:\sI \rightarrow \iC, i \mapsto X_i$,
 one can consider the associated projective limit in $\pro\iC$ which we will denote
 by 
$$
\plim{i \in \sI} X_i
$$
 following the classical notation (see \cite[I, 8.5.3.2]{SGA4}).

According to \cite[Prop. 5.1.1, 5.1.2]{BHH}, given two
 $\infty$-functors $X:\sI \rightarrow \iC$ and $Y:\sJ \rightarrow \iC$,
 there exists a canonical equivalence:
\begin{equation}\label{eq:Map-pro}
\Map(\plim{i \in \sI} X_i,\plim{j \in \sJ} Y_j)
 \simeq \underset{j \in \sJ}\lim {i \in \sI}\colim \Map(X_i,Y_j).
\end{equation}
One should be careful that, in general, $$\Ho(\pro\iC) \neq \pro\Ho \iC.$$
 But see \Cref{rem:compare-old&new_DMgen} for a case of interest.
\end{num} 
\section{Categorical presentation of cycle modules}\label{sec:cmod}

One can describe Rost's cycle premodules
 as covariant functors from a category whose morphisms are defined by generators and relations.
 We reproduce below the definition from \cite[\textsection 4.1]{Deg5} for the reader's convenience.
\begin{df}\label{df:efld}
We let $\efld_k$ be the category
 whose objects are given by pairs $(E,n)$ where $E/k$ is a function field,
 and $n \in \ZZ$ an integer.
 The morphisms of $\efld_k$ are abelian groups defined by the following generators
 $(\text D*)$ and relations $(\text R*)$ as displayed below, where maps are always assumed to be
 $k$-algebra morphisms,
 and the valuations are always understood as (geometric) valuations of function fields:

\noindent \underline{Generators}:
\begin{itemize}
\item[\textbf{D1:}] $\varphi_*:(E,n) \rightarrow (L,n)$ for
 $\varphi:E \rightarrow L$, $n \in \ZZ$.
\item[\textbf{D2:}] $\varphi^!:(L,n) \rightarrow (E,n)$, for $\varphi:E
 \rightarrow L$ finite, $n \in \ZZ$.
\item[\textbf{D3:}] $\gamma_x:(E,n) \rightarrow (E,n+r)$, for
  $x \in K_r^M(E)$, $n \in \ZZ$.
\item[\textbf{D4:}] $\partial_v:(E,n) \rightarrow (\kappa(v),n-1)$,
 for $(E,v)$ valued function field over $k$, $n \in \ZZ$.
\end{itemize}

\noindent \underline{Relations}:
\begin{itemize}
\item[\textbf{R0:}] For all $x,y \in K_*^M(E)$,
$\gamma_x \circ \gamma_y=\gamma_{x.y}$.
\item[\textbf{R1a:}] $(\psi \circ \varphi)_*=\psi_* \circ \varphi_*$.
\item[\textbf{R1b:}] $(\psi \circ \varphi)^!=\varphi^! \circ \psi^!$.
\item[\textbf{R1c:}] Let $\varphi:K \rightarrow E$,
$\psi:K \rightarrow L$ be finite. For any $z \in \spec{E \otimes_K L}$,
let $\bar \varphi_z:L \rightarrow E \otimes_K L/z$ and
$\bar \psi_z:E \rightarrow E \otimes_K L/z$ be the induced morphisms:

$\psi_*\varphi^!=\sum_{z \in \spec{E \otimes_K L}}
\mathrm{lg}\big(E \otimes_K L_z\big).(\bar \varphi_z)^!(\bar \psi_z)_*$,

where $\mathrm{lg}(A)$ is the length of an Artin ring $A$.
\item[\textbf{R2a:}] For all $\varphi:E \rightarrow L$, $x \in K_*^M(E)$,
$\varphi_* \circ \gamma_x=\gamma_{\varphi_*(x)} \circ
\varphi_*$.
\item[\textbf{R2b:}] For $\varphi:E \rightarrow L$ finite and all $x \in K_*^M(E)$,
$\varphi^! \circ \gamma_{\varphi_*(x)}=
\gamma_x \circ \varphi^!$.
\item[\textbf{R2c:}] For $\varphi:E \rightarrow L$ fini and all $y \in K_*^M(L)$,
$\varphi^! \circ \gamma_y \circ \varphi_*=
\gamma_{\varphi^!(y)}$.
\item[\textbf{R3a:}] Let $\varphi:E \rightarrow L$ be a morphism,
 $v$ and $w$ be valuations $L/k$ and $E/k$ respectively,
 and $e>0$ an integer such that $v|_{E^\times}=e.w$. Let
 $\bar\varphi:\kappa(w) \rightarrow \kappa(v)$ be the induced morphism:

$\partial_v \circ \varphi_* =e.{\bar \varphi}_* \circ \partial_w$.
\item[\textbf{R3b:}] Let $\varphi:E \rightarrow L$ be finite, and $v$ a valuation on $E/k$.
 Given a valuation $w$ on $L$ which extends $v$, we let $\bar \varphi_w:\kappa(w) \rightarrow \kappa(v)$
 be the induced morphism on the respective residue fields:

$\partial_v \circ \varphi^!=\sum_{w/v} \bar \varphi_w^! \circ \partial_w$.
\item[\textbf{R3c:}] Let $\varphi:E \rightarrow L$, $v$ a valuation on $L/k$
 which is zero on $E^\times$: $\partial_v \circ \varphi_*=0$.
\item[\textbf{R3d:}] Let $(E,v)$ be a valued function field over $k$ and $\pi$ uniformizer of $v$:
 $\partial_v \circ \gamma_{\{-\pi\}} \circ \varphi_* ={\bar \varphi}_*$.
\item[\textbf{R3e:}] Let $(E,v)$ be a valued function field over $k$, $u \in E^\times$ such that $v(u)=0$:
$\partial_v \circ \gamma_{\{u\}}
=-\gamma_{\{\bar u\}} \circ \partial_v$.
\end{itemize}
\end{df}
According to \cite[Def. 1.1]{Rost}, a cycle premodule over $k$ is simply a (covariant) functor
 $M:\efld \rightarrow \ab$.

\begin{num}\textit{Higher order valuations}.--
The preceding category admits a more elegant description due to Rost.
 We introduce some notation to state it.

Let $F$ be field with an arbitrary valuation (not necessarily discrete)
$$
v:F^\times \rightarrow \Gamma
$$
where $\Gamma$ is a totally ordered abelian group. 
 We refer the reader to \cite{Vaquie} about arbitrary valuation rings.

As usual, we denote by:
\begin{align*}
\cO_v=\{x \in F \mid v(x)\geq 0\} \\
\cM_v=\{x \in F \mid v(x)>0\}
\end{align*}
the associated valuation ring, and the maximal ideal of the latter,
 where we extend $v$ as usual by setting $v(0)=\infty$,
 for $\Gamma_\infty=\Gamma \cup \{\infty\}$,
 obtained by adding a maximal element $\infty$.
 We also let $\kappa_v=\cO_v/\cM_v$ be the residue (class) field.

Recall that the rank of $v$ is the rank of the abelian group $\mathrm{Im}(v)$
 (\emph{loc. cit.} Definition after Th. 1.7),
 or equivalently the (Krull) dimension of the valuation ring $\cO_v$
 (\emph{loc. cit.} Corollary after Th. 1.7).
 If $v$ has finite rank $r$, then $\Gamma\simeq \ZZ^r$ and $v$ can be defined
 as a \emph{composition} of $r$ valuations $(v_1,\cdots,v_r)$, $v_{i+1}$ being a valuation
 on the residue field $\kappa_{v_i}$ (see \emph{loc. cit.} Remark 1.7).
\end{num}

\begin{num}\textit{Milnor K-theory of valuation rings}.--
We consider an arbitrary valuation $v:F \rightarrow \Gamma$.
 Using the above notation, we have an exact sequence of abelian groups:
\begin{equation}\label{eq:units_val}
1 \rightarrow (1+\cM_v) \rightarrow \cO_v^\times \rightarrow \kappa_v^\times \rightarrow 1\textbf{}
\end{equation}
Following \cite[Rem. 1.6]{Rost}, we define a $\ZZ$-graded ring associated to $v$ by the formula:
$$
\KM*(v)=\KM*(F)/(1+\cM_v).
$$
Considering the obvious projection map $p:\KM*(F) \rightarrow \KM*(v)$,
 there exists a morphism $i:\KM*(\kappa_v) \rightarrow \KM*(v)$ of graded rings uniquely
 defined by the universal property\footnote{This follows from the exact sequence
 \eqref{eq:units_val} by induction on the degree of symbols in Milnor K-theory}:
$$
\xymatrix@=14pt{
\KM*(\cO_v)\ar^p[r]\ar[d] & \KM*(v) \\
\KM*(\kappa_v)\ar@{-->}_i[ru] &  
}
$$
where the vertical map is induced by the obvious surjection.
The following proposition was stated in \cite[Rem. 1.10]{Rost}.
\end{num}
\begin{prop}\label{prop:efld_Rost_present}
Let $k$ be an arbitrary field,
 and $E$, $F$ be function fields over $k$.
 Given an integer $r \geq 0$, we consider the following set:
$$
\cI_r(F,E)=\left\{(v_*,L) \left|
\begin{array}{l}
v_1,\cdots,v_r \textit{ geom. valuations on $F/k$}, v=v_1 \circ \hdots v_r \\
L=\kappa_x, x \in \Spec(E \otimes_k \kappa_v), 
 E \xrightarrow{\varphi} L \xleftarrow{\psi} \kappa_v, \varphi \text{ finite}
\end{array}\right.
 \right\}
$$
Then the following map:
\begin{align*}
\bigoplus_{r\geq 0, (v_*,L) \in \cI_r(F,E)} \KM*(L) \otimes_{\KM*(\kappa_v)} \KM*(v) & \rightarrow \Hom_{\efld_k}\big((F,0),(L,*)\big) \\
\sigma \otimes \overline{\tau} & \mapsto \varphi^! \circ \gamma_\sigma \circ \psi_* \circ \partial_{v_1},
 \circ \hdots \circ \partial_{v_r} \circ \gamma_{\tau}
\end{align*}
where $\tau \in \KM*(\cO_v)$ is an arbitrary element such that $p(\tau)=\bar \tau$,
 is well-defined and induces an isomorphism of graded abelian groups.
\end{prop}
Obviously, given that the functor $(E,0) \mapsto (E,1)$ is an auto-equivalence of the category $\efld_k$,
 this completely describes the morphism in the latter category.
 The proof follows by applying the relations $(\text R*)$ to put morphisms
 in $\efld_k$ in the above normalized form. A key point is Lemma 1.9 of \cite{Rost}.
 
\bibliographystyle{amsalpha}
\bibliography{genmot}

\end{document}